\documentclass[10pt,a4paper,oneside]{article}
\usepackage{a4,graphicx,latexsym,amsmath,amsthm,amssymb,color,verbatim,mathrsfs,bbm}
\usepackage [latin1] {inputenc}
\usepackage{listings, algorithm, algorithmic} 
\usepackage{xspace}
\usepackage{subfigure}
\usepackage[enableskew]{youngtab}
\usepackage{fancyhdr}
\usepackage[bf,small]{caption}
\usepackage{url}
\usepackage {csquotes}
\usepackage{wrapfig}
\usepackage{paralist}
\usepackage{environ}
\usepackage{array}
\usepackage{textcomp}

\usepackage {tikz}
\usetikzlibrary{arrows} 
\usetikzlibrary {decorations,decorations.pathmorphing}

\newdimen\arrowsize
\pgfarrowsdeclare{my}{my}
{
  \arrowsize=0.2pt
  \advance\arrowsize by .5\pgflinewidth
  \pgfarrowsleftextend{-4\arrowsize-.5\pgflinewidth}
  \pgfarrowsrightextend{.5\pgflinewidth}
}
{
  \arrowsize=0.2pt
  \advance\arrowsize by .5\pgflinewidth
  \pgfsetdash{}{0pt} 
  \pgfsetroundjoin   
  \pgfsetroundcap    
  \pgfpathmoveto{\pgfpointorigin}
  \pgfpathlineto{\pgfpoint{-3\arrowsize}{4\arrowsize}}
  \pgfpathlineto{\pgfpoint{-3\arrowsize}{-4\arrowsize}}
  \pgfpathclose
  \pgfusepathqfill
}

\tikzset{rhrhombithickside/.style={ultra thick}}
\tikzset{rhrhombiarrow/.style={}}
\tikzset{rhrhombifinearrow/.style={line width=1pt}}
\tikzset{rhrhombipatharrow/.style={thin}}
\tikzset{rhrhombidraw/.style={ultra thin}}
\tikzset{rhrhombifill/.style={thick,fill=black!40}}
\newcommand{\FIGtriangulararray}{\begin{tikzpicture}
\draw[rhrhombidraw] (0.0pt,0.0pt) -- (-4.619pt,-8.0pt) -- (4.619pt,-8.0pt) -- cycle;
\draw[rhrhombidraw] (-4.619pt,-8.0pt) -- (-9.238pt,-16.0pt) -- (0.0pt,-16.0pt) -- cycle;
\draw[rhrhombidraw] (4.619pt,-8.0pt) -- (0.0pt,-16.0pt) -- (9.238pt,-16.0pt) -- cycle;
\draw[rhrhombidraw] (-9.238pt,-16.0pt) -- (-13.857pt,-24.0pt) -- (-4.619pt,-24.0pt) -- cycle;
\draw[rhrhombidraw] (0.0pt,-16.0pt) -- (-4.619pt,-24.0pt) -- (4.619pt,-24.0pt) -- cycle;
\draw[rhrhombidraw] (9.238pt,-16.0pt) -- (4.619pt,-24.0pt) -- (13.857pt,-24.0pt) -- cycle;
\draw[rhrhombidraw] (-13.857pt,-24.0pt) -- (-18.476pt,-32.0pt) -- (-9.238pt,-32.0pt) -- cycle;
\draw[rhrhombidraw] (-4.619pt,-24.0pt) -- (-9.238pt,-32.0pt) -- (0.0pt,-32.0pt) -- cycle;
\draw[rhrhombidraw] (4.619pt,-24.0pt) -- (0.0pt,-32.0pt) -- (9.238pt,-32.0pt) -- cycle;
\draw[rhrhombidraw] (13.857pt,-24.0pt) -- (9.238pt,-32.0pt) -- (18.476pt,-32.0pt) -- cycle;
\draw[rhrhombidraw] (-18.476pt,-32.0pt) -- (-23.095pt,-40.0pt) -- (-13.857pt,-40.0pt) -- cycle;
\draw[rhrhombidraw] (-9.238pt,-32.0pt) -- (-13.857pt,-40.0pt) -- (-4.619pt,-40.0pt) -- cycle;
\draw[rhrhombidraw] (0.0pt,-32.0pt) -- (-4.619pt,-40.0pt) -- (4.619pt,-40.0pt) -- cycle;
\draw[rhrhombidraw] (9.238pt,-32.0pt) -- (4.619pt,-40.0pt) -- (13.857pt,-40.0pt) -- cycle;
\draw[rhrhombidraw] (18.476pt,-32.0pt) -- (13.857pt,-40.0pt) -- (23.095pt,-40.0pt) -- cycle;
\end{tikzpicture}}
\newcommand{\FIGdualtriangulararray}{\begin{tikzpicture}
\draw[rhrhombidraw] (0.0pt,0.0pt) -- (-9.238pt,-16.0pt) -- (9.238pt,-16.0pt) -- cycle;
\draw (-4.619pt,-8.0pt) -- (0.0pt,-10.656pt);
\draw (4.619pt,-8.0pt) -- (0.0pt,-10.656pt);
\draw (0.0pt,-16.0pt) -- (0.0pt,-10.656pt);
\draw (0.0pt,-16.0pt) -- (0.0pt,-21.344pt);
\draw (-4.619pt,-24.0pt) -- (0.0pt,-21.344pt);
\draw (4.619pt,-24.0pt) -- (0.0pt,-21.344pt);
\fill[color=white] (-4.619pt,-8.0pt) circle (1.2pt);
\draw[rhrhombidraw] (-4.619pt,-8.0pt) circle (1.2pt);
\fill[color=white] (4.619pt,-8.0pt) circle (1.2pt);
\draw[rhrhombidraw] (4.619pt,-8.0pt) circle (1.2pt);
\fill[color=white] (0.0pt,-16.0pt) circle (1.2pt);
\draw[rhrhombidraw] (0.0pt,-16.0pt) circle (1.2pt);
\fill (0.0pt,-10.656pt) circle (1.5pt);
\fill (0.0pt,-21.344pt) circle (1.5pt);
\draw[rhrhombidraw] (-9.238pt,-16.0pt) -- (-18.476pt,-32.0pt) -- (0.0pt,-32.0pt) -- cycle;
\draw (-13.857pt,-24.0pt) -- (-9.238pt,-26.656pt);
\draw (-4.619pt,-24.0pt) -- (-9.238pt,-26.656pt);
\draw (-9.238pt,-32.0pt) -- (-9.238pt,-26.656pt);
\draw (-9.238pt,-32.0pt) -- (-9.238pt,-37.344pt);
\draw (-13.857pt,-40.0pt) -- (-9.238pt,-37.344pt);
\draw (-4.619pt,-40.0pt) -- (-9.238pt,-37.344pt);
\fill[color=white] (-13.857pt,-24.0pt) circle (1.2pt);
\draw[rhrhombidraw] (-13.857pt,-24.0pt) circle (1.2pt);
\fill[color=white] (-4.619pt,-24.0pt) circle (1.2pt);
\draw[rhrhombidraw] (-4.619pt,-24.0pt) circle (1.2pt);
\fill[color=white] (-9.238pt,-32.0pt) circle (1.2pt);
\draw[rhrhombidraw] (-9.238pt,-32.0pt) circle (1.2pt);
\fill (-9.238pt,-26.656pt) circle (1.5pt);
\fill (-9.238pt,-37.344pt) circle (1.5pt);
\draw[rhrhombidraw] (9.238pt,-16.0pt) -- (0.0pt,-32.0pt) -- (18.476pt,-32.0pt) -- cycle;
\draw (4.619pt,-24.0pt) -- (9.238pt,-26.656pt);
\draw (13.857pt,-24.0pt) -- (9.238pt,-26.656pt);
\draw (9.238pt,-32.0pt) -- (9.238pt,-26.656pt);
\draw (9.238pt,-32.0pt) -- (9.238pt,-37.344pt);
\draw (4.619pt,-40.0pt) -- (9.238pt,-37.344pt);
\draw (13.857pt,-40.0pt) -- (9.238pt,-37.344pt);
\fill[color=white] (4.619pt,-24.0pt) circle (1.2pt);
\draw[rhrhombidraw] (4.619pt,-24.0pt) circle (1.2pt);
\fill[color=white] (13.857pt,-24.0pt) circle (1.2pt);
\draw[rhrhombidraw] (13.857pt,-24.0pt) circle (1.2pt);
\fill[color=white] (9.238pt,-32.0pt) circle (1.2pt);
\draw[rhrhombidraw] (9.238pt,-32.0pt) circle (1.2pt);
\fill (9.238pt,-26.656pt) circle (1.5pt);
\fill (9.238pt,-37.344pt) circle (1.5pt);
\draw[rhrhombidraw] (-18.476pt,-32.0pt) -- (-27.714pt,-48.0pt) -- (-9.238pt,-48.0pt) -- cycle;
\draw (-23.095pt,-40.0pt) -- (-18.476pt,-42.656pt);
\draw (-13.857pt,-40.0pt) -- (-18.476pt,-42.656pt);
\draw (-18.476pt,-48.0pt) -- (-18.476pt,-42.656pt);
\draw (-18.476pt,-48.0pt) -- (-18.476pt,-53.344pt);
\draw (-23.095pt,-56.0pt) -- (-18.476pt,-53.344pt);
\draw (-13.857pt,-56.0pt) -- (-18.476pt,-53.344pt);
\fill[color=white] (-23.095pt,-40.0pt) circle (1.2pt);
\draw[rhrhombidraw] (-23.095pt,-40.0pt) circle (1.2pt);
\fill[color=white] (-13.857pt,-40.0pt) circle (1.2pt);
\draw[rhrhombidraw] (-13.857pt,-40.0pt) circle (1.2pt);
\fill[color=white] (-18.476pt,-48.0pt) circle (1.2pt);
\draw[rhrhombidraw] (-18.476pt,-48.0pt) circle (1.2pt);
\fill (-18.476pt,-42.656pt) circle (1.5pt);
\fill (-18.476pt,-53.344pt) circle (1.5pt);
\draw[rhrhombidraw] (0.0pt,-32.0pt) -- (-9.238pt,-48.0pt) -- (9.238pt,-48.0pt) -- cycle;
\draw (-4.619pt,-40.0pt) -- (0.0pt,-42.656pt);
\draw (4.619pt,-40.0pt) -- (0.0pt,-42.656pt);
\draw (0.0pt,-48.0pt) -- (0.0pt,-42.656pt);
\draw (0.0pt,-48.0pt) -- (0.0pt,-53.344pt);
\draw (-4.619pt,-56.0pt) -- (0.0pt,-53.344pt);
\draw (4.619pt,-56.0pt) -- (0.0pt,-53.344pt);
\fill[color=white] (-4.619pt,-40.0pt) circle (1.2pt);
\draw[rhrhombidraw] (-4.619pt,-40.0pt) circle (1.2pt);
\fill[color=white] (4.619pt,-40.0pt) circle (1.2pt);
\draw[rhrhombidraw] (4.619pt,-40.0pt) circle (1.2pt);
\fill[color=white] (0.0pt,-48.0pt) circle (1.2pt);
\draw[rhrhombidraw] (0.0pt,-48.0pt) circle (1.2pt);
\fill (0.0pt,-42.656pt) circle (1.5pt);
\fill (0.0pt,-53.344pt) circle (1.5pt);
\draw[rhrhombidraw] (18.476pt,-32.0pt) -- (9.238pt,-48.0pt) -- (27.714pt,-48.0pt) -- cycle;
\draw (13.857pt,-40.0pt) -- (18.476pt,-42.656pt);
\draw (23.095pt,-40.0pt) -- (18.476pt,-42.656pt);
\draw (18.476pt,-48.0pt) -- (18.476pt,-42.656pt);
\draw (18.476pt,-48.0pt) -- (18.476pt,-53.344pt);
\draw (13.857pt,-56.0pt) -- (18.476pt,-53.344pt);
\draw (23.095pt,-56.0pt) -- (18.476pt,-53.344pt);
\fill[color=white] (13.857pt,-40.0pt) circle (1.2pt);
\draw[rhrhombidraw] (13.857pt,-40.0pt) circle (1.2pt);
\fill[color=white] (23.095pt,-40.0pt) circle (1.2pt);
\draw[rhrhombidraw] (23.095pt,-40.0pt) circle (1.2pt);
\fill[color=white] (18.476pt,-48.0pt) circle (1.2pt);
\draw[rhrhombidraw] (18.476pt,-48.0pt) circle (1.2pt);
\fill (18.476pt,-42.656pt) circle (1.5pt);
\fill (18.476pt,-53.344pt) circle (1.5pt);
\draw[rhrhombidraw] (-27.714pt,-48.0pt) -- (-36.952pt,-64.0pt) -- (-18.476pt,-64.0pt) -- cycle;
\draw (-32.333pt,-56.0pt) -- (-27.714pt,-58.656pt);
\draw (-23.095pt,-56.0pt) -- (-27.714pt,-58.656pt);
\draw (-27.714pt,-64.0pt) -- (-27.714pt,-58.656pt);
\draw (-27.714pt,-64.0pt) -- (-27.714pt,-69.344pt);
\draw (-32.333pt,-72.0pt) -- (-27.714pt,-69.344pt);
\draw (-23.095pt,-72.0pt) -- (-27.714pt,-69.344pt);
\fill[color=white] (-32.333pt,-56.0pt) circle (1.2pt);
\draw[rhrhombidraw] (-32.333pt,-56.0pt) circle (1.2pt);
\fill[color=white] (-23.095pt,-56.0pt) circle (1.2pt);
\draw[rhrhombidraw] (-23.095pt,-56.0pt) circle (1.2pt);
\fill[color=white] (-27.714pt,-64.0pt) circle (1.2pt);
\draw[rhrhombidraw] (-27.714pt,-64.0pt) circle (1.2pt);
\fill (-27.714pt,-58.656pt) circle (1.5pt);
\fill (-27.714pt,-69.344pt) circle (1.5pt);
\draw[rhrhombidraw] (-9.238pt,-48.0pt) -- (-18.476pt,-64.0pt) -- (0.0pt,-64.0pt) -- cycle;
\draw (-13.857pt,-56.0pt) -- (-9.238pt,-58.656pt);
\draw (-4.619pt,-56.0pt) -- (-9.238pt,-58.656pt);
\draw (-9.238pt,-64.0pt) -- (-9.238pt,-58.656pt);
\draw (-9.238pt,-64.0pt) -- (-9.238pt,-69.344pt);
\draw (-13.857pt,-72.0pt) -- (-9.238pt,-69.344pt);
\draw (-4.619pt,-72.0pt) -- (-9.238pt,-69.344pt);
\fill[color=white] (-13.857pt,-56.0pt) circle (1.2pt);
\draw[rhrhombidraw] (-13.857pt,-56.0pt) circle (1.2pt);
\fill[color=white] (-4.619pt,-56.0pt) circle (1.2pt);
\draw[rhrhombidraw] (-4.619pt,-56.0pt) circle (1.2pt);
\fill[color=white] (-9.238pt,-64.0pt) circle (1.2pt);
\draw[rhrhombidraw] (-9.238pt,-64.0pt) circle (1.2pt);
\fill (-9.238pt,-58.656pt) circle (1.5pt);
\fill (-9.238pt,-69.344pt) circle (1.5pt);
\draw[rhrhombidraw] (9.238pt,-48.0pt) -- (0.0pt,-64.0pt) -- (18.476pt,-64.0pt) -- cycle;
\draw (4.619pt,-56.0pt) -- (9.238pt,-58.656pt);
\draw (13.857pt,-56.0pt) -- (9.238pt,-58.656pt);
\draw (9.238pt,-64.0pt) -- (9.238pt,-58.656pt);
\draw (9.238pt,-64.0pt) -- (9.238pt,-69.344pt);
\draw (4.619pt,-72.0pt) -- (9.238pt,-69.344pt);
\draw (13.857pt,-72.0pt) -- (9.238pt,-69.344pt);
\fill[color=white] (4.619pt,-56.0pt) circle (1.2pt);
\draw[rhrhombidraw] (4.619pt,-56.0pt) circle (1.2pt);
\fill[color=white] (13.857pt,-56.0pt) circle (1.2pt);
\draw[rhrhombidraw] (13.857pt,-56.0pt) circle (1.2pt);
\fill[color=white] (9.238pt,-64.0pt) circle (1.2pt);
\draw[rhrhombidraw] (9.238pt,-64.0pt) circle (1.2pt);
\fill (9.238pt,-58.656pt) circle (1.5pt);
\fill (9.238pt,-69.344pt) circle (1.5pt);
\draw[rhrhombidraw] (27.714pt,-48.0pt) -- (18.476pt,-64.0pt) -- (36.952pt,-64.0pt) -- cycle;
\draw (23.095pt,-56.0pt) -- (27.714pt,-58.656pt);
\draw (32.333pt,-56.0pt) -- (27.714pt,-58.656pt);
\draw (27.714pt,-64.0pt) -- (27.714pt,-58.656pt);
\draw (27.714pt,-64.0pt) -- (27.714pt,-69.344pt);
\draw (23.095pt,-72.0pt) -- (27.714pt,-69.344pt);
\draw (32.333pt,-72.0pt) -- (27.714pt,-69.344pt);
\fill[color=white] (23.095pt,-56.0pt) circle (1.2pt);
\draw[rhrhombidraw] (23.095pt,-56.0pt) circle (1.2pt);
\fill[color=white] (32.333pt,-56.0pt) circle (1.2pt);
\draw[rhrhombidraw] (32.333pt,-56.0pt) circle (1.2pt);
\fill[color=white] (27.714pt,-64.0pt) circle (1.2pt);
\draw[rhrhombidraw] (27.714pt,-64.0pt) circle (1.2pt);
\fill (27.714pt,-58.656pt) circle (1.5pt);
\fill (27.714pt,-69.344pt) circle (1.5pt);
\draw[rhrhombidraw] (-36.952pt,-64.0pt) -- (-46.19pt,-80.0pt) -- (-27.714pt,-80.0pt) -- cycle;
\draw (-41.571pt,-72.0pt) -- (-36.952pt,-74.656pt);
\draw (-32.333pt,-72.0pt) -- (-36.952pt,-74.656pt);
\draw (-36.952pt,-80.0pt) -- (-36.952pt,-74.656pt);
\fill[color=white] (-41.571pt,-72.0pt) circle (1.2pt);
\draw[rhrhombidraw] (-41.571pt,-72.0pt) circle (1.2pt);
\fill[color=white] (-32.333pt,-72.0pt) circle (1.2pt);
\draw[rhrhombidraw] (-32.333pt,-72.0pt) circle (1.2pt);
\fill[color=white] (-36.952pt,-80.0pt) circle (1.2pt);
\draw[rhrhombidraw] (-36.952pt,-80.0pt) circle (1.2pt);
\fill (-36.952pt,-74.656pt) circle (1.5pt);
\draw[rhrhombidraw] (-18.476pt,-64.0pt) -- (-27.714pt,-80.0pt) -- (-9.238pt,-80.0pt) -- cycle;
\draw (-23.095pt,-72.0pt) -- (-18.476pt,-74.656pt);
\draw (-13.857pt,-72.0pt) -- (-18.476pt,-74.656pt);
\draw (-18.476pt,-80.0pt) -- (-18.476pt,-74.656pt);
\fill[color=white] (-23.095pt,-72.0pt) circle (1.2pt);
\draw[rhrhombidraw] (-23.095pt,-72.0pt) circle (1.2pt);
\fill[color=white] (-13.857pt,-72.0pt) circle (1.2pt);
\draw[rhrhombidraw] (-13.857pt,-72.0pt) circle (1.2pt);
\fill[color=white] (-18.476pt,-80.0pt) circle (1.2pt);
\draw[rhrhombidraw] (-18.476pt,-80.0pt) circle (1.2pt);
\fill (-18.476pt,-74.656pt) circle (1.5pt);
\draw[rhrhombidraw] (0.0pt,-64.0pt) -- (-9.238pt,-80.0pt) -- (9.238pt,-80.0pt) -- cycle;
\draw (-4.619pt,-72.0pt) -- (0.0pt,-74.656pt);
\draw (4.619pt,-72.0pt) -- (0.0pt,-74.656pt);
\draw (0.0pt,-80.0pt) -- (0.0pt,-74.656pt);
\fill[color=white] (-4.619pt,-72.0pt) circle (1.2pt);
\draw[rhrhombidraw] (-4.619pt,-72.0pt) circle (1.2pt);
\fill[color=white] (4.619pt,-72.0pt) circle (1.2pt);
\draw[rhrhombidraw] (4.619pt,-72.0pt) circle (1.2pt);
\fill[color=white] (0.0pt,-80.0pt) circle (1.2pt);
\draw[rhrhombidraw] (0.0pt,-80.0pt) circle (1.2pt);
\fill (0.0pt,-74.656pt) circle (1.5pt);
\draw[rhrhombidraw] (18.476pt,-64.0pt) -- (9.238pt,-80.0pt) -- (27.714pt,-80.0pt) -- cycle;
\draw (13.857pt,-72.0pt) -- (18.476pt,-74.656pt);
\draw (23.095pt,-72.0pt) -- (18.476pt,-74.656pt);
\draw (18.476pt,-80.0pt) -- (18.476pt,-74.656pt);
\fill[color=white] (13.857pt,-72.0pt) circle (1.2pt);
\draw[rhrhombidraw] (13.857pt,-72.0pt) circle (1.2pt);
\fill[color=white] (23.095pt,-72.0pt) circle (1.2pt);
\draw[rhrhombidraw] (23.095pt,-72.0pt) circle (1.2pt);
\fill[color=white] (18.476pt,-80.0pt) circle (1.2pt);
\draw[rhrhombidraw] (18.476pt,-80.0pt) circle (1.2pt);
\fill (18.476pt,-74.656pt) circle (1.5pt);
\draw[rhrhombidraw] (36.952pt,-64.0pt) -- (27.714pt,-80.0pt) -- (46.19pt,-80.0pt) -- cycle;
\draw (32.333pt,-72.0pt) -- (36.952pt,-74.656pt);
\draw (41.571pt,-72.0pt) -- (36.952pt,-74.656pt);
\draw (36.952pt,-80.0pt) -- (36.952pt,-74.656pt);
\fill[color=white] (32.333pt,-72.0pt) circle (1.2pt);
\draw[rhrhombidraw] (32.333pt,-72.0pt) circle (1.2pt);
\fill[color=white] (41.571pt,-72.0pt) circle (1.2pt);
\draw[rhrhombidraw] (41.571pt,-72.0pt) circle (1.2pt);
\fill[color=white] (36.952pt,-80.0pt) circle (1.2pt);
\draw[rhrhombidraw] (36.952pt,-80.0pt) circle (1.2pt);
\fill (36.952pt,-74.656pt) circle (1.5pt);
\end{tikzpicture}}
\newcommand{\rhtikzrhombus}[1]{\smash{\raisebox{-2.5pt}{\scalebox{0.7}{\begin{tikzpicture}
  \draw[rhrhombidraw] (0.0pt,0.0pt) -- (-4.619pt,8.0pt) -- (0.0pt,16.0pt) -- (4.619pt,8.0pt) -- cycle;
  #1
\end{tikzpicture}}}}}

\newcommand{\rhc}{\rhtikzrhombus{}}

\newcommand{\rhsc}{\rhtikzrhombus{\draw[rhrhombithickside] (4.619pt,8.0pt) -- (-4.619pt,8.0pt);}}

\newcommand{\rhrhoul}{\rhtikzrhombus{\fill[rhrhombifill] (-4.618pt,8.0pt) -- (4.62pt,8.0pt) -- (0.0010pt,16.0pt) -- (-9.237pt,16.0pt) -- cycle;}}
\newcommand{\rhrhour}{\rhtikzrhombus{\fill[rhrhombifill] (0.0010pt,16.0pt) -- (-4.618pt,8.0pt) -- (4.62pt,8.0pt) -- (9.239pt,16.0pt) -- cycle;}}
\newcommand{\rhrholl}{\rhtikzrhombus{\fill[rhrhombifill] (-4.618pt,8.0pt) -- (-9.237pt,0.0pt) -- (0.0010pt,0.0pt) -- (4.62pt,8.0pt) -- cycle;}}
\newcommand{\rhrholr}{\rhtikzrhombus{\fill[rhrhombifill] (0.0010pt,0.0pt) -- (9.239pt,0.0pt) -- (4.62pt,8.0pt) -- (-4.618pt,8.0pt) -- cycle;}}
\newcommand{\rhrhul}{\rhtikzrhombus{\fill[rhrhombifill] (-9.237pt,16.0pt) -- (-13.856pt,8.0pt) -- (-4.618pt,8.0pt) -- (0.0010pt,16.0pt) -- cycle;}}
\newcommand{\rhrhur}{\rhtikzrhombus{\fill[rhrhombifill] (4.62pt,8.0pt) -- (13.858pt,8.0pt) -- (9.239pt,16.0pt) -- (0.0010pt,16.0pt) -- cycle;}}
\newcommand{\rhrhll}{\rhtikzrhombus{\fill[rhrhombifill] (-9.237pt,0.0pt) -- (0.0010pt,0.0pt) -- (-4.618pt,8.0pt) -- (-13.856pt,8.0pt) -- cycle;}}
\newcommand{\rhrhlr}{\rhtikzrhombus{\fill[rhrhombifill] (4.62pt,8.0pt) -- (0.0010pt,0.0pt) -- (9.239pt,0.0pt) -- (13.858pt,8.0pt) -- cycle;}}
\newcommand{\rhrhpl}{\rhtikzrhombus{\fill[rhrhombifill] (-4.619pt,8.0pt) -- (-9.238pt,16.0pt) -- (-13.857pt,8.0pt) -- (-9.238pt,0.0pt) -- cycle;}}
\newcommand{\rhrhpr}{\rhtikzrhombus{\fill[rhrhombifill] (13.857pt,8.0pt) -- (9.238pt,16.0pt) -- (4.619pt,8.0pt) -- (9.238pt,0.0pt) -- cycle;}}

\newcommand{\rhacM}{\rhtikzrhombus{\draw[rhrhombithickside] (4.619pt,8.0pt) -- (-4.619pt,8.0pt);
\draw[->,rhrhombiarrow] (0.0pt,4.0pt) -- (0.0pt,12.0pt);}}
\newcommand{\rhacW}{\rhtikzrhombus{\draw[rhrhombithickside] (4.619pt,8.0pt) -- (-4.619pt,8.0pt);
\draw[->,rhrhombiarrow] (0.0pt,12.0pt) -- (0.0pt,4.0pt);}}
\newcommand{\rhaoulW}{\rhtikzrhombus{\draw[rhrhombithickside] (-4.618pt,8.0pt) -- (0.0010pt,16.0pt);
\draw[->,rhrhombiarrow] (1.155pt,10.0pt) -- (-5.773pt,14.0pt);}}
\newcommand{\rhaourM}{\rhtikzrhombus{\draw[rhrhombithickside] (0.0010pt,16.0pt) -- (4.62pt,8.0pt);
\draw[->,rhrhombiarrow] (5.774pt,14.0pt) -- (-1.154pt,10.0pt);}}
\newcommand{\rhaourW}{\rhtikzrhombus{\draw[rhrhombithickside] (0.0010pt,16.0pt) -- (4.62pt,8.0pt);
\draw[->,rhrhombiarrow] (-1.154pt,10.0pt) -- (5.775pt,14.0pt);}}
\newcommand{\rhaollM}{\rhtikzrhombus{\draw[rhrhombithickside] (-4.618pt,8.0pt) -- (0.0010pt,0.0pt);
\draw[->,rhrhombiarrow] (1.155pt,6.0pt) -- (-5.773pt,2.0pt);}}
\newcommand{\rhaollW}{\rhtikzrhombus{\draw[rhrhombithickside] (-4.618pt,8.0pt) -- (0.0010pt,0.0pt);
\draw[->,rhrhombiarrow] (-5.773pt,2.0pt) -- (1.156pt,6.0pt);}}
\newcommand{\rhaolrM}{\rhtikzrhombus{\draw[rhrhombithickside] (0.0010pt,0.0pt) -- (4.62pt,8.0pt);
\draw[->,rhrhombiarrow] (-1.154pt,6.0pt) -- (5.775pt,2.0pt);}}
\newcommand{\rhaolrW}{\rhtikzrhombus{\draw[rhrhombithickside] (0.0010pt,0.0pt) -- (4.62pt,8.0pt);
\draw[->,rhrhombiarrow] (5.774pt,2.0pt) -- (-1.154pt,6.0pt);}}
\newcommand{\rhpcMl}{\rhtikzrhombus{\draw[rhrhombipatharrow,-my] (0.0pt,8.0pt) arc (0:60:4.619pt);}}
\newcommand{\rhpcMr}{\rhtikzrhombus{\draw[rhrhombipatharrow,-my] (0.0pt,8.0pt) arc (180:120:4.619pt);}}

\newcommand{\rhpcMrl}{\rhtikzrhombus{\draw[rhrhombipatharrow,-my] (0.0pt,8.0pt) arc (180:120:4.619pt) arc (300:360:4.619pt);}}
\newcommand{\rhpcMrr}{\rhtikzrhombus{\draw[rhrhombipatharrow,-my] (0.0pt,8.0pt) arc (180:120:4.619pt) arc (120:60:4.619pt);}}
\newcommand{\rhpcMrrl}{\rhtikzrhombus{\draw[rhrhombipatharrow,-my] (0.0pt,8.0pt) arc (180:120:4.619pt) arc (120:60:4.619pt) arc (240:300:4.619pt);}}
\newcommand{\rhpcMrrr}{\rhtikzrhombus{\draw[rhrhombipatharrow,-my] (0.0pt,8.0pt) arc (180:120:4.619pt) arc (120:60:4.619pt) arc (60:0:4.619pt);}}
\newcommand{\rhpcWr}{\rhtikzrhombus{\draw[rhrhombipatharrow,-my] (0.0pt,8.0pt) arc (0:-60:4.619pt);}}
\newcommand{\rhpcWrl}{\rhtikzrhombus{\draw[rhrhombipatharrow,-my] (0.0pt,8.0pt) arc (0:-60:4.619pt) arc (120:180:4.619pt);}}
\newcommand{\rhpoulMl}{\rhtikzrhombus{\draw[rhrhombipatharrow,-my] (-2.309pt,12.0pt) arc (240:300:4.619pt);}}
\newcommand{\rhpoulMr}{\rhtikzrhombus{\draw[rhrhombipatharrow,-my] (-2.309pt,12.0pt) arc (60:0:4.619pt);}}

\newcommand{\rhpoulMlr}{\rhtikzrhombus{\draw[rhrhombipatharrow,-my] (-2.309pt,12.0pt) arc (240:300:4.619pt) arc (120:60:4.619pt);}}
\newcommand{\rhpoulMrl}{\rhtikzrhombus{\draw[rhrhombipatharrow,-my] (-2.309pt,12.0pt) arc (60:0:4.619pt) arc (180:240:4.619pt);}}
\newcommand{\rhpoulMrr}{\rhtikzrhombus{\draw[rhrhombipatharrow,-my] (-2.309pt,12.0pt) arc (60:0:4.619pt) arc (0:-60:4.619pt);}}
\newcommand{\rhpourMl}{\rhtikzrhombus{\draw[rhrhombipatharrow,-my] (2.31pt,12.0pt) arc (120:180:4.619pt);}}
\newcommand{\rhpourMr}{\rhtikzrhombus{\draw[rhrhombipatharrow,-my] (2.31pt,12.0pt) arc (300:240:4.619pt);}}
\newcommand{\rhpourMll}{\rhtikzrhombus{\draw[rhrhombipatharrow,-my] (2.31pt,12.0pt) arc (120:180:4.619pt) arc (180:240:4.619pt);}}
\newcommand{\rhpourMlr}{\rhtikzrhombus{\draw[rhrhombipatharrow,-my] (2.31pt,12.0pt) arc (120:180:4.619pt) arc (0:-60:4.619pt);}}
\newcommand{\rhpourWl}{\rhtikzrhombus{\draw[rhrhombipatharrow,-my] (2.31pt,12.0pt) arc (300:360:4.619pt);}}
\newcommand{\rhpollMl}{\rhtikzrhombus{\draw[rhrhombipatharrow,-my] (-2.309pt,4.0pt) arc (120:180:4.619pt);}}
\newcommand{\rhpollMr}{\rhtikzrhombus{\draw[rhrhombipatharrow,-my] (-2.309pt,4.0pt) arc (300:240:4.619pt);}}

\newcommand{\rhpollWl}{\rhtikzrhombus{\draw[rhrhombipatharrow,-my] (-2.309pt,4.0pt) arc (300:360:4.619pt);}}
\newcommand{\rhpollWr}{\rhtikzrhombus{\draw[rhrhombipatharrow,-my] (-2.309pt,4.0pt) arc (120:60:4.619pt);}}
\newcommand{\rhpollWll}{\rhtikzrhombus{\draw[rhrhombipatharrow,-my] (-2.309pt,4.0pt) arc (300:360:4.619pt) arc (0:60:4.619pt);}}
\newcommand{\rhpollWlr}{\rhtikzrhombus{\draw[rhrhombipatharrow,-my] (-2.309pt,4.0pt) arc (300:360:4.619pt) arc (180:120:4.619pt);}}
\newcommand{\rhpolrWl}{\rhtikzrhombus{\draw[rhrhombipatharrow,-my] (2.31pt,4.0pt) arc (60:120:4.619pt);}}

\newcommand{\rhpolrWlr}{\rhtikzrhombus{\draw[rhrhombipatharrow,-my] (2.31pt,4.0pt) arc (60:120:4.619pt) arc (300:240:4.619pt);}}
\newcommand{\rhpolrWrl}{\rhtikzrhombus{\draw[rhrhombipatharrow,-my] (2.31pt,4.0pt) arc (240:180:4.619pt) arc (0:60:4.619pt);}}
\newcommand{\rhpolrWrr}{\rhtikzrhombus{\draw[rhrhombipatharrow,-my] (2.31pt,4.0pt) arc (240:180:4.619pt) arc (180:120:4.619pt);}}

\newcommand{\rhpulWr}{\rhtikzrhombus{\draw[rhrhombipatharrow,-my] (-6.928pt,12.0pt) arc (120:60:4.619pt);}}

\newcommand{\rhplrMrl}{\rhtikzrhombus{\draw[rhrhombipatharrow,-my] (6.929pt,4.0pt) arc (300:240:4.619pt) arc (60:120:4.619pt);}}

\newcommand{\rhppulWl}{\rhtikzrhombus{\draw[rhrhombipatharrow,-my] (-4.619pt,16.0pt) arc (180:240:4.619pt);}}

\newcommand{\rhpplrMlr}{\rhtikzrhombus{\draw[rhrhombipatharrow,-my] (4.619pt,0.0pt) arc (0:60:4.619pt) arc (240:180:4.619pt);}}

\newcommand{\rhpoulMlXolrWl}{\rhtikzrhombus{\draw[rhrhombipatharrow,-my] (-2.309pt,12.0pt) arc (240:300:4.619pt);\draw[rhrhombipatharrow,-my] (2.31pt,4.0pt) arc (60:120:4.619pt);}}

\NewEnviron{tikzpictured}{\smash{\scalebox{0.7}{\begin{tikzpicture}[baseline=4pt]\BODY\end{tikzpicture}}}}
\NewEnviron{tikzpicturednosmash}{{\scalebox{0.7}{\begin{tikzpicture}[baseline=4pt]\BODY\end{tikzpicture}}}}

\newcommand{\mydebug}[1]{\textcolor{red}{\textbf{MYDEBUG}}}

\newcommand{\nopar}{}

\newcommand{\N}{\ensuremath{\mathbb{N}}}
\newcommand{\Z}{\ensuremath{\mathbb{Z}}}
\newcommand{\R}{\ensuremath{\mathbb{R}}}
\newcommand{\C}{\ensuremath{\mathbb{C}}}

\newcommand{\gO}{\ensuremath{\mathcal{O}}}
\newcommand{\LRC}[3]{\ensuremath{c_{#1,#2}^{#3}}}

\newcommand{\resp}[2]{\ensuremath{R^{#2}_{#1}}}
\newcommand{\res}{\ensuremath{R}}
\newcommand{\resf}[1]{\ensuremath{R_{#1}}}

\newcommand{\s}[2]{\sigma(#1,#2)}

\newcommand{\Algo}{\ensuremath{\mathscr{A}}}

\newcommand{\nffx}{\mbox{\text{nearly $f$\dash flat}}}
\newcommand{\nff}{\nffx\xspace}

\newcommand\dash{\nobreakdash-\hspace{0pt}}

\newcommand{\sP}{\bf \ensuremath{\#\P}}
\def\GL{\mathrm{GL}}

\def\CP{\sP}
\def\P{\mathrm{\bf P}}
\def\NP{\mathrm{\bf NP}}
\def\la{\lambda}

\newcommand{\oF}{\overline{F}}

\newcommand{\iot}{{\iota}}

\newcommand{\comb}{{\pi}}
\newcommand{\Gamm}[1]{{\tilde{\Gamma}(#1)}}
\newcommand{\pstart}[1]{\textup{\textsf{start}}({#1})}
\newcommand{\pend}[1]{\textup{\textsf{end}}({#1})}
\let\lparen=(\let\rparen=)
\newcommand{\print}[1]{\ensuremath{\texttt{print\lparen}#1\texttt{\rparen}}}
\newcommand{\algocaptionwithoutparam}[1]{\noindent\mbox{{\scshape #1}}}

\def\Return{\textbf{return}\xspace}
\def\True{\textbf{TRUE}\xspace}
\def\False{\textbf{FALSE}\xspace}
\newcommand{\p}{\ensuremath{\Psi}}
\newcommand{\poly}{\text{\normalfont{\textsf{poly}}}}

\newcommand{\inflo}{\mathsf{inflow}}
\newcommand{\outflo}{\mathsf{outflow}}

\newcommand{\blsquare}{\scalebox{0.7}{\ensuremath{\blacksquare}}} 

\swapnumbers
\newtheorem{theorem}{Theorem}[section]

\newtheorem{lemma}[theorem]{Lemma}
\newtheorem{keylemma}[theorem]{Key Lemma}
\newtheorem{proposition}[theorem]{Proposition}
\newtheorem{claim}[theorem]{Claim}
\newtheorem{problem}[theorem]{Problem}
\newtheorem{lemma_}{Lemma}

\theoremstyle{definition}

\newtheorem{INTERNALdefinition}[theorem]{Definition}
\newenvironment{definition}[1][default]{\ifthenelse{\equal{#1}{default}}{\begin{INTERNALdefinition}}{\begin{INTERNALdefinition}[#1]}\begin{samepage}}{~\hfill $\blsquare$\end{samepage}\end{INTERNALdefinition}}

\newtheorem{INTERNALremark}[theorem]{Remark}
\newenvironment{remark}{\begin{INTERNALremark}}{~\hfill $\blsquare$\end{INTERNALremark}}
\newtheorem{INTERNALdefinition_}[lemma_]{Definition}

\title{Small Littlewood-Richardson coefficients}
\author{Christian Ikenmeyer\thanks{University of Paderborn, Germany, Email address: ciken$@$math.upb.de}}

\let\oldsection=\section
\renewcommand{\section}{\setcounter{theorem}{0}\oldsection}

\begin{document}
\sloppy

\maketitle

\begin{abstract}
We develop structural insights into the Littlewood-Richardson graph, whose number of vertices equals the Littlewood-Richardson coefficient $\LRC \la \mu \nu$ for given partitions $\la$, $\mu$ and $\nu$.
This graph was first introduced in \cite{bi:12}, where its connectedness was proved.

Our insights are useful for the design of algorithms for computing the Littlewood-Richardson coefficient:
We design an algorithm for the exact computation of $\LRC \la \mu \nu$ with running time
$\gO\big((\LRC \la \mu \nu)^2 \cdot \poly(n)\big)$, where $\la$, $\mu$, and~$\nu$ are partitions of length at most $n$.
Moreover, we introduce an algorithm for deciding whether $\LRC \la \mu \nu \geq t$ whose running time is
$\gO\big(t^2 \cdot \poly(n)\big)$.
Even the existence of a polynomial-time algorithm for deciding whether $\LRC \la \mu \nu \geq 2$ is a nontrivial new result on its own.

Our insights also lead to the proof of a conjecture by King, Tollu, and Toumazet posed in \cite{ktt:04},
stating that $\LRC \la \mu \nu = 2$ implies $\LRC {M\la} {M\mu} {M\nu} = M+1$ for all $M \in \N$.
Here, the stretching of partitions is defined componentwise.
\end{abstract}

\noindent{\small 2010 MSC: 05E10, 22E46, 90C27}

\medskip

{\small
\renewcommand{\arraystretch}{0.9}
\setlength{\tabcolsep}{0cm}
\noindent\begin{tabular}{rcl}
Keywords: & \hphantom{i} & Littlewood-Richardson coefficient, hive model,\\
& \hphantom{i} &efficient algorithms, flows in networks
\end{tabular}
}

\tableofcontents

\section{Introduction}
Let $\la,\mu,\nu\in\Z^n$ be nonincreasing $n$-tuples of integers.
The \emph{Littlewood-Richardson coefficient} $\LRC{\lambda}{\mu}{\nu}$ of $\la$, $\mu$ and $\nu$
is defined as the multiplicity of the irreducible $\GL_n(\C)$-representation~$V_\nu$ 
with dominant weight $\nu$ in the tensor product $V_\la\otimes V_\mu$.
These coefficients appear not only in representation theory and 
algebraic combinatorics, but also in topology and enumerative geometry.

Different combinatorial characterizations of the Littlewood-Richardson coefficients are known. 
The classic Littlewood-Richardson rule (cf.~\cite{fult:97}) counts certain skew tableaux, 
while in~\cite{bz:92} the number of integer points of certain polytopes are counted.  
A beautiful characterization was given by Knutson and Tao~\cite{knta:99}, 
who characterized Littlewood-Richardson coefficients either as the number of honeycombs 
or hives with prescribed boundary conditions. 

The focus of this paper is on the complexity of computing the Littlewood-Richardson coefficient 
$\LRC{\lambda}{\mu}{\nu}$ on input $\la,\mu,\nu$. Without loss of generality we assume that 
the components of $\la,\mu,\nu$ are nonnegative integers and put 
$|\la| := \sum_i \la_i$. Moreover we write 
$\ell(\la)$ for the number of nonzero components of~$\la$.
Then $|\nu| = |\la| + |\mu|$ and $\nu_1\geq \max\{\la_1,\mu_1\}$ are necessary conditions for $\LRC{\lambda}{\mu}{\nu} > 0$. 
All known algorithms for computing Littlewood-Richardson coefficients 
take exponential time in the worst case. 
Narayanan~\cite{nara:06} proved that this is unavoidable:
the computation of $\LRC{\la}{\mu}{\nu}$ is a $\CP$-complete problem.
Hence there does not exist a polynomial time algorithm
for computing $\LRC{\la}{\mu}{\nu}$ under the widely believed hypothesis $\P\ne\NP$.

\smallskip

\noindent {\bf Main results.} 
This is a follow-up paper to \cite{bi:12}.
We use the characterization of $\LRC{\la}{\mu}{\nu}$ as the number of 
{\em hive flows with prescribed border throughput} on the honeycomb graph~$G$,
cf.\ Figures~\ref{fig:diagrams}--\ref{fig:degengraph}.
Besides capacity constraints given by $\la,\mu,\nu$, these flows have to satisfy 
rhombus inequalities corresponding to the ones considered in~\cite{knta:99, buc:00}.  

The integral hive flows form the integral points of the \emph{hive flow polytope} $P(\la,\mu,\nu)$.
The vertices of this polytope, together with edges given by cycles on the honeycomb graph,
form an undirected graph, whose connectedness was proved in \cite{bi:12}, see Section~\ref{subsec:seccycles}.
To compute $\LRC{\la}{\mu}{\nu}$ we design a variant of breadth-first-search
that lists all points in $P(\la,\mu,\nu)$ with only polynomial delay between the single outputs.
This enables us to decide $\LRC \la \mu \nu \geq t$ in time
$\gO\big(t^2 \cdot \poly(n)\big)$, see Theorem~\ref{thm:decide},
where $\poly(n)$ stands for a function that is polynomially bounded in $n$.
Even the polynomial time algorithm for deciding
$\LRC \la \mu \nu \geq 2$ is a new and nontrivial result.
Also we get an algorithm for computing $\LRC \la \mu \nu$
which runs in time $\gO\big((\LRC \la \mu \nu)^2 \cdot \poly(n)\big)$, see Theorem~\ref{thm:compute}.
All algorithms in this paper only use addition, multiplication, and comparison
and the running time is defined to be the number of these operations.

With only minor modifications our algorithms can be used to enumerate efficiently all hive flows
corresponding to a given tensor product $V_\la\otimes V_\mu$, as asked in~\cite[p.~186]{kt:01}.

Our algorithms are implemented and available.
We encourage the reader to try out our Java applet at\\
\indent{\tt http://www-math.upb.de/agpb/flowapplet/flowapplet.html}.

Moreover, our insights into the structure of the hive flow polytope lead to the proof
of the following conjecture of King, Tollu, and Toumazet posed in \cite{ktt:04}.
\begin{theorem}
\label{thm:ktt}
Given partitions $\la$, $\mu$ and $\nu$ such that $|\nu|=|\la|+|\mu|$.
Then $\LRC \la \mu \nu = 2$ implies $\LRC {M\la} {M\mu} {M\nu} = M+1$ for all $M \in \N$.
\end{theorem}

We remark that \cite[Conj.~3.3]{ktt:04} also contains the conjecture
that $\LRC \la \mu \nu = 3$ implies either $\LRC {M\la} {M\mu} {M\nu} = 2M+1$
or $\LRC {M\la} {M\mu} {M\nu} = (M+1)(M+2)/2$.
We think that a careful refinement of the methods in this paper
can be used to prove this and similar conjectures as well.

Several sections of this paper have a large overlap with \cite{bi:12}.
Most of the original content appears in Sections \ref{sec:enumeration}, \ref{sec:neighborhoodgen} and \ref{sec:ktt}.
The reader accustomed to the arguments in \cite{bi:12} will find it easier to understand this follow-up paper.
The Sections \ref{sec:flowdescr}-\ref{sec:neighborhoodgen} use results from~\cite{bi:12},
but when such results are used, then they are explicitly stated,
so that these sections can be studied without reading~\cite{bi:12}.
Section~\ref{sec:shoturncyc} however needs a profound understanding of the proof of Theorem~4.8 in~\cite{bi:12}.
Section~\ref{sec:ktt} is nearly self-contained, except for some references to Section~\ref{sec:shoturncyc}.

\paragraph{Acknowledgements}
I benefitted tremendously from the long, intense, and invaluable discussions with my PhD thesis advisor Peter B\"urgisser.
Furthermore, I thank the Deutsche Forschungsgemeinschaft for their financial support (DFG-grants BU 1371/3-1 and BU 1371/3-2).

\section{Flow description of LR coefficients}
\label{sec:flowdescr}

\subsection{Flows on digraphs}
\label{se:flows}

We fix some terminology regarding flows on directed graphs, compare~\cite{amo:93}.
Let $D$ be a digraph with vertex set $V(D)$ and edge set $E(D)$. 
Let $e_\textsf{start}$ denote the vertex where the edge~$e$ starts and 
$e_\textsf{end}$ the vertex where $e$ ends.
The {\em inflow} and {\em outflow} 
of a map $f\colon E(D)\rightarrow \R$ 
at a vertex $v\in V(D)$ are defined as 
$$
\inflo(v,f) := \sum\limits_{e_\textsf{end}=v} f(e),\quad 
\outflo(v,f) := \sum\limits_{e_\textsf{start}=v} f(e),
$$
respectively. 
A {\em flow on $D$} is defined as a map
$f\colon E(D)\rightarrow \R$ that satisfies 
Kirchhoff's conservation laws: 
$\inflo(v,f) = \outflo(v,f)$ for all 
$v \in V(D)$. 

The set of flows on $D$ is a vector space that we denote by $F(D)$.
A flow is called {\em integral} if it takes only integer values and 
we denote by $F(D)_\Z$ the group of integral flows on $D$.

A {\em path~$p$} in $D$ is defined as a sequence 
$x_0,\ldots,x_\ell$ of pairwise distinct vertices of~$D$ such that 
$(x_{i-1},x_i) \in E$ for all $1\le i\le\ell$.
A sequence $x_0,\ldots,x_\ell$ of vertices of $D$ is called a 
{\em cycle~$c$} if $x_0,\ldots,x_{\ell-1}$ are pairwise distinct,
$x_{\ell}=x_0$, and  $(x_{i-1},x_i) \in E$ 
for all $1\le i\le\ell$. 
It will be sometimes useful to identify a path or a cycle with the subgraph consisting of its vertices $x_0,x_1,\ldots,x_\ell$
and edges $(x_0,x_1),\ldots,(x_{\ell-1},x_\ell)$.
Since the starting vertex~$x_0$ of a cycle is not relevant, this does not harm. 

A cycle~$c$ in $D$ defines a flow~$f$ on $D$ by setting 
$f(e) := 1$ if $e\in c$ and $f(e):=0$ otherwise. 
It will be convenient to denote this flow with~$c$ as well. 
A flow is called {\em nonnegative} if $f(e)\ge 0$ for all edges $e\in E$.

We will study flows in two rather different situations. 
The residual digraph $R$ introduced in Section~\ref{sec:residual}
has the property that it never contains an edge $(u,v)$ 
and its reverse edge $(v,u)$. Only nonnegative flows on $R$ 
will be of interest. 

On the other hand, we also need to look at flows 
on digraphs resulting from an undirected graph $G$ 
by replacing each of its undirected edges $\{u,v\}$ by 
the directed edge $e:=(u,v)$ and its reverse  $-e := (v,u)$.
We shall denote the resulting digraph also by $G$. 
To a flow $f$ on~$G$ we assign its 
{\em reduced representative} $\tilde{f}$ defined by 
$\tilde{f}(e) := f(e) -f(-e) \ge 0$ and $\tilde{f}(-e) :=0$ 
if $f(e) \ge f(-e)$, 
and setting 
$\tilde{f}(e) := 0$ and $\tilde{f}(-e) := f(-e) -f(e)$ if $f(e) < f(-e)$. 
It will be convenient to interpret 
$f$ and $\tilde{f}$ as manifestations of the same flow:
Formally, we consider the linear subspace 
$N(G) := \{f \in \R^{E(G)} \mid \forall e \in E(G): f(e)=f(-e)\}$
of ``null flows'' and the factor space 
\begin{equation*}
 \oF(G) := F(G) / N(G).
\end{equation*}
We call the elements of $\oF(G)$ {\em flow classes on $G$} 
(or simply flows) and denote them by the same symbols as for flows. 
No confusion should arise from this abuse of notation in the context at hand. 
We usually identify flow classes with their reduced representative. 
A flow class is called {\em integral} if its reduced representative is integral 
and we denote by $\oF(G)_\Z$ the group of integral flow classes on $G$.

\subsection{Flows on the honeycomb graph $G$}
\label{se:honey}

We start with a triangular array of vertices, $n+1$ on each side, as seen in Figure~\ref{fig:triangulararray}.
The resulting planar graph $\Delta$ shall be called the {\em triangular graph} with parameter~$n$, 
we denote its vertex set with $V(\Delta)$ and its edge set with $E(\Delta)$. 
A triangle consisting of three edges in~$\Delta$ is called a \emph{hive triangle}.
Note that there are two types of hive triangles: upright (`$\bigtriangleup$') and downright oriented ones (`$\bigtriangledown$'). 
A \emph{rhombus} is defined to be the union of an upright and a downright hive triangle which share a common side.
In contrast to the usual geometric definition of the term \emph{rhombus}
we use this term here in this very restricted sense only.
Note that the angles at the corners of a rhombus are either acute of $60^\circ$ or obtuse of $120^\circ$.
Two distinct rhombi are called \emph{overlapping} if they share a hive triangle.
\begin{figure}[h] 
  \begin{center}
      \subfigure[The triangular graph $\Delta$ for $n=5$.]
        {\scalebox{1.8}{\FIGtriangulararray}
        \nopar\label{fig:triangulararray}} \hspace{0.5cm}
      \subfigure[The honeycomb graph $G$. The vertex of the outer face is omitted.]
        {\scalebox{0.9}{\FIGdualtriangulararray}
        \nopar\label{fig:dualtriangulararray}} \hspace{0.5cm}
      \subfigure[The fixed border throughputs.] 
        {
\scalebox{0.7}{
\begin{tikzpicture}[scale=5]\draw[rhrhombidraw] (0.0pt,0.0pt) -- (12.5pt,21.65pt) -- (25.0pt,0.0pt) -- cycle;\draw[rhrhombidraw] (2.5pt,4.33pt) -- (22.5pt,4.33pt) ;\draw[rhrhombidraw] (5.0pt,8.66pt) -- (20.0pt,8.66pt) ;\draw[rhrhombidraw] (7.5pt,12.99pt) -- (17.5pt,12.99pt) ;\draw[rhrhombidraw] (10.0pt,17.32pt) -- (15.0pt,17.32pt) ;\draw[rhrhombidraw] (15.0pt,17.32pt) -- (5.0pt,0.0pt) ;\draw[rhrhombidraw] (17.5pt,12.99pt) -- (10.0pt,0.0pt) ;\draw[rhrhombidraw] (20.0pt,8.66pt) -- (15.0pt,0.0pt) ;\draw[rhrhombidraw] (22.5pt,4.33pt) -- (20.0pt,0.0pt) ;\draw[rhrhombidraw] (20.0pt,0.0pt) -- (10.0pt,17.32pt) ;\draw[rhrhombidraw] (15.0pt,0.0pt) -- (7.5pt,12.99pt) ;\draw[rhrhombidraw] (10.0pt,0.0pt) -- (5.0pt,8.66pt) ;\draw[rhrhombidraw] (5.0pt,0.0pt) -- (2.5pt,4.33pt) ;\draw[rhrhombithickside] (12.5pt,21.65pt) -- (15.0pt,17.32pt);\draw[->,rhrhombiarrow] (15.915pt,20.735pt) -- (11.585pt,18.235pt);\draw[rhrhombithickside] (15.0pt,17.32pt) -- (17.5pt,12.99pt);\draw[->,rhrhombiarrow] (18.415pt,16.405pt) -- (14.085pt,13.905pt);\draw[rhrhombithickside] (17.5pt,12.99pt) -- (20.0pt,8.66pt);\draw[->,rhrhombiarrow] (20.915pt,12.075pt) -- (16.585pt,9.575pt);\draw[rhrhombithickside] (20.0pt,8.66pt) -- (22.5pt,4.33pt);\draw[->,rhrhombiarrow] (23.415pt,7.745pt) -- (19.085pt,5.245pt);\draw[rhrhombithickside] (22.5pt,4.33pt) -- (25.0pt,0.0pt);\draw[->,rhrhombiarrow] (25.915pt,3.415pt) -- (21.585pt,0.915pt);\draw[rhrhombithickside] (25.0pt,0.0pt) -- (20.0pt,0.0pt);\draw[->,rhrhombiarrow] (22.5pt,-2.5pt) -- (22.5pt,2.5pt);\draw[rhrhombithickside] (20.0pt,0.0pt) -- (15.0pt,0.0pt);\draw[->,rhrhombiarrow] (17.5pt,-2.5pt) -- (17.5pt,2.5pt);\draw[rhrhombithickside] (15.0pt,0.0pt) -- (10.0pt,0.0pt);\draw[->,rhrhombiarrow] (12.5pt,-2.5pt) -- (12.5pt,2.5pt);\draw[rhrhombithickside] (10.0pt,0.0pt) -- (5.0pt,0.0pt);\draw[->,rhrhombiarrow] (7.5pt,-2.5pt) -- (7.5pt,2.5pt);\draw[rhrhombithickside] (5.0pt,0.0pt) -- (0.0pt,0.0pt);\draw[->,rhrhombiarrow] (2.5pt,-2.5pt) -- (2.5pt,2.5pt);\draw[rhrhombithickside] (12.5pt,21.65pt) -- (10.0pt,17.32pt);\draw[->,rhrhombiarrow] (13.415pt,18.235pt) -- (9.085pt,20.735pt);\draw[rhrhombithickside] (10.0pt,17.32pt) -- (7.5pt,12.99pt);\draw[->,rhrhombiarrow] (10.915pt,13.905pt) -- (6.585pt,16.405pt);\draw[rhrhombithickside] (7.5pt,12.99pt) -- (5.0pt,8.66pt);\draw[->,rhrhombiarrow] (8.415pt,9.575pt) -- (4.085pt,12.075pt);\draw[rhrhombithickside] (5.0pt,8.66pt) -- (2.5pt,4.33pt);\draw[->,rhrhombiarrow] (5.915pt,5.245pt) -- (1.585pt,7.745pt);\draw[rhrhombithickside] (2.5pt,4.33pt) -- (0.0pt,0.0pt);\draw[->,rhrhombiarrow] (3.415pt,0.915pt) -- (-0.915pt,3.415pt);\node at (17.5pt,21.65pt) {{$\la_1$}};\node at (20.0pt,17.32pt) {{$\la_2$}};\node at (22.5pt,12.99pt) {{$\la_3$}};\node at (25.0pt,8.66pt) {{$\la_4$}};\node at (27.5pt,4.33pt) {{$\la_5$}};\node at (7.5pt,21.65pt) {{$\nu_1$}};\node at (5.0pt,17.32pt) {{$\nu_2$}};\node at (2.5pt,12.99pt) {{$\nu_3$}};\node at (0.0pt,8.66pt) {{$\nu_4$}};\node at (-2.5pt,4.33pt) {{$\nu_5$}};\node at (22.5pt,-4.33pt) {\rotatebox{45}{$\mu_1$}};\node at (17.5pt,-4.33pt) {\rotatebox{45}{$\mu_2$}};\node at (12.5pt,-4.33pt) {\rotatebox{45}{$\mu_3$}};\node at (7.5pt,-4.33pt) {\rotatebox{45}{$\mu_4$}};\node at (2.5pt,-4.33pt) {\rotatebox{45}{$\mu_5$}};\end{tikzpicture}
}
        \nopar\label{fig:triangulararraywithflow}}
    \caption{Graph constructions.} 
    \nopar\label{fig:diagrams}
  \end{center}
\end{figure}
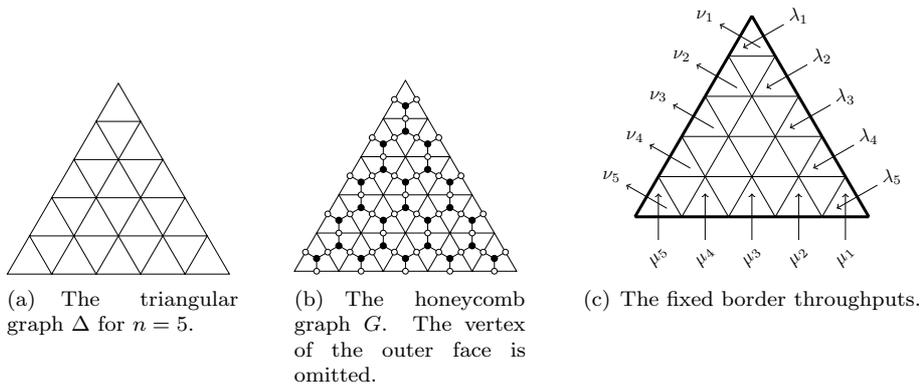

To realize the dual graph of $\Delta$, as in~\cite{buc:00}, 
we introduce a black vertex in the middle of each hive triangle
and a white vertex on each hive triangle side, see Figure~\ref{fig:dualtriangulararray}.
Moreover, in each hive triangle~$T$, we introduce edges connecting 
the three white vertices of $T$ with the black vertex. 
Moreover, we add a single black vertex to the outer face and connect it
with each white vertex at the border of $\Delta$.
Clearly, the resulting (undirected) graph~$G$ is bipartite and planar. 
We shall call $G$ the {\em honeycomb graph} with parameter~$n$. 
The cycles on $G$ that do not pass the outer vertex are called \emph{proper cycles}.

We study now the vector space $\oF(G)$ of flow classes on $G$ 
introduced in Section~\ref{se:flows}. Recall that for this, 
we have to replace each edge of $G$ by the corresponding two directed edges.  
Correspondingly, we will consider $G$ as a directed graph.
Since any cycle~$c$ in the digraph $G$ defines a flow,
it also defines a flow class on $G$ that we denote by~$c$ as well.

\begin{definition}[Throughput]
For an edge $k \in E(\Delta)$ let $e \in E(G)$ denote the edge
pointing from the white vertex on $k$ into the upright triangle.
The throughput $\delta(k,f)$ w.r.t.\ a flow $f\in\oF(G)$ is defined as
$f(e)-f(-e)$.
\end{definition}

In this paper, it will be extremely helpful to have some graphical way 
of describing rhombi and throughputs.
We shall denote a rhombus $\varrho$ by the pictogram $\rhc$, 
even though $\varrho$ may lie in any of the three positions 
``$\rhc$'', ``%
\raisebox{-0.1cm}{%
\begin{tikzpictured}\draw[rhrhombidraw] (-4.619pt,8.0001pt) -- (4.619pt,8.0001pt) -- (9.238pt,16.0002pt) -- (0.0pt,16.0002pt) -- cycle;\fill[white] (0.0pt,0.0pt) circle (0.1pt);\end{tikzpictured}%
}%
'' or ``%
\raisebox{-0.1cm}{%
\begin{tikzpictured}\fill[white] (0.0pt,0.0pt) circle (0.1pt);\draw[rhrhombidraw] (-9.238pt,16.0002pt) -- (-4.619pt,8.0001pt) -- (4.619pt,8.0001pt) -- (0.0pt,16.0002pt) -- cycle;\end{tikzpictured}%
}%
''
obtained by rotating with a multiple of $60^\circ$. 
Let $\rhsc$ denote the edge~$k$ of~$\Delta$ given by the \emph{diagonal} of $\varrho$ 
connecting its two obtuse angles.
Then we denote by $\rhacM (f) := \delta(k,f)$ the throughput of $f$ through~$k$ 
(going into the upright hive triangle). 
Similarly, we define the throughput 
$ \rhacW (f) := - \delta(k,f)$. 
The advantage of this notation is that 
if the throughput is positive, then the flow goes 
in the direction of the arrow.
For instance, using the symbolic notation, 
we note the following consequence of the flow conservation laws: 
\begin{equation*}
 \rhaollW(f) + \rhaolrW(f) \ =\  \rhaoulW(f) + \rhaourW(f) .
\end{equation*}

\subsection{Hive flows}


All definitions up to this point were fairly standard.
We describe now the relation to the Littlewood-Richardson coefficients.

\begin{definition}\label{def:slack}
The \emph{slack} of the rhombus $\rhc$ with respect to $f\in \oF(G)$ is defined as 
$$
\s \rhc f \ :=\ \rhaoulW(f) + \rhaolrM(f).
$$
The rhombus $\rhc$ is called \emph{$f$\dash flat} if $\s \rhc f$ = 0.
\end{definition}

It is clear that $\oF(G)\to\R, f\mapsto \s \varrho f$ is a linear form.
Note also that by flow conservation,
the slack can be written in various different ways: 
$$
 \s \rhc f \ =\ \rhaoulW(f) +  \rhaolrM(f) \ = \  
 \rhaoulW(f) -  \rhaolrW(f) \ =\  \rhaollW(f) - \rhaourW(f) 
 \ = \ \rhaollW(f) + \rhaourM(f) .
$$

For calculating the slack values of rhombi,
the \emph{hexagon equality} described below will be useful.
The straightforward proof is omitted.
\begin{claim}[Hexagon equality]
\label{cla:BZ}
The union of two overlapping rhombi $\varrho_1$ and $\varrho_2$ forms a trapezoid.
Glueing together two such trapezoids $(\varrho_1,\varrho_2)$ and $(\varrho'_1,\varrho'_2)$ at their longer side,
we get a hexagon. 
We have $\s {\varrho_1} f + \s {\varrho_2} f = \s {\varrho'_1} f + \s {\varrho'_2} f$
for each flow $f\in \oF(G)$.
In pictorial notation, the hexagon equality can be succinctly expressed as 
$$
\s {
\begin{tikzpictured}\fill[rhrhombifill] (9.238pt,0.0pt) -- (4.619pt,8.0001pt) -- (9.238pt,16.0002pt) -- (13.857pt,8.0001pt) -- cycle;\fill (0.0pt,0.0pt) circle (0.4pt);\fill (9.238pt,0.0pt) circle (0.4pt);\fill (4.619pt,8.0001pt) circle (0.4pt);\fill (-4.619pt,8.0001pt) circle (0.4pt);\fill (13.857pt,8.0001pt) circle (0.4pt);\fill (0.0pt,16.0002pt) circle (0.4pt);\fill (9.238pt,16.0002pt) circle (0.4pt);\draw[rhrhombidraw] (0.0pt,16.0002pt) -- (9.238pt,0.0pt) ;\end{tikzpictured}
} f + \s {
\begin{tikzpictured}\fill[rhrhombifill] (13.857pt,8.0001pt) -- (9.238pt,16.0002pt) -- (0.0pt,16.0002pt) -- (4.619pt,8.0001pt) -- cycle;\fill (0.0pt,0.0pt) circle (0.4pt);\fill (9.238pt,0.0pt) circle (0.4pt);\fill (4.619pt,8.0001pt) circle (0.4pt);\fill (-4.619pt,8.0001pt) circle (0.4pt);\fill (13.857pt,8.0001pt) circle (0.4pt);\fill (0.0pt,16.0002pt) circle (0.4pt);\fill (9.238pt,16.0002pt) circle (0.4pt);\draw[rhrhombidraw] (9.238pt,0.0pt) -- (0.0pt,16.0002pt) ;\end{tikzpictured}
} f = \s {
\begin{tikzpictured}\fill[rhrhombifill] (0.0pt,16.0002pt) -- (-4.619pt,8.0001pt) -- (0.0pt,0.0pt) -- (4.619pt,8.0001pt) -- cycle;\fill (0.0pt,0.0pt) circle (0.4pt);\fill (9.238pt,0.0pt) circle (0.4pt);\fill (4.619pt,8.0001pt) circle (0.4pt);\fill (-4.619pt,8.0001pt) circle (0.4pt);\fill (13.857pt,8.0001pt) circle (0.4pt);\fill (0.0pt,16.0002pt) circle (0.4pt);\fill (9.238pt,16.0002pt) circle (0.4pt);\draw[rhrhombidraw] (0.0pt,16.0002pt) -- (9.238pt,0.0pt) ;\end{tikzpictured}
} f + \s {
\begin{tikzpictured}\fill[rhrhombifill] (-4.619pt,8.0001pt) -- (0.0pt,0.0pt) -- (9.238pt,0.0pt) -- (4.619pt,8.0001pt) -- cycle;\fill (0.0pt,0.0pt) circle (0.4pt);\fill (9.238pt,0.0pt) circle (0.4pt);\fill (4.619pt,8.0001pt) circle (0.4pt);\fill (-4.619pt,8.0001pt) circle (0.4pt);\fill (13.857pt,8.0001pt) circle (0.4pt);\fill (0.0pt,16.0002pt) circle (0.4pt);\fill (9.238pt,16.0002pt) circle (0.4pt);\draw[rhrhombidraw] (0.0pt,16.0002pt) -- (9.238pt,0.0pt) ;\end{tikzpictured}
} f.
$$
\end{claim}



\begin{definition}\label{def:hiveflow}
A flow $f \in \oF(G)$ is called a \emph{hive flow} iff for all rhombi~$\varrho$ we have $\s \varrho f \geq 0$.
\end{definition}

Note that the set of hive flows is a cone in $\oF(G)$. 
Figure~\ref{fig:degengraph} provides an example of a hive flow.
\begin{figure}[h]
\begin{center}
\scalebox{1.2}{
\begin{tikzpicture}\draw[rhrhombithickside] (60.0pt,51.96pt) -- (30.0pt,0.0pt);\draw[rhrhombithickside] (75.0pt,25.98pt) -- (15.0pt,25.98pt);\draw[rhrhombithickside] (60.0pt,0.0pt) -- (30.0pt,51.96pt);\draw[rhrhombithickside] (75.0pt,25.98pt) -- (60.0pt,0.0pt);\draw[rhrhombidraw] (0.0pt,0.0pt) -- (45.0pt,77.94pt) -- (90.0pt,0.0pt) -- cycle;\draw[rhrhombidraw] (30.0pt,51.96pt) -- (60.0pt,51.96pt) ;\draw[rhrhombidraw] (15.0pt,25.98pt) -- (75.0pt,25.98pt) ;\draw[rhrhombidraw] (30.0pt,0.0pt) -- (60.0pt,51.96pt) ;\draw[rhrhombidraw] (60.0pt,0.0pt) -- (75.0pt,25.98pt) ;\draw[rhrhombidraw] (60.0pt,0.0pt) -- (30.0pt,51.96pt) ;\draw[rhrhombidraw] (30.0pt,0.0pt) -- (15.0pt,25.98pt) ;\draw[butt cap-latex,line width=1pt,black] (75.0pt,-6.0pt) -- (75.0pt,9.0pt);\fill[white] (75.0pt,0.0pt) circle (2.5pt);\node at (75.0pt,0.0pt) {\tiny 5};\draw[butt cap-latex,line width=1pt,black] (45.0pt,-6.0pt) -- (45.0pt,9.0pt);\fill[white] (45.0pt,0.0pt) circle (2.5pt);\node at (45.0pt,0.0pt) {\tiny 2};\draw[butt cap-latex,line width=1pt,black] (72.696pt,9.99pt) -- (59.706pt,17.49pt);\fill[white] (67.5pt,12.99pt) circle (2.5pt);\node at (67.5pt,12.99pt) {\tiny 5};\draw[butt cap-latex,line width=1pt,black] (60.0pt,19.98pt) -- (60.0pt,34.98pt);\fill[white] (60.0pt,25.98pt) circle (2.5pt);\node at (60.0pt,25.98pt) {\tiny 4};\draw[butt cap-latex,line width=1pt,black] (45.0pt,45.96pt) -- (45.0pt,60.96pt);\fill[white] (45.0pt,51.96pt) circle (2.5pt);\node at (45.0pt,51.96pt) {\tiny 2};\draw[butt cap-latex,line width=1pt,black] (57.696pt,35.97pt) -- (44.706pt,43.47pt);\fill[white] (52.5pt,38.97pt) circle (2.5pt);\node at (52.5pt,38.97pt) {\tiny 6};\draw[butt cap-latex,line width=1pt,black] (42.696pt,9.99pt) -- (29.706pt,17.49pt);\fill[white] (37.5pt,12.99pt) circle (2.5pt);\node at (37.5pt,12.99pt) {\tiny 3};\draw[butt cap-latex,line width=1pt,black] (42.696pt,61.95pt) -- (29.706pt,69.45pt);\fill[white] (37.5pt,64.95pt) circle (2.5pt);\node at (37.5pt,64.95pt) {\tiny 6};\draw[butt cap-latex,line width=1pt,black] (27.696pt,35.97pt) -- (14.706pt,43.47pt);\fill[white] (22.5pt,38.97pt) circle (2.5pt);\node at (22.5pt,38.97pt) {\tiny 4};\draw[butt cap-latex,line width=1pt,black] (12.696pt,9.99pt) -- (-0.294pt,17.49pt);\fill[white] (7.5pt,12.99pt) circle (2.5pt);\node at (7.5pt,12.99pt) {\tiny 3};\draw[butt cap-latex,line width=1pt,black] (57.696pt,67.95pt) -- (44.706pt,60.45pt);\fill[white] (52.5pt,64.95pt) circle (2.5pt);\node at (52.5pt,64.95pt) {\tiny 4};\draw[butt cap-latex,line width=1pt,black] (72.696pt,41.97pt) -- (59.706pt,34.47pt);\fill[white] (67.5pt,38.97pt) circle (2.5pt);\node at (67.5pt,38.97pt) {\tiny 2};\draw[butt cap-latex,line width=1pt,black] (42.696pt,41.97pt) -- (29.706pt,34.47pt);\fill[white] (37.5pt,38.97pt) circle (2.5pt);\node at (37.5pt,38.97pt) {\tiny 4};\draw[butt cap-latex,line width=1pt,black] (57.696pt,15.99pt) -- (44.706pt,8.49pt);\fill[white] (52.5pt,12.99pt) circle (2.5pt);\node at (52.5pt,12.99pt) {\tiny 1};\draw[butt cap-latex,line width=1pt,black] (27.696pt,15.99pt) -- (14.706pt,8.49pt);\fill[white] (22.5pt,12.99pt) circle (2.5pt);\node at (22.5pt,12.99pt) {\tiny 3};\end{tikzpicture}
}
    \caption{A hive flow for $n=3$, 
                   $\la =  (4,2,0)$, 
                   $\mu=(5,2,0)$, 
                  $\nu=(6,4,3)$.
                   The numbers give the throughputs through edges of $\Delta$ in the directions of the arrows. 
                   The diagonals of rhombi with positive slack are drawn thicker.
                   In this example, the slack of the rhombi ranges from 0 to 2.
                   }
    \nopar\label{fig:degengraph}
\end{center}
\end{figure}
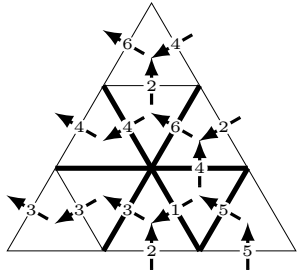


We now fix the throughputs at the border of~$\Delta$,
depending on a chosen triple $\la,\mu,\nu\in\N^n$ of partitions
satisfying $|\nu|=|\la|+|\mu|$:
To the $i$th border edge~$k$ of~$\Delta$ on the right border of~$\Delta$, counted from top to bottom, 
we assign the \emph{fixed throughput} $\bar\delta(k):=\la_i$, see Figure~\ref{fig:triangulararraywithflow}.
Further, we set $\bar\delta(k):= \mu_i$ for the $i$th edge~$k$ on the bottom border of~$\Delta$, counted from right to left.
Finally, we set $\bar\delta(k'):= -\nu_i$ for the $i$th edge~$k'$ on the left border of~$\Delta$, counted from top to bottom.
Recall that $\delta(k,f)$ denotes the throughput of a flow $f$ into $\Delta$, 
while $-\delta(k',f)$ denotes the throughput of~$f$ out of $\Delta$.
\begin{definition}\label{def:hiveflowpolytope}
Let $\la,\mu,\nu\in\N^n$ be a triple of partitions satisfying $|\nu|=|\la|+|\mu|$.
The \emph{hive flow polytope} $P:=P(\la,\mu,\nu) \subseteq \oF(G)$ 
is defined to be the set of hive flows $f\in\oF(G)$ satisfying 
$
 \delta(k,f) = \bar \delta(k)
$
for all border edges~$k$ of $\Delta$.
We also set $P_\Z := P \cap \oF(G)_\Z$.
\end{definition}

The following description of the Littlewood-Richardson coefficient is heavily based on \cite{knta:99}, see~\cite{bi:12}.
\begin{proposition}[{\protect\cite[Prop.~2.7]{bi:12}}]
\label{pro:flowdescription}
The Littlewood-Richardson coefficient $\LRC \la \mu \nu$ equals the number of integral hive flows in $P(\la,\mu,\nu)$,
i.e., $\LRC \la \mu \nu = |P(\la,\mu,\nu)_\Z|$.
\end{proposition}

\subsection{The LR graph and secure cycles}
\label{subsec:seccycles}
We now focus on cycles and their induced flow,
starting out with the following fairly simple claim.
\begin{claim}
\label{cla:smallabsoluteslack}
 For each cycle $c$ and each rhombus $\varrho$ we have $\s \varrho c \in \{-2,-1,0,1,2\}$.
\end{claim}
\begin{proof}
Fix a rhombus $\rhc$. The claim follows from $\rhaoulW(c) \in \{-1,0,1\}$ and $\rhaolrM(c) \in \{-1,0,1\}$.
\end{proof}
For the exact slack calculation we introduce the following notation,
in which we write paths in a pictorial notation.
As always, $\rhc$ can appear in any situation rotated by a multiple of $60^\circ$.
\begin{definition}\label{def:slack-contr}
The sets of paths, interpreted as subsets of $E(G)$,  
$$
 \p_+(\rhc) := \{\rhpourMr, \rhpollWr, \rhpourMll, \rhpollWll\},\quad
 \p_-(\rhc) := \{\rhpoulMl, \rhpolrWl, \rhpoulMrr, \rhpolrWrr\},\mbox{ and } 
 \p_0(\rhc) := \{\rhpoulMrl,\rhpourMlr,\rhpollWlr,\rhpolrWrl\}
$$
are called the sets of  of 
\emph{positive},
\emph{negative}, and 
\emph{neutral slack contributions} of 
the rhombus $\rhc$, respectively. 
\end{definition}
The following Proposition~\ref{pro:slackcalc} gives a method to determine the slack.
\begin{proposition}[{\protect\cite[Obs.~3.3]{bi:12}}]
\label{pro:slackcalc}
Let $c$ be a cycle in $G$, let $\varrho$ be a rhombus, and let $E_\varrho$ denote the set of edges of $G$ contained in a rhombus~$\varrho$. 
Then $c\cap E_\varrho$ is either empty, or it is a union of one or two slack contributions~$q$.
The slack $\s \varrho c$ is obtained by adding $1$, $0$, or $-1$ 
over the contributions~$q$ contained in $c$, according to whether~$q$ is   
is  positive, negative, or neutral.
\end{proposition}

\begin{definition}
\label{def:fhivepres}
Fix a hive flow $f \in P$.
A flow $d \in \oF(G)$ is called \emph{$f$\dash hive preserving},
if there exists $\varepsilon>0$ such that $f+\varepsilon d \in P$.
We call a cycle $f$\dash hive preserving if its induced flow is $f$\dash hive preserving.
\end{definition}
We can easily check whether a cycle is $f$\dash hive preserving using the following lemma.
\begin{lemma}
\label{lem:fhivepresiffnoneg}
 A cycle $c$ is $f$\dash hive preserving, iff $c$ does not use any negative contribution in $f$\dash flat rhombi.
\end{lemma}
\begin{proof}
According to Proposition~\ref{pro:slackcalc},
if $c$ does not use any negative contribution in $f$\dash flat rhombi,
then $f+\varepsilon c \in P$ for $\varepsilon$ small enough.
Conversely, assume that $c$ uses a negative contribution in an $f$\dash flat rhombus $\rhc$.
If $c$ uses $\rhpoulMrr$ or $\rhpolrWrr$, then $c$ uses no other contribution in $\rhc$
and hence $\s \rhc {f+\varepsilon c} = -\varepsilon < 0$.
If $c$ uses $\rhpoulMl$, then by a topological argument, the only other contribution
that $c$ can use is $\rhpolrWl$, which is also negative.
As before we conclude $\s \rhc {f+\varepsilon c} < 0$.
The argument for $c$ using $\rhpolrWl$ is analogous.
\end{proof}

\begin{definition}
We say that $f,g\in P_\Z$ are \emph{neighbors} iff $g-f$ is induced by a cycle in~$G$.
The resulting graph with the set of vertices $P_\Z$ is also denoted by~$P_\Z$
and it is called the \emph{Littlewood-Richardson graph} or \emph{LR graph} for short.
The neighborhood of $f$ is denoted with~$\Gamma(f)$.
\end{definition}
Clearly, the neighborhood relation is symmetric.
In the next Definition~\ref{def:secure} and in Proposition~\ref{pro:secureneighbors}
we focus on which cycles serve as edges in the LR graph~$P_\Z$.
\begin{definition}
\label{def:secure}
 A rhombus $\varrho$ is called \emph{nearly $f$\dash flat}, if $\s \varrho f = 1$.
 We call a cycle~$c$ $f$\dash \emph{secure}, if $c$ is $f$\dash hive preserving
 and if additionally $c$ does not use both negative contributions $\rhpoulMl$ and $\rhpolrWl$
 at the acute angles of any nearly $f$\dash flat rhombus~$\rhc$.\end{definition}
\begin{proposition}[{\protect\cite[Prop.~3.8]{bi:12}}]
\label{pro:secureneighbors}
Assume $f\in P_\Z$. 
If $g\in P_\Z$ is a neighbor of $f$, then $g-f$ is an $f$-secure cycle.
Conversely, if $c$ is a proper $f$-secure cycle, then $f+c\in P_\Z$ is a neighbor of~$f$.
\end{proposition}

One main result in \cite{bi:12} is the following Theorem~\ref{thm:connectedness}.
\begin{theorem}[{Connectedness Theorem, \protect\cite[Thm.~3.12]{bi:12}}]
\label{thm:connectedness}
The LR graph $P_\Z$ is connected.
\end{theorem}

\section{Enumerating hive flows}
\label{sec:enumeration}
In this section we give a nontechnical overview of the enumeration algorithms.

The following theorem states that (a superset of) the neighborhood $\Gamma(f)$ can be efficiently enumerated.
We postpone the proof of Theorem~\ref{thm:neigh} to Section~\ref{sec:neighborhoodgen}.
\begin{theorem}[Neighbourhood generator]
\label{thm:neigh}
There exists an algorithm \mbox{\textsc{NeighGen}} which on input $f \in P_\Z$ outputs the elements of a set
$\Gamm f \subseteq P_\Z$ one by one such that $\Gamma(f) \subseteq \Gamm f$.
The computation of the first $k$ elements takes time $\gO\big(k \cdot \poly(n)\big)$.
\end{theorem}

We define the directed graph $\tilde P_\Z$ to be the graph with vertex set $V(P_\Z)$ and
(possibly asymmetric) neighborhood function $\tilde \Gamma$ as given by Theorem~\ref{thm:neigh}.
Since $P_\Z$ is connected by Theorem~\ref{thm:connectedness}, $\tilde P_\Z$ is strongly connected.
Therefore,
breadth-first-search on $\tilde P_\Z$ started at any hive flow in $P_\Z$ visits all flows in $P_\Z$.
A variant of this breadth-first-search is realized in the following Algorithm~\ref{alg:lrthreshold},
which gets an additional threshold parameter $t$ such that Algorithm~\ref{alg:lrthreshold} visits at most $t$ flows.
\begin{algorithm}[H]
\caption{\algocaptionwithoutparam{LR-Threshold}}
\nopar\label{alg:lrthreshold}
\begin{algorithmic}[1]
\REQUIRE Partitions $\la$, $\mu$, $\nu$ with $|\la|+|\mu|=|\nu|$; $t \in \N_{>0}$
\ENSURE \True, if $\LRC \la \mu \nu \geq t$. \False otherwise.
\STATE If $P(\la,\mu,\nu) = \emptyset$, \Return \False.\nopar\label{alg:firstreturnfalse}
\STATE Compute an integral hive flow $f \in P(\la,\mu,\nu)_\Z$.
\STATE Initially, set $S \leftarrow \{f\}$, $T \leftarrow \emptyset$.\nopar\label{alg:invariants}
\WHILE{$T \subsetneq S$}
\STATE Choose an $f \in S \setminus T$.
\FOR{each $g \in \Gamm f$ generated one by one by \textsc{NeighGen} via Theorem~\ref{thm:neigh}}
 \STATE $S \leftarrow S \cup \{g\}$.
 \STATE If $|S| \geq t$, then \Return \True.\nopar\label{alg:neverreturn}
\ENDFOR
\STATE $T \leftarrow T \cup \{f\}$.
\ENDWHILE
\STATE \Return \False.\nopar\label{alg:returnfalse}
\end{algorithmic}
\end{algorithm}
The first two lines of Algorithm~\ref{alg:lrthreshold} deal with computing a hive flow $f \in P_\Z$ if there exists one.
This can be done in time strongly polynomial in $n$ using Tardos' algorithm~\cite{Tar:86, gll:93, knta:99} as stated in~\cite{GCT3} and~\cite{deloera:06}.
We can also use the combinatorial algorithm presented in \cite{bi:12} for this purpose, which is especially designed for this problem, but note that although it has a much smaller exponent in the running time, its running time depends on the bitsize of the input partitions.
Here, the \emph{running time} is defined to be the number of additions, multiplications and comparisons.
In practice, the algorithm in \cite{bi:12} may be the better choice, but in this paper we focus on algorithms
whose running time does not depend on the input bitsize.
If Tardos' algorithm is used as a subalgorithm in Algorithm~\ref{alg:lrthreshold},
then this is the case and hence we choose this option.
\begin{theorem}
Given partitions $\la$, $\mu$, $\nu$ with $|\la|+|\mu|=|\nu|$ and a natural number $t \geq 1$,
then Algorithm~\ref{alg:lrthreshold} decides $\LRC \la \mu \nu \geq t$
in time $\gO\big(t^2 \cdot \poly(n)\big)$.
\label{thm:decide}
\end{theorem}
\begin{proof}
Recall that according to Proposition~\ref{pro:flowdescription} we have $|P_\Z|=\LRC \la \mu \nu$.
Now observe that, starting after line~\ref{alg:invariants}, Algorithm~\ref{alg:lrthreshold} preserves the three invariants $T \subseteq S \subseteq P_\Z$,
$|S|\leq t$ and  $\forall f \in T: \Gamm f \subseteq S$.

If the algorithm returns \True, then $|S|\geq t$.
As $S \subseteq P_\Z$ and $|P_\Z|=\LRC \la \mu \nu$, we have $\LRC \la \mu \nu \geq t$.

If the algorithm returns \False, then $|S| < t$ and $S = T$.
Moreover, $\Gamm f \subseteq S$ for all $f \in S$.
Since the digraph $\tilde P_\Z$ is strongly connected, it follows that $P_\Z = S$.
Therefore we have $\LRC \la \mu \nu = |P_\Z| = |S| < t$.

We have shown that the algorithm works correctly.

Now we analyze its running time.
Recall that the first two lines of Algorithm~\ref{alg:lrthreshold} run in time $\poly(n)$,
because of Tardos' algorithm.
The outer loop runs at most $t$ times, because in each iteration, $|T|$ increases and $|T| \leq |S| \leq t$.
If in the inner loop we have $|\Gamm f|< t$, then the inner loop runs for at most $t-1$ iterations
and hence $\Gamm f$ can be generated in time $\gO\big(t \cdot \poly(n)\big)$ via Theorem~\ref{thm:neigh}.
If in the inner loop we have $|\Gamm f|\geq t$, then after $t$ iterations we have $|S|\geq t$ and the algorithm returns immediately.
The first $t$ elements of $\Gamm f$ can be generated via Theorem~\ref{thm:neigh} in time $\gO\big(t \cdot \poly(n)\big)$.
Therefore we get an overall running time of $\gO\big(t^2 \cdot \poly(n)\big)$.
\end{proof}

\begin{theorem}
\label{thm:compute}
Given partitions $|\la|+|\mu|=|\nu|$. Then $\LRC \la \mu \nu$ can be computed
in time $\gO\big((\LRC \la \mu \nu)^2 \cdot \poly(n)\big)$ by a variant of Algorithm~\ref{alg:lrthreshold}.
\end{theorem}
\begin{proof}
Use Algorithm~\ref{alg:lrthreshold} with the input $t = \infty$ as a formal symbol,
but instead of returning \False in line~\ref{alg:firstreturnfalse}, return 0 and instead of returning \False in line~\ref{alg:returnfalse}, return $|S|$.
Note that the algorithm never returns \True, because ``$|S|>\infty$'' in line~\ref{alg:neverreturn} is always false.
If the algorithm terminates, then $P_\Z=S$ and thus the algorithm works correctly.
Note that if started with $t=\infty$ the algorithm behaves exactly as if started with $t = \LRC \la \mu \nu+1$.
Thus it runs in time $\gO\big((\LRC \la \mu \nu)^2 \cdot \poly(n)\big)$.
\end{proof}

\section{The residual digraph $R_f$}
\label{sec:residual}
We would like to have a direct algorithm that prints out the elements of the neighborhood $\Gamma(f)$.
A naive approach would list all cycles $c$ and reject those with $f+c \notin P$.
But we cannot control this algorithm's running time when there are many rejections.
Note that there are exponentially many cycles!
The solution is to a priori generate only those cycles $c$ with $f+c \in P$.
According to Lemma~\ref{lem:fhivepresiffnoneg}, these $c$ use no negative contributions in $f$\dash flat rhombi.
So we now introduce the digraph $\resf f$ in which the possibility of using
negative contributions in $f$\dash flat rhombi is eliminated.
The digraph $\resf f$ will arise as a subgraph of the following digraph $\res$.
\begin{definition}[Digraph $\res$]
A \emph{turn} is defined to be a path in $G$ consisting of two edges,
starting at a white vertex, using a black vertex of a hive triangle, and ending at a different white vertex.
The {\em digraph $\res$} has as vertices the turns, 
henceforth called \emph{turnvertices}.
The edges of $\res$ are ordered pairs of turns that can be concatenated to a path in $G$,
henceforth called \emph{turnedges}.
\end{definition}
Note that there are six turnvertices in each hive triangle:
\begin{tikzpicture}\draw[rhrhombidraw] (0.0pt,0.0pt) -- (9.238pt,0.0pt) -- (4.619pt,8.0001pt) -- cycle;\draw[thin,-my] (2.3095pt,4.0001pt) arc (240:300:4.619pt);\end{tikzpicture}%
,
\begin{tikzpicture}\draw[rhrhombidraw] (0.0pt,0.0pt) -- (9.238pt,0.0pt) -- (4.619pt,8.0001pt) -- cycle;\draw[thin,-my] (2.3095pt,4.0001pt) arc (60:0:4.619pt);\end{tikzpicture}%
,
\begin{tikzpicture}\draw[rhrhombidraw] (0.0pt,0.0pt) -- (9.238pt,0.0pt) -- (4.619pt,8.0001pt) -- cycle;\draw[thin,-my] (4.619pt,0.0pt) arc (0:60:4.619pt);\end{tikzpicture}%
,
\begin{tikzpicture}\draw[rhrhombidraw] (0.0pt,0.0pt) -- (9.238pt,0.0pt) -- (4.619pt,8.0001pt) -- cycle;\draw[thin,-my] (4.619pt,0.0pt) arc (180:120:4.619pt);\end{tikzpicture}%
,
\begin{tikzpicture}\draw[rhrhombidraw] (0.0pt,0.0pt) -- (9.238pt,0.0pt) -- (4.619pt,8.0001pt) -- cycle;\draw[thin,-my] (6.9285pt,4.0001pt) arc (120:180:4.619pt);\end{tikzpicture}%
, and
\begin{tikzpicture}\draw[rhrhombidraw] (0.0pt,0.0pt) -- (9.238pt,0.0pt) -- (4.619pt,8.0001pt) -- cycle;\draw[thin,-my] (6.9285pt,4.0001pt) arc (300:240:4.619pt);\end{tikzpicture}%
.
Pictorially, we can write a turnedge like $\rhpoulMrr:=\big(\rhpoulMr,\rhpcWr\big)$.
Note that a turnedge corresponds to a path in $G$ of length 4.
Each rhombus $\rhc$ contains exactly the eight turnedges $\rhpoulMrr, \rhpoulMrl, \rhpourMll, \rhpourMlr, \rhpollWll, \rhpollWlr, \rhpolrWrr, \rhpolrWrl$.
Paths on $\res$ are called \emph{turnpaths} and cycles on $\res$ are called \emph{turncycles}.
Note that a turnpath can for example use both the turnedge $\rhpoulMrl$ and the turnedge $\rhpourMlr$ in a rhombus $\rhc$,
because both turnedges have no common turnvertex.
We denote by $\pstart p$ the first turnvertex of a turnpath $p$ and
by $\pend p$ its last turnvertex.

By using turnvertices and turnedges to define $\res$
and by focusing on the more complicated graph structure of $\res$ instead of the easy structure of $G$,
we now have the possibility to \emph{delete turnedges},
which is done in the next definition.
\begin{definition}[Digraph $\resf f$]
\label{def:resf}
Let $f \in P_\Z$.
We define the digraph $\resf f$ by deleting from $\res$ the turnvertices $\rhpoulMl$ and $\rhpolrWl$
and the turnedges $\rhpoulMrr$ and $\rhpolrWrr$ for each $f$\dash flat rhombus~$\rhc$.
\end{definition}
Note that the deleted parts correspond exactly to negative slack-contributions in $f$-flat rhombi (cp.~Definition~\ref{def:slack-contr}).

\begin{definition}[Turncycles]
We denote with $C(\res)$ the set of turncycles on $\res$,
and with $C(\resf f)$ the set of turncycles on $\resf f$.
\end{definition}

\begin{definition}[Throughput]
\label{def:newthroughput}
Let $c \in C(\res)$ and let $e \in E(\Delta)$.
For a turnvertex $v \in V(\res)$ we write $c(v)=1$ if $v \in c$ and $c(v)=0$ otherwise.
If $e$ is a diagonal of a rhombus, then we can write $e = \rhsc$
and we define the throughput as
$$
\rhacW(c) := c(\rhpoulMr) + c(\rhpourMl) - c(\rhpcMl) - c(\rhpcMr) \in \{-2,-1,0,1,2\}.
$$
If $e$ is a border edge of $\Delta$, we define the throughput analogously.
\end{definition}
\begin{remark}
Note that the throughput of cycles $c' \in C(G)$ can only range from $-1$ to $1$,
whereas for turncycles $c \in C(\res)$ the throughput can range from $-2$ to $2$.
The larger throughput range makes things complicated and in fact
it is the price we pay for being able to delete single turnedges in Definition~\ref{def:resf}.
\end{remark}


To a turncycle we can associate a flow on~$G$ as follows.

\begin{definition}
The map $\comb:C(\res)\rightarrow \oF(G)$ is the unique linear map preserving the throughput:
For $c \in C(\res)$ we define $\comb(c)$ such that $\rhacW\big(\comb(c)\big) = \rhacW(c)$
for all edges $\rhsc \in E(\Delta)$.
\end{definition}
Note that $\comb$ is well-defined, since $\rhacW(c)+\rhaoulW(c)+\rhaourW(c)=0$, see~\cite[Sec.~2.2]{bi:12}.

We define the slack w.r.t.\ turncycles via $\comb$ as follows:
$\s \rhc c := \s \rhc {\comb(c)} = \rhaoulW\big(\comb(c)\big) + \rhaolrM\big(\comb(c)\big)\in \{-4,-3,\ldots,3,4\}$.

We can determine the slack of rhombi for turncycles similarly to
determining the slack of rhombi for cycles in Proposition~\ref{pro:slackcalc}, as we see in the
following Claim~\ref{cla:slackforturncycles}.
Recall Definition~\ref{def:slack-contr} and interpret turnedges and turnvertices that
start and end at sides of a rhombus $\varrho$ as slack contributions in $\varrho$.
\begin{claim}[{\protect\cite[Lemma 4.4]{bi:12}}]
\label{cla:slackforturncycles}
Let $c$ be a turncycle and $\varrho$ be a rhombus.
The slack $\s \varrho c$ is obtained by adding $1$, $0$, or $-1$
over the turnedges $q$ used by~$c$ in~$\varrho$, according to whether~$q$
is positive, negative, or neutral,
and by further adding
$1$ or $-1$
over the remaining turnvertices $q'$ used by $c$ in $\varrho$, according to whether~$q'$
is positive ($\rhpourMr$ or $\rhpollWr$) or negative ($\rhpoulMl$ or $\rhpolrWl$).
\end{claim}

\begin{lemma}
\label{lem:direction}
Let $f\in P$. Each turncycle $c$ in $\resf f$ is $f$\dash hive preserving.
\end{lemma}
\begin{proof}
 By construction of $\resf f$ and Claim~\ref{cla:slackforturncycles},
 see the discussion in \cite{bi:12} after Lemma 4.7 there.
\end{proof}

There is a canonical injective map from the set of proper cycles~$c$ on~$G$
to the set of turncycles $c'$ on~$\res$: consecutive turns in~$c$ correspond to turnedges in~$c'$.
A turncycle $c'$ in the image of this map is \emph{ordinary}, which means the following:
$c'$ uses only a single turnvertex in each hive triangle.
This is the desired behaviour of turncycles when we want to simulate cycles on~$G$.
\begin{definition}
Turnpaths and turncycles are called \emph{ordinary}, if they
use at most one turnvertex in each hive triangle.
\end{definition}
Obviously, there is a bijection between the set of proper cycles on $G$
and the set of ordinary turncycles on $\res$.

The following definition of secure turnpaths is related to the Definition~\ref{def:secure} of secure cycles.
\begin{definition}[Secure turnpaths]
A turnpath $p$ on $\resf f$ is called \emph{$f$\dash secure},
if $p$ is ordinary and if additionally
$p$ does not use both counterclockwise turnvertices $\rhpoulMl$ and $\rhpolrWl$ 
at the acute angles of any nearly $f$\dash flat rhombus $\rhc$.
We define $f$\dash secure turncycles analogously.\end{definition}
Since $f$\dash secure turncycles are ordinary,
there is a bijection between $f$\dash secure proper cycles and $f$\dash secure turncycles.
To prove Theorem~\ref{thm:neigh},
we want to list all $f$\dash secure proper cycles (cf.\ Prop.~\ref{pro:secureneighbors}).
We have seen that we may as well list the $f$\dash secure turncycles.

\section{The Neighbourhood Generator}
\label{sec:neighborhoodgen}
This section is devoted to the proof of Theorem~\ref{thm:neigh} by describing and analyzing
the algorithm \textsc{NeighGen}.
This algorithm is inspired by the binary partitioning method used in \cite{fuma:94}.

Given $f \in P_\Z$, \textsc{NeighGen} prints out the elements of a set $\Gamm f$ with $\Gamma(f) \subseteq \Gamm f \subseteq P_\Z$.
Note that we would like to have a direct algorithm that prints the elements of $\Gamma(f)$,
but we do not know how to do this efficiently.

Although we can treat $f$\dash hive preserving turncycles algorithmically,
there are problems when it comes to $f$\dash secure turncycles.
In fact we do not know how to solve the following crucial
Secure Extension Problem~\ref{prob:secure}.
\begin{problem}[Secure extension problem]\label{prob:secure}
Given $f \in P_\Z$ and an $f$\dash secure turnpath~$p$,
decide in time $\poly(n)$ whether there exists an $f$\dash secure turncycle~$c$ containing~$p$ or not.
\end{problem}
If in Problem~\ref{prob:secure} an extension $c$ exists for a given~$p$, then we call~$p$ \emph{$f$\dash securely extendable}.

The usefulness of having a solution to Problem~\ref{prob:secure} will be made clear in the next subsection,
where we introduce an algorithm \textsc{NeighGen'} that proves Theorem~\ref{thm:neigh} under the assumption that
Problem~\ref{prob:secure} is has a positive solution.

\subsection{A first approach}
\label{subsec:hypo}
Assume that $\Algo$ is an algorithm that on input $(f,p)$ with $f \in P_\Z$ and $p$ an $f$\dash secure turnpath in $\resf f$ returns whether $p$ is $f$\dash securely extendable or not.
Notationally,
\begin{equation}
 \tag{\dag}
 \Algo(f,p)=\begin{cases}
             \True\text{, if $p$ is $f$\dash securely extendable}\\
             \False \text{, otherwise.}
            \end{cases}
\end{equation}
We denote by $T(\Algo,n)$ the worst case running time of $\Algo(f,p)$
over all partitions $\la,\mu,\nu$ into $n$ parts, all $f \in P(\la,\mu,\nu)_\Z$ and all $f$\dash secure turnpaths~$p$ in $\resf f$.
If $\Algo$ solves Problem~\ref{prob:secure}, then $T(\Algo,n)$ is polynomially bounded in~$n$ ---
but remember that we do not know of such~$\Algo$.

The algorithms presented in this subsection use $\Algo$ as a subroutine
and hence they are only polynomial time algorithms if $T(\Algo,n)$ is polynomially bounded.
In fact this subsection is only meant to prepare the reader for the more complicated approach used in the Subsection~\ref{subsec:bypassing},
where $\Algo$ is modified in a way such that polynomial running time is achieved.

The main subalgorithm of this subsection is Algorithm~\ref{alg:findneighHypoth} below.
\begin{algorithm}[h]
\caption{\algocaptionwithoutparam{FindNeighWithBlackBox}}
\nopar\label{alg:findneighHypoth}
\begin{algorithmic}[1]
\REQUIRE $f \in P_\Z$; $p$ an $f$-securely extendable turnpath on $\resf f$;
$\Algo$ as in~$(\dag)$
\ENSURE Prints all integral flows $f+c \in P_\Z$, where $c \in C(G)$ is a cycle that contains~$p$.
Prints at least one element.
\IF{$p$ is not just a turnpath, but a turncycle}
\STATE $\print{f+\comb(p)}$ and \Return.\nopar\label{alg:firstprintHypoth}
\ENDIF
\FOR{\textbf{both} turnedges $e:=(\pend p,z) \in E(\resf f)$\nopar\label{alg:appendHypoth}}
\STATE Concatenate $p' \leftarrow pe$.
\STATE If $p'$ is not $f$\dash secure, \textbf{continue} with the next $e$.\nopar\label{alg:ifnotsecure}
\IF{$\Algo(f,p')$\nopar\label{alg:searchpathHypoth}}
 \STATE Recursively call {\scshape FindNeighWithBlackBox}$(f, p', \Algo)$.
\ENDIF
\ENDFOR
\end{algorithmic}
\end{algorithm}
Note that the statement \mbox{\textbf{for both}} in line~\ref{alg:appendHypoth} means \emph{for all},
as there are at most two turnedges $e$ in $\resf f$ starting at $\pend p$.
\begin{lemma}\label{lem:hypoiscorrectI}
Algorithm~\ref{alg:findneighHypoth} works according to its output specification.
\end{lemma}
\begin{proof}
Since the input $p$ is $f$\dash securely extendable, there exists at least one turnedge $e=(\pend p,z)$
such that $p'=pe$ is $f$\dash secure in line~\ref{alg:ifnotsecure} and $f$\dash securely extendable in line~\ref{alg:searchpathHypoth}.
Hence Algorithm~\ref{alg:findneighHypoth} prints at least one element or calls itself recursively.
The lemma follows now easily by induction on the number of turns in~$G$ not used by~$p$.
\end{proof}
We will see that it is crucial for the running time that for each call of
Algorithm~\ref{alg:findneighHypoth} we can ensure that an element is printed.
\begin{lemma}\label{lem:hypoiscorrectII}
Let $f \in P_\Z$.
On input $f \in P_\Z$, Algorithm~\ref{alg:findneighHypoth} prints out distinct elements.
The first $k$ elements are printed in time $\gO(k n^2 T(\Algo,n))$.
\end{lemma}
\begin{proof}
Algorithm~\ref{alg:findneighHypoth} traverses a binary recursion tree of depth at most $|E(\resf f)|=\gO(n^2)$ with depth-first-search.
The time needed at each recursion tree node is $\gO(T(\Algo))$,
if the implementation is done in a reasonable manner:
Checking if a turnpath is a turncycle can be done in time $\gO(1)$
and testing $p'$ for $f$\dash security under the assumption that $p$ was $f$\dash secure
can also be in time $\gO(1)$.
Thus the time the algorithm spends between two leafs is at most $\gO(n^2 T(\Algo))$.
The lemma is proved by the fact that at each leaf a distinct element is printed.
\end{proof}
We can define an algorithm {\scshape NeighGen'} as required for Theorem~\ref{thm:neigh} (besides polynomial running time) as follows:

\begin{samepage}
\begin{algorithmic}[1]
\STATE Let $f \in P_\Z$ be an input.
\FOR{all turnedges $p \in E(\resf f)$}
 \STATE Compute $\Algo(f,p)$.
 \IF{$\Algo(f,p)=\True$}
  \STATE Call Algorithm~\ref{alg:findneighHypoth} on $(f,p,\Algo)$.
 \ENDIF
\ENDFOR
\end{algorithmic}
\end{samepage}
Lemma~\ref{lem:hypoiscorrectI} ensures that $\Gamma(f)$ is printed by {\scshape NeighGen'}.
We now analyze the running time of {\scshape NeighGen'}.

Since each element in $\Gamma(f)$ is printed at most once during each of the $\gO(n^2)$ calls of Algorithm~\ref{alg:findneighHypoth}
and each call of Algorithm~\ref{alg:findneighHypoth} prints pairwise distinct elements,
it follows that each element is printed by {\scshape NeighGen'} at most $\gO(n^2)$ times.
Hence, according to Lemma~\ref{lem:hypoiscorrectII}, the first $k$ elements
are printed in time $\gO(k \cdot n^4 \cdot T(\Algo,n))$.

Since the existence of $\Algo$ was hypothetical,
we have to bypass Problem~\ref{prob:secure}.
This is achieved in the next subsection.

\subsection{Bypassing the secure extension problem}
\label{subsec:bypassing}
We need a polynomial time algorithm that solves a problem similar to Problem~\ref{prob:secure}.
A first approach for this is the following (which will fail for several reasons explained below):
Instead of extending $p$ to an $f$\dash secure cycle, we compute a \emph{trivial extension~$q$ of~$p$},
which is a shortest turnpath $q$ in $\resf f$
starting at $\pend p$ and ending at $\pstart p$.
A turncycle~$c$ containing~$p$ can then be obtained as the concatenation $c=pq$.
But $c$ might not be secure and might not even be ordinary
and in the worst case we could have $f+\comb(pq) \notin P$.
It will be crucial in the following to find $q$ such that $f+\comb(pq) \in P$,
so in the upcoming examples we have a look at the difficulties that may arise
for a trivial extension~$q$ and how we can fix them.

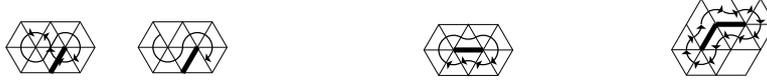
\begin{figure}[h]
\begin{center}
      \subfigure[Example~(a). Since the intersection of the curves of $p$ and $q$ decreases readability,
      in the right picture, only $p$ is drawn.]
        {
\hspace{1cm}
\scalebox{1.2}{
\begin{tikzpicture}\draw[rhrhombidraw] (0.0pt,0.0pt) -- (9.238pt,16.0002pt) -- (18.476pt,0.0pt) ;\draw[rhrhombidraw] (18.476pt,16.0002pt) -- (9.238pt,0.0pt) -- (0.0pt,16.0002pt) ;\draw[rhrhombidraw] (-4.619pt,8.0001pt) -- (23.095pt,8.0001pt) ;\draw[rhrhombithickside] (13.857pt,8.0001pt) -- (9.238pt,0.0pt);\draw[thin,-my] (16.1665pt,4.0001pt) arc (300:240:4.619pt);\draw[thin,-my] (11.5475pt,4.0001pt) arc (240:180:4.619pt);\draw[thin,-my] (9.238pt,8.0001pt) arc (0:60:4.619pt);\draw[thin,-my] (6.9285pt,12.0002pt) arc (60:120:4.619pt);\draw[thin] (2.3095pt,12.0002pt) arc (120:180:4.619pt) arc (180:240:4.619pt) arc (240:300:4.619pt);\draw[thin] (6.9285pt,4.0001pt) --  (11.5475pt,12.0002pt);;\draw[thin,-my] (11.5475pt,12.0002pt) arc (120:60:4.619pt) arc (60:0:4.619pt) arc (0:-60:4.619pt);\draw[rhrhombidraw] (-4.619pt,8.0001pt) -- (0.0pt,16.0002pt) -- (18.476pt,16.0002pt) -- (23.095pt,8.0001pt) -- (18.476pt,0.0pt) -- (0.0pt,0.0pt) -- cycle;\end{tikzpicture}
\quad
\begin{tikzpicture}\draw[rhrhombidraw] (0.0pt,0.0pt) -- (9.238pt,16.0002pt) -- (18.476pt,0.0pt) ;\draw[rhrhombidraw] (18.476pt,16.0002pt) -- (9.238pt,0.0pt) -- (0.0pt,16.0002pt) ;\draw[rhrhombidraw] (-4.619pt,8.0001pt) -- (23.095pt,8.0001pt) ;\draw[rhrhombithickside] (13.857pt,8.0001pt) -- (9.238pt,0.0pt);\draw[rhrhombidraw] (-4.619pt,8.0001pt) -- (0.0pt,16.0002pt) -- (18.476pt,16.0002pt) -- (23.095pt,8.0001pt) -- (18.476pt,0.0pt) -- (0.0pt,0.0pt) -- cycle;\draw[thin,-my] (2.3095pt,12.0002pt) arc (120:180:4.619pt) arc (180:240:4.619pt) arc (240:300:4.619pt) arc (300:360:4.619pt) arc (180:120:4.619pt) arc (120:60:4.619pt) arc (60:0:4.619pt) arc (0:-60:4.619pt);\end{tikzpicture}
}
\hspace{1cm}
\nopar\label{fig:examples:I}
}
\hspace{0.1cm}
      \subfigure[Example~(b).]
        {
\hspace{0.5cm}
\scalebox{1.2}{
\begin{tikzpicture}\draw[rhrhombidraw] (-4.619pt,8.0001pt) -- (23.095pt,8.0001pt) ;\draw[rhrhombidraw] (0.0pt,0.0pt) -- (9.238pt,16.0002pt) -- (18.476pt,0.0pt) ;\draw[rhrhombidraw] (0.0pt,16.0002pt) -- (9.238pt,0.0pt) -- (18.476pt,16.0002pt) ;\draw[rhrhombithickside] (4.619pt,8.0001pt) -- (13.857pt,8.0001pt);\draw[thin,-my] (2.3095pt,4.0001pt) arc (240:180:4.619pt) arc (180:120:4.619pt) arc (120:60:4.619pt) arc (240:300:4.619pt) arc (120:60:4.619pt) arc (60:0:4.619pt) arc (0:-60:4.619pt);\draw[thin,-my] (16.1665pt,4.0001pt) arc (300:240:4.619pt);\draw[thin,-my] (11.5475pt,4.0001pt) arc (60:120:4.619pt);\draw[thin,-my] (6.9285pt,4.0001pt) arc (300:240:4.619pt);\draw[rhrhombidraw] (-4.619pt,8.0001pt) -- (0.0pt,16.0002pt) -- (18.476pt,16.0002pt) -- (23.095pt,8.0001pt) -- (18.476pt,0.0pt) -- (0.0pt,0.0pt) -- cycle;\end{tikzpicture}
}
\hspace{0.5cm}
\nopar\label{fig:examples:II}
}
\hspace{0.1cm}
      \subfigure[Example~(c) with $p$ consisting of two turns at the upper left.]
        {
\hspace{0.5cm}
\scalebox{1.2}{
\begin{tikzpicture}\draw[rhrhombidraw] (4.619pt,8.0001pt) -- (23.095pt,8.0001pt) ;\draw[rhrhombidraw] (-4.619pt,-8.0001pt) -- (4.619pt,8.0001pt) ;\draw[thin,-my] (0.0pt,8.0001pt) arc (180:120:4.619pt) arc (120:60:4.619pt);\draw[thin,-my] (6.9285pt,12.0002pt) arc (240:300:4.619pt);\draw[thin,-my] (11.5475pt,12.0002pt) arc (120:60:4.619pt);\draw[thin,-my] (16.1665pt,12.0002pt) arc (60:0:4.619pt);\draw[thin,-my] (18.476pt,8.0001pt) arc (0:-60:4.619pt);\draw[thin,-my] (16.1665pt,4.0001pt) arc (300:240:4.619pt);\draw[thin,-my] (11.5475pt,4.0001pt) arc (60:120:4.619pt);\draw[thin,-my] (6.9285pt,4.0001pt) arc (120:180:4.619pt);\draw[thin,-my] (4.619pt,0.0pt) arc (0:-60:4.619pt);\draw[thin,-my] (2.3095pt,-4.0001pt) arc (300:240:4.619pt);\draw[thin,-my] (-2.3095pt,-4.0001pt) arc (240:180:4.619pt);\draw[thin,-my] (-4.619pt,0.0pt) arc (180:120:4.619pt);\draw[thin,-my] (-2.3095pt,4.0001pt) arc (300:360:4.619pt);\draw[rhrhombidraw] (0.0pt,16.0002pt) -- (9.238pt,0.0pt) -- (18.476pt,16.0002pt) ;\draw[rhrhombidraw] (-4.619pt,8.0001pt) -- (4.619pt,-8.0001pt) ;\draw[rhrhombidraw] (9.238pt,16.0002pt) -- (18.476pt,0.0pt) ;\draw[rhrhombidraw] (9.238pt,0.0pt) -- (-9.238pt,0.0pt) ;\draw[rhrhombithickside] (0.0pt,0.0pt) -- (4.619pt,8.0001pt);\draw[rhrhombithickside] (4.619pt,8.0001pt) -- (13.857pt,8.0001pt);\draw[rhrhombidraw] (4.619pt,-8.0001pt) -- (9.238pt,0.0pt) -- (18.476pt,0.0pt) -- (13.857pt,-8.0001pt) -- cycle;\draw[rhrhombidraw] (4.619pt,-8.0001pt) -- (-4.619pt,-8.0001pt) -- (-9.238pt,0.0pt) -- (0.0pt,16.0002pt) -- (18.476pt,16.0002pt) -- (23.095pt,8.0001pt) -- (18.476pt,0.0pt) ;\end{tikzpicture}
}
\hspace{0.5cm}
        \nopar\label{fig:examples:III}
}
    \nopar\label{fig:examples}
\caption{Illustration of the examples in Subsection~\ref{subsec:bypassing}. Rhombi where the diagonal is not drawn are $f$-flat.
All other rhombi are not $f$-flat.
The secure turnpath $p$ is drawn as a single long arrow, while the shortest turnpath $q$ from $\pend p$ to $\pstart p$ is drawn as several single turns.
Fat lines indicate diagonals of nearly $f$-flat rhombi $\varrho$ which satisfy $\s \varrho {pq} = -2$. Hence $f+\comb(pq) \notin P$.}
\end{center}
\end{figure}

\textbf{Example (a):}
Figure~\ref{fig:examples:I} shows a secure turnpath~$p$ and a trivial extension $q$
that uses no turnvertices of~$p$. Nevertheless, $f+\comb(pq) \notin P$.
In the light of Example~(a) we want to consider only those~$q$
where $q$ uses no turnvertices in those hive triangles where $p$ uses turnvertices.

\textbf{Example (b):}
Figure~\ref{fig:examples:II} shows a secure turnpath~$p$ and a trivial extension $q$
that uses no turnvertices of~$p$ and uses no turnvertices in hive triangles where $p$ uses turnvertices.
But still we have $f+\comb(pq) \notin P$.
Example~(b) gives the idea to consider only those~$q$ that use no turnvertices in nearly $f$\dash flat rhombi
where $p$ uses a negative slack contribution.

\textbf{Example (c):}
Figure~\ref{fig:examples:III} shows a secure turnpath~$p$ and a trivial extension $q$
that uses no turnvertices of~$p$ and uses no turnvertices in hive triangles where $p$ uses turnvertices.
Moreover, $q$ uses no turnvertices in nearly $f$\dash rhombi in which $p$ uses a negative slack contribution.
Nevertheless, $f+\comb(pq) \notin P$.
A situation as in Example~(c) is possible, because $p$ consists of only 2 turns.
We will see that 3 turns are enough to avoid these problems.

Considering more and more special subclasses of turnpaths $q$ in $\resf f$
leads to the forthcoming definition of the digraph $\resp f p$.
The turnpaths $q$ will be shortest turnpaths on $\resp f p$.
\begin{definition}[the digraph $\resp f p$]
\label{def:resfp}
Let $f \in P_\Z$ and $p$ be an $f$\dash secure turnpath on $\resf f$.
We define the digraph $\resp f p$ by further deleting from $\resf f$ a subset of its turnvertices:
Delete from $\resf f$ each turnvertex lying in a hive triangle in which $p$ uses turnvertices.
Moreover, for all \nff rhombi $\rhc$ in which $p$ uses $\rhpoulMl$ or $\rhpolrWl$,
delete all turnvertices of $\rhc$.
If we deleted $\pstart p$ or $\pend p$, add them back.
%
\end{definition}
We denote with $C'(\resp f p)$ the set of turnpaths in $\resp f p$ from $\pend p$ to $\pstart p$.
Hence for each $q \in C'(\resp f p)$ we have that $pq \in C(\resf f)$.
If for an $f$\dash secure turnpath $p$ in $\resf f$ we have that $C'(\resp f p) \neq \emptyset$,
then we call $p$ \emph{$f$\dash extendable}.
Note the following crucial fact: All $f$\dash securely extendable turnpaths are also $f$\dash extendable.
Unlike $f$\dash secure extendability, $f$\dash extendability can be handled efficiently:
We can check in time $\gO(n^2)$ whether an $f$\dash secure turnpath $p$ is $f$\dash extendable or not by
breadth-first-search on $\resp f p$.

Definition~\ref{def:resfp} is precisely what we want, which can be seen in
the following result, which is a variant of Theorem~4.8 in~\cite{bi:12}.
\begin{theorem}[Shortest Turncycle Theorem]
\label{thm:shoturncyc}
Let $p$ be an $f$\dash secure turnpath in~$\resf f$, consisting of at least 3 turns.
Let $q$ be a shortest turnpath in~$C'(\resp f {p})$.
Then $f+\comb(pq) \in P$.
\end{theorem}
The proof is very similar to the one of Theorem~4.8 in \cite{bi:12}
and in Section~\ref{sec:shoturncyc} we outline the main changes that need to be made when adapting it.

Consider now Algorithm~\ref{alg:findneigh}, which is a refinement of Algorithm~\ref{alg:findneighHypoth}.
\begin{algorithm}
\caption{\algocaptionwithoutparam{FindNeigh}}
\nopar\label{alg:findneigh}
\begin{algorithmic}[1]
\REQUIRE $f \in P_\Z$; an $f$\dash extendable turnpath $p$ in $\resf f$ consisting of at least 3 turns
\ENSURE (1) Prints at least all integral flows $f+c \in P_\Z$, where $c \in C(G)$ is a cycle that contains~$p$,
but may print other elements of $P_\Z$ as well. (2) Prints at least one element.
\IF{$\pstart p = \pend p$}
\STATE $\print{f+\comb(p)}$ and \Return.\nopar\label{alg:firstprint}
\ENDIF
\STATE \textbf{local} \texttt{foundpath} $\leftarrow$ \False.
\FOR{\textbf{both} turnedges $e:=(\pend p,z) \in E(\resp f p)$}
\STATE Concatenate $p' \leftarrow pe$.
\IF{$p'$ is $f$\dash extendable}
 \STATE Recursively call {\scshape FindNeigh}$(f, p')$.
 \STATE \texttt{foundpath} $\leftarrow$ \True.
\ENDIF
\ENDFOR
\IF{\textbf{not} \texttt{foundpath}\nopar\label{alg:notfoundpath}}
 \STATE $\print{f+\comb(p q)}$ with a shortest $q \in C'(\resp f p)$ and \Return.\nopar\label{alg:secondprint}
\ENDIF
\end{algorithmic}
\end{algorithm}

We will see in Remark~\ref{rem:lastpart} in which situations the condition of line~\ref{alg:notfoundpath} is satisfied.
\begin{lemma}
\label{lem:findneighcorrect}
Algorithm~\ref{alg:findneigh} works according to its output specification.
\end{lemma}
\begin{proof}
Recall that by definition, all $f$\dash extendable turnpaths are $f$\dash secure.
Hence in line~\ref{alg:firstprint}, only elements of $P_\Z$ are printed.
The flows printed in line~\ref{alg:secondprint} are elements of $P_\Z$ because of the Shortest Turncycle Theorem~\ref{thm:shoturncyc}.

If $p$ is $f$\dash securely extendable, then the binary recursion tree of Algorithm~\ref{alg:findneighHypoth}
is a subtree of the binary recursion tree of Algorithm~\ref{alg:findneigh},
because $f$\dash secure extendability implies $f$\dash extendability.
Hence in this case, Algorithm~\ref{alg:findneigh} meets the output specification requirement (1).

If $p$ is not $f$\dash securely extendable, then there exists no
$f$\dash secure turncycle~$c$ containing $p$ such that $f+\comb(c) \in P_\Z$,
see Proposition~\ref{pro:secureneighbors}.
Thus in this case, Algorithm~\ref{alg:findneigh} trivially meets the output specification requirement (1).

The output specification requirement (2) is satisfied because exactly one of the following three cases occurs:
(a) Algorithm~\ref{alg:findneigh} prints an element in line~\ref{alg:firstprint} and returns
or (b) Algorithm~\ref{alg:findneigh} calls itself recursively
or (c) Algorithm~\ref{alg:findneigh} prints an element in line~\ref{alg:secondprint} and returns.
\end{proof}
\begin{remark}
\label{rem:lastpart}
If during a run of Algorithm~\ref{alg:findneigh} we have \texttt{foundpath}$=$\False in line~\ref{alg:notfoundpath},
then we have an $f$\dash extendable turnpath $p$ such that
its concatenation $p' \leftarrow pe$ with any further turnedge $e$ results in $p'$ being not $f$\dash extendable.
For example, this can happen if $p$ ends with $\rhpollWl$ and $q \in C'(\resp f p)$ continues with $\rhpcMrl$,
but $q$ also uses $\rhpoulMlr$.
Then $q$ uses two turnvertices in
\begin{tikzpictured}\draw[rhrhombidraw] (0.0pt,0.0pt) -- (-4.619pt,8.0001pt) -- (0.0pt,16.0002pt) -- (4.619pt,8.0001pt) -- cycle;\fill[rhrhombifill] (-4.619pt,8.0001pt) -- (4.619pt,8.0001pt) -- (0.0pt,16.0002pt) -- cycle;\end{tikzpictured}%
.
If $q$ is the only element in $C'(\resp f p)$, then adding a turnedge to $p$ destroys $f$\dash extendability:
A turnpath $q' \in C'(\resp f {p'})$ would mean that there exists a concatenated turncycle $p'q'$ that uses
only a single turnvertex in the hive triangle
\begin{tikzpictured}\draw[rhrhombidraw] (0.0pt,0.0pt) -- (-4.619pt,8.0001pt) -- (0.0pt,16.0002pt) -- (4.619pt,8.0001pt) -- cycle;\fill[rhrhombifill] (-4.619pt,8.0001pt) -- (4.619pt,8.0001pt) -- (0.0pt,16.0002pt) -- cycle;\end{tikzpictured}%
, a contradiction to the uniqueness of~$q$.
In this situation, $f+\comb(pq)$ is printed in line~\ref{alg:secondprint}
and since $pq$ is not ordinary, this means that an element of $P_\Z$ is printed
that does lie in $\Gamma (f)$.
\end{remark}
\begin{lemma}
\label{lem:time}
Let $f \in P_\Z$.
On input $f \in P_\Z$, Algorithm~\ref{alg:findneigh} prints out distinct elements.
The first $k$ elements are printed in time $\gO(k \cdot n^4)$.
\end{lemma}
\begin{proof}
We use the fact that $f$\dash extendability can be decided in time $\gO(n^2)$.
The rest of the proof is analogous to the proof of Lemma~\ref{lem:hypoiscorrectII}.
Note that it is crucial in this proof that for each recursive algorithm call we can ensure
that at least one element of $P_\Z$ is printed.
This is guaranteed by line~\ref{alg:secondprint}.
\end{proof}
\begin{remark}
If we delete line~\ref{alg:secondprint} from Algorithm~\ref{alg:findneigh},
then its execution gives a recursion tree where at some leafs no elements are printed.
Then it is not clear how much time is spent visiting elementless leafs
and we cannot prove Lemma~\ref{lem:time}.
\end{remark}

Analogously to the definition of {\scshape NeighGen'},
we can define the algorithm {\scshape NeighGen} as required for Theorem~\ref{thm:neigh} as follows:
Call Algorithm~\ref{alg:findneigh} several times with fixed $f \in P_\Z$,
but each time with a different secure turnpath $p \in E(\resf f)$
consisting of 3 turns such that $p$ is $f$\dash extendable.
We now prove the correctness and running time of {\scshape NeighGen}.

Let $\Gamm f$ be the set of flows printed by {\scshape NeighGen}.
Since $f$\dash secure extendability implies $f$\dash extendability,
each flow printed by {\scshape NeighGen'} is also printed by {\scshape NeighGen},
so $\Gamma(f) \subseteq \Gamm f$.
Lemma~\ref{lem:findneighcorrect} implies $\Gamm f \subseteq P_\Z$.

Since we have $\gO(n^2)$ calls of Algorithm~\ref{alg:findneigh},
we get a total running time of $\gO(kn^6)$.

This proves Theorem~\ref{thm:neigh} with one hole remaining: The proof of the Shortest Turncycle Theorem \ref{thm:shoturncyc}.

\section{The Shortest Turncycle Theorem}
\label{sec:shoturncyc}
In this section we discuss how to prove the generalization of~\cite[Theorem~4.8]{bi:12},
namely the Shortest Turncycle Theorem~\ref{thm:shoturncyc}.
At the same time, we use this section for proving Propositions~\ref{pro:noreverse} and~\ref{pro:special},
which will be needed in Section~\ref{sec:ktt} in a little bit more general form than they are found in~\cite{bi:12}.
To achieve this, we slightly generalize the term ``turnpath''.
\begin{definition}
Let $k \in E(\Delta)$ and let $\eta$ be one of the two directions in which $k$ can be crossed by turnedges.
Then $p:=(k,\eta)$ is called a \emph{turnpath of length~0}.
For turnpaths of length 0,
the symbol $C'(\resp f p)$ is defined as the set of all turnpaths in $\resf f$ that cross $k$ in direction~$\eta$.
By a \emph{generalized turnpath} we understand a turnpath or a turnpath of length~0.
Turnpaths of length~0 are defined to be secure.
\end{definition}

In the following, we state some facts about shortest turncycles, which are variants from propositions proved in~\cite[Sec.~7.1]{bi:12}.

Each turnvertex in $\res$ has a \emph{reverse turnvertex} in $R$ 
that points in the other direction, e.g., the reverse turnvertex of $\rhpcMl$ is $\rhpoulMr$.
\begin{proposition}
\label{pro:noreverse}
Let $f \in P_\Z$ and let $p$ be an $f$\dash secure generalized turnpath on $\resf f$.
A shortest turnpath in $C'(\resp f p)$ cannot use a turnvertex and its reverse.
\end{proposition}
\begin{proof}
Similar to the proof of~\cite[Prop.~7.1]{bi:12}.
A little care must be taken, because $\resp f p$ has only a subset of the turnvertices of $\resf f$.
\end{proof}

In the following we fix $f\in P_\Z$, a secure generalized turnpath $p$, and a \emph{shortest} turnpath $q \in C'(\resp f p)$.
Proposition~\ref{pro:noreverse} directly implies that
the diagonal of each rhombus is crossed at most twice by $q$.
A rhombus~$\varrho$ is called {\em special} if the turnpath~$q \in C'(\resp f p)$ crosses its diagonal twice.
If the crossing is in the same direction, then 
$\varrho$ is called  \emph{confluent},  
otherwise, if the crossing is in opposite directions, 
$\varrho$ is called \emph{contrafluent}.
\begin{proposition}
\label{pro:special}
\begin{enumerate}
 \item \label{pro:special:edges}In a confluent rhombus, $q$ uses exactly the turnedges $\rhpoulMrl$ and $\rhpourMlr$ and no other turnvertex.
In a contrafluent rhombus, $q$ uses exactly the turnedges $\rhpoulMrl$ and $\rhpollWlr$ and no other turnvertex.
 \item \label{pro:special:nooverlap}Special rhombi do not overlap.
\end{enumerate}
\end{proposition}
\begin{proof}
Analogous to the proofs in~\cite[Sec.~7.1]{bi:12}.
\end{proof}
Let $c:=pq \in C(\resf f)$.
We set 
\begin{equation*}
 \varepsilon := \max \{ t \in \R \mid f + t \comb(c) \in P\}, \ \ g:=f+\varepsilon \comb(c).
\end{equation*}
Then we have $g \in P$ and by Lemma~\ref{lem:direction} we have $\varepsilon > 0$.
For the proof of the Shortest Turncycle Theorem~\ref{thm:shoturncyc} 
it suffices to show that $\varepsilon \geq 1$, since then $f+\comb(c) \in P_\Z$.

If all rhombi are $f$\dash flat, then there are no turncycles in $\resf f$ (see~\cite[Prop.~4.12 and Fig.~6]{bi:12}),
in contradiction to the existence of $q \in C'(\resp f p)$.
So in the following we suppose that 
not all rhombi are $f$\dash flat.
We shall argue indirectly and assume that $\varepsilon<1$.

\begin{definition}\label{def:critical}
A rhombus is called \emph{critical} if it is not $f$\dash flat, but $g$\dash flat.
\end{definition}
The following claim is easy to see.
\begin{claim}
\label{cla:verynegslack}
 For critical rhombi $\varrho$ we have $\s \varrho c \leq -2$.
\end{claim}
\begin{claim}
\label{cla:pnotcritical}
If $p$ consists of at least 3 turns, then
$p$ does not use turnvertices in critical rhombi.
\end{claim}
\begin{proof}
Assume that $p$ uses turnvertices in a critical rhombus $\rhc$.
W.l.o.g. let $p$ use a turnvertex in the hive triangle
\begin{tikzpictured}\draw[rhrhombidraw] (0.0pt,0.0pt) -- (-4.619pt,8.0001pt) -- (0.0pt,16.0002pt) -- (4.619pt,8.0001pt) -- cycle;\fill[rhrhombifill] (-4.619pt,8.0001pt) -- (4.619pt,8.0001pt) -- (0.0pt,16.0002pt) -- cycle;\end{tikzpictured}%
.
If $p$ uses $\rhpcMl$, then $\rhaoulW(c)=1$.
Since $\rhaolrM(c) \in\{-2,-1,0,1,2\}$ we have $\s \rhc {c} \geq -1$.
This implies that $\rhc$ is not critical according to Claim~\ref{cla:verynegslack}.
Analogous arguments show that $p$ does not use $\rhpourMr$ or $\rhpourMl$.

If $p$ uses $\rhpoulMl$ and $\s \rhc f = 1$, then by construction of $\resp f p$, $c$ uses no further turnvertex in~$\rhc$
and hence $\rhc$ is not critical.
If $p$ uses $\rhpoulMl$ and $\s \rhc f > 1$, then $\s \rhc c \leq -3$.
This implies $\rhaolrW(c)=2$ and $\rhaollW(c)=2$, which results in overlapping special rhombi $\rhrholr$ and $\rhrholl$:
A contradiction to Proposition~\ref{pro:special}(\ref{pro:special:nooverlap}).

Assume that $p$ uses $\rhpcMr$.
If $\rhaollM(c)=2$, then $q$ uses $\rhpolrWlr$ and $\rhpcWrl$, in contradiction to $p$ using $\rhpcMr$.
Hence $\rhaollM(c)\leq 1$.
Since $\rhc$ is critical, it follows $\rhaollM(c) = 1$.
But since $p$ uses $\rhpcMr$, $q$ uses both $\rhpplrMlr$ and $\rhplrMrl$.
Therefore, $\rhrholr$ is $f$\dash flat.
The hexagon equality (Claim~\ref{cla:BZ}) implies that either
we have $\s \rhrhur f = 0$ and $\s \rhrhpr f = 1$
or we have $\s \rhrhur f = 1$ and $\s \rhrhpr f = 0$.
In the former case, $p$ continues with $\rhpcMrrl$.
This implies that all turnvertices in the hive triangle
\begin{tikzpictured}\draw[rhrhombidraw] (0.0pt,0.0pt) -- (-4.619pt,8.0001pt) -- (0.0pt,16.0002pt) -- (4.619pt,8.0001pt) -- cycle;\fill[rhrhombifill] (4.619pt,8.0001pt) -- (13.857pt,8.0001pt) -- (9.238pt,0.0pt) -- cycle;\end{tikzpictured}
\ are deleted in $\resp f p$,
in contradiction to $q$ using $\rhplrMrl$.
Therefore $\s \rhrhur f = 1$ and $\s \rhrhpr f = 0$.
But by construction of $\resf f$, $c$ uses
\begin{tikzpictured}\draw[rhrhombidraw] (0.0pt,0.0pt) -- (-4.619pt,8.0001pt) -- (0.0pt,16.0002pt) -- (4.619pt,8.0001pt) -- cycle;\draw[thin,-my] (11.5475pt,12.0002pt) arc (120:180:4.619pt) arc (0:-60:4.619pt) arc (300:240:4.619pt) arc (60:120:4.619pt);\end{tikzpictured}%
.
The turnpath $p$ can continue as $\rhpcMrl$ or as $\rhpcMrr$,
but both cases will yield a contradiction.
Assume $p$ continues as $\rhpcMrl$.
Then the hive triangle
\begin{tikzpictured}\draw[rhrhombidraw] (0.0pt,0.0pt) -- (-4.619pt,8.0001pt) -- (0.0pt,16.0002pt) -- (4.619pt,8.0001pt) -- cycle;\fill[rhrhombifill] (4.619pt,8.0001pt) -- (13.857pt,8.0001pt) -- (9.238pt,16.0002pt) -- cycle;\end{tikzpictured}
\ has no turnvertices in $\resp f p$,
in contradiction to $c$ using
\begin{tikzpictured}\draw[rhrhombidraw] (0.0pt,0.0pt) -- (-4.619pt,8.0001pt) -- (0.0pt,16.0002pt) -- (4.619pt,8.0001pt) -- cycle;\draw[thin,-my] (11.5475pt,12.0002pt) arc (120:180:4.619pt) arc (0:-60:4.619pt) arc (300:240:4.619pt) arc (60:120:4.619pt);\end{tikzpictured}%
.
So now assume that $p$ continues as $\rhpcMrr$.
Since $\rhrhpr$ is $f$\dash flat, $p$ continues as $\rhpcMrrr$,
also in contradiction to $c$ using
\begin{tikzpictured}\draw[rhrhombidraw] (0.0pt,0.0pt) -- (-4.619pt,8.0001pt) -- (0.0pt,16.0002pt) -- (4.619pt,8.0001pt) -- cycle;\draw[thin,-my] (11.5475pt,12.0002pt) arc (120:180:4.619pt) arc (0:-60:4.619pt) arc (300:240:4.619pt) arc (60:120:4.619pt);\end{tikzpictured}%
.
%

The proof that $p$ does not use $\rhpoulMr$ is analogous to above argument.
\end{proof}
With Claim~\ref{cla:pnotcritical} in mind,
the rest of the proof of the Shortest Turncycle Theorem~\ref{thm:shoturncyc} is analogous to \cite[Sec.~7.2]{bi:12}.
The requirement that $p$ must have at least 3 turns is needed
when proving the analogues of \cite[Claim~7.13]{bi:12} and \cite[Claim~7.16]{bi:12}
when deducing the contradiction to the choice of the first critical rhombus.

\section{Proof of the King-Tollu-Toumazet Conjecture}
\label{sec:ktt}
In this section we prove Theorem~\ref{thm:ktt},
more precisely, we prove the following equivalent geometric formulation.
\begin{theorem}
\label{thm:ktt:lineseg}
Let $\LRC \la \mu \nu = 2$ and let $f_1$ and $f_2$ be the two integral points of $P(\la,\mu,\nu)$.
Then $P(\la,\mu,\nu)$ is exactly the line segment between $f_1$ and $f_2$.
\end{theorem}
\begin{claim}
Theorem~\ref{thm:ktt:lineseg} is equivalent to Theorem~\ref{thm:ktt}.
\end{claim}
\begin{proof}
Let $\LRC \la \mu \nu = 2$ and let $P = P(\la,\mu,\nu)$.
For every natural number $M$, the stretched polytope $MP$ contains
at least the $M+1$ integral points
\[\kappa_M := \{M f_1, \  (M-1)f_1 + f_2, \ (M-2)f_1 + 2 f_2, \ \ldots, \ f_1 + (M-1)f_2, \ M f_2\}\]
and hence $\LRC {M\la} {M\mu} {M\nu} \geq M+1$.

If $P$ contains a point besides the line segment between $f_1$ and $f_2$,
then $P$ contains a rational point $x$ besides this line segment.
But then there exists $M$ such that $Mx$ is integral and hence
$\LRC {M\la} {M\mu} {M\nu} > M+1$.

On the other hand,
if $P$ contains no other point besides the line segment,
then it is easy to see that these $M+1$ points are the only integral points in $MP$ and hence
$\LRC {M\la} {M\mu} {M\nu} = M+1$.
\end{proof}
From now on, fix partitions $\la$, $\mu$ and $\nu$ such that $\LRC \la \mu \nu = 2$.
Let $f_1$, $f_2$ be the two flows in $P(\la,\mu,\nu)_\Z$.
Let $c_1 := f_2 - f_1$ denote the $f_1$\dash secure cycle that connects the integral points in $P$,
and analogously define $c_2 := -c_1 = f_1-f_2$ the $f_2$\dash secure cycle running in the other direction.
W.l.o.g.\ $c_1$ runs in counterclockwise direction, otherwise we switch $f_1$ and $f_2$.

For $M \in \N$, the stretched polytope $MP$ contains at least the set of integral points
\[\kappa_M := \{Mf_1 + m c_1 \mid m \in \N, \ 0 \leq m \leq M\}.\]
To prove Theorem~\ref{thm:ktt:lineseg} it remains to show that for all $M \in \N_{\geq 2}$
these are the only integral points in $MP$.
According to the Connectedness Theorem~\ref{thm:connectedness},
this is equivalent to the statement
that for each integral flow $\xi \in \kappa_M$
the neighborhood $\Gamma(\xi)$ is contained in $\kappa_M$.
To prove this,
we fix some $M \in \N_{\geq 2}$.
Now we choose an arbitrary $\xi \in \kappa_M$,
which means choosing a natural number $0 \leq m \leq M$
such that $\xi = Mf_1 + mc_1$.
If $\xi \in \{Mf_1,Mf_2\}$, we say that $\xi$ is \emph{extremal},
otherwise $\xi$ is called \emph{inner}.
We are interested in the neighborhood $\Gamma(\xi)$.

If $\xi \neq f_2$, then
$\xi+c_1 = Mf_1 + (m+1)c_1 \in MP$,
which implies that the cycle $c_1$ is $\xi$\dash secure by Proposition~\ref{pro:secureneighbors}.
Analogously, $c_2$ is $\xi$\dash secure for $\xi \neq f_2$.
It follows that both $c_1$ and $c_2$ are $\xi$\dash secure for an inner~$\xi$.
It remains to show that for an inner $\xi$ the only
$\xi$\dash secure proper cycles are $c_1$ and $c_2$,
and for both $\iot \in \{1,2\}$
the only $f_\iot$\dash secure proper cycle is $c_\iot$.
In fact, we are going to show the following slightly stronger statement.
\begin{equation}
\tag{\textreferencemark}
\begin{minipage}{11cm}
For an inner $\xi$, the proper cycles $c_1$ and $c_2$ are the only $\xi$\dash hive preserving proper cycles.
For both $\iot \in \{1,2\}$ the only $f_\iot$\dash hive preserving proper cycle is $c_\iot$.
\end{minipage}
\end{equation}
Note that a proper cycle $c$ is $\xi$\dash hive preserving iff
$c$ is $\frac \xi M$\dash hive preserving.
Hence for proving (\textreferencemark) we can assume w.l.o.g.\
that $\xi = xf_1+(1-x)f_2$ with $0 \leq x \leq 1$ is a convex combination of $f_1$ and~$f_2$.
To classify all $\xi$\dash hive preserving proper cycles we use the following important proposition
which we prove in the next subsection.
\begin{proposition}
\label{pro:ktt:goal}
Let $\LRC \la \mu \nu = 2$ and let $f_1$ and $f_2$ be the two integral hive flows in $P(\la,\mu,\nu)$.
%
For each convex combination $\xi = x f_1 + (1-x) f_2$, $0 \leq x \leq 1$ we have
that each $\xi$\dash hive preserving proper cycle is
\emph{either} $f_1$\dash hive preserving \emph{or} $f_2$\dash hive preserving.
\end{proposition}
According to Proposition~\ref{pro:ktt:goal},
for proving (\textreferencemark)
it suffices to show that $c_1$ is the only $f_1$\dash hive preserving proper cycle
and $c_2$ is the only $f_2$\dash hive preserving proper cycle.
But this can be seen with the following key lemma,
whose proof we postpone to Subsection~\ref{subsec:ktt:secure}.
\begin{keylemma}
\label{keylem:ktt:secure}
Let $\LRC \la \mu \nu = 2$ and let $f_1$ and $f_2$ be the two integral points of $P(\la,\mu,\nu)$.
For each $\iot\in\{1,2\}$
each ordinary turncycle in $\resf {f_\iot}$ is $f_\iot$\dash secure.
\end{keylemma}
We apply Key Lemma~\ref{keylem:ktt:secure} as follows.
Let $\iot \in \{1,2\}$ and let $c' \neq c_\iot$ be an $f_\iot$\dash hive preserving proper cycle.
Then Key Lemma~\ref{keylem:ktt:secure} implies that $c'$ is $f_\iot$\dash secure.
The fact that $\LRC \la\mu\nu = 2$ implies $c' = c_\iot$.
Therefore we have shown that $c_\iot$ is the only $f_\iot$\dash hive preserving proper cycle
and are done proving Theorem~\ref{thm:ktt:lineseg}.

It remains to prove Key Lemma~\ref{keylem:ktt:secure} and Proposition~\ref{pro:ktt:goal}.

\subsection{Proof of Proposition~\protect\ref{pro:ktt:goal}}
\label{subsec:proofofkttgoal}
In this subsection we already use Key Lemma~\ref{keylem:ktt:secure}, which is proved later in Subsection~\ref{subsec:ktt:secure}.

We first prove Proposition~\ref{pro:ktt:goal} for extremal $\xi \in \{f_1,f_2\}$.
In this case it remains to show that there is no proper cycle that is both
$f_1$\dash hive preserving and $f_2$\dash hive preserving.
Indeed, for the sake of contradiction, assume there is such a cycle $c$.
Key Lemma~\ref{keylem:ktt:secure} implies that $c$ is both $f_1$\dash secure and $f_2$\dash secure,
in contradiction to \mbox{$\LRC \la\mu\nu = 2$}.

From now on, we assume that $\xi$ is inner.
Recall that in this case both $c_1$ and $c_2$ are $\xi$\dash hive preserving.

First of all, we note that if a rhombus $\rhc$ is $f_1$\dash flat and $f_2$\dash flat,
then $\rhc$ is also $\xi$\dash flat, because $\xi$ lies between $f_1$ and $f_2$.
The converse is also true: if a rhombus $\rhc$ is $\xi$\dash flat, then $\rhc$ is both $f_1$\dash flat
and $f_2$\dash flat, because $\s \rhc \xi = 0$ and $\s \rhc f_1 > 0$ would imply $\s \rhc f_2 < 0$,
in contradiction to $f_2 \in P$.

From this consideration, it follows that a cycle $c$ is $\xi$\dash hive preserving
if $c$ is $f_1$\dash hive preserving or $f_2$\dash hive preserving.

In the following Claim~\ref{cla:ktt:mustusepipe} we begin to rule out the existence
of $\xi$\dash hive preserving cycles whose curves have no intersection with the curve of~$c_1$.
We first introduce some terminology (see Figure~\ref{fig:ktt:pipes}):
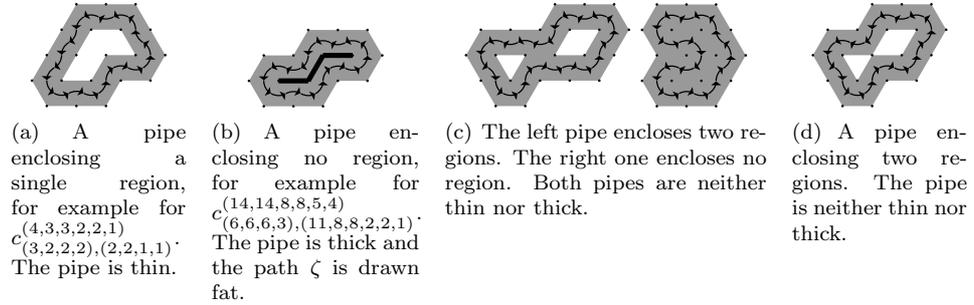
\begin{figure}[h]
\begin{center}
      \subfigure[A pipe enclosing a single region, for example for $\LRC{(3,2,2,2)}{(2,2,1,1)}{(4,3,3,2,2,1)}$. The pipe is thin.]
        {
\scalebox{1.2}{
\begin{tikzpicture}\fill[rhrhombifill] (23.095pt,8.0001pt) -- (4.619pt,8.0001pt) -- (-9.238pt,-16.0002pt) -- (-4.619pt,-24.0003pt) -- (13.857pt,-24.0003pt) -- (18.476pt,-16.0002pt) -- (27.714pt,-16.0002pt) -- (32.333pt,-8.0001pt) -- (23.095pt,8.0001pt) -- (18.476pt,0.0pt) -- (23.095pt,-8.0001pt) -- (13.857pt,-8.0001pt) -- (9.238pt,-16.0002pt) -- (0.0pt,-16.0002pt) -- (9.238pt,0.0pt) -- (18.476pt,0.0pt) -- cycle;\draw[thin,-my] (11.5475pt,-20.0003pt) arc (300:360:4.619pt);\draw[thin,-my] (13.857pt,-16.0002pt) arc (180:120:4.619pt);\draw[thin,-my] (16.1665pt,-12.0002pt) arc (120:60:4.619pt);\draw[thin,-my] (20.7855pt,-12.0002pt) arc (240:300:4.619pt);\draw[thin,-my] (25.4045pt,-12.0002pt) arc (300:360:4.619pt);\draw[thin,-my] (27.714pt,-8.0001pt) arc (0:60:4.619pt);\draw[thin,-my] (25.4045pt,-4.0001pt) arc (240:180:4.619pt);\draw[thin,-my] (23.095pt,0.0pt) arc (0:60:4.619pt);\draw[thin,-my] (20.7855pt,4.0001pt) arc (60:120:4.619pt);\draw[thin,-my] (16.1665pt,4.0001pt) arc (300:240:4.619pt);\draw[thin,-my] (11.5475pt,4.0001pt) arc (60:120:4.619pt);\draw[thin,-my] (6.9285pt,4.0001pt) arc (120:180:4.619pt);\draw[thin,-my] (4.619pt,0.0pt) arc (0:-60:4.619pt);\draw[thin,-my] (2.3095pt,-4.0001pt) arc (120:180:4.619pt);\draw[thin,-my] (0.0pt,-8.0001pt) arc (0:-60:4.619pt);\draw[thin,-my] (-2.3095pt,-12.0002pt) arc (120:180:4.619pt);\draw[thin,-my] (-4.619pt,-16.0002pt) arc (180:240:4.619pt);\draw[thin,-my] (-2.3095pt,-20.0003pt) arc (240:300:4.619pt);\draw[thin,-my] (2.3095pt,-20.0003pt) arc (120:60:4.619pt);\draw[thin,-my] (6.9285pt,-20.0003pt) arc (240:300:4.619pt);\fill (9.238pt,0.0pt) circle (0.4pt);\fill (18.476pt,0.0pt) circle (0.4pt);\fill (4.619pt,-8.0001pt) circle (0.4pt);\fill (13.857pt,-8.0001pt) circle (0.4pt);\fill (23.095pt,-8.0001pt) circle (0.4pt);\fill (0.0pt,-16.0002pt) circle (0.4pt);\fill (9.238pt,-16.0002pt) circle (0.4pt);\fill (4.619pt,8.0001pt) circle (0.4pt);\fill (13.857pt,8.0001pt) circle (0.4pt);\fill (23.095pt,8.0001pt) circle (0.4pt);\fill (27.714pt,0.0pt) circle (0.4pt);\fill (32.333pt,-8.0001pt) circle (0.4pt);\fill (27.714pt,-16.0002pt) circle (0.4pt);\fill (18.476pt,-16.0002pt) circle (0.4pt);\fill (13.857pt,-24.0003pt) circle (0.4pt);\fill (4.619pt,-24.0003pt) circle (0.4pt);\fill (-4.619pt,-24.0003pt) circle (0.4pt);\fill (-9.238pt,-16.0002pt) circle (0.4pt);\fill (-4.619pt,-8.0001pt) circle (0.4pt);\fill (0.0pt,0.0pt) circle (0.4pt);\end{tikzpicture}
}
\nopar\label{fig:ktt:pipes:thin}}
\hspace{0.1cm}
      \subfigure[A pipe enclosing no region, for example for $\LRC{(6,6,6,3)}{(11,8,8,2,2,1)}{(14,14,8,8,5,4)}$. The pipe is thick and the path $\zeta$ is drawn fat.]
        {
\hspace{0.1cm}
\scalebox{1.2}{
\begin{tikzpicture}\fill[rhrhombifill] (0.0pt,16.0002pt) -- (-4.619pt,8.0001pt) -- (0.0pt,0.0pt) -- (18.476pt,0.0pt) -- (23.095pt,8.0001pt) -- (32.333pt,8.0001pt) -- (36.952pt,16.0002pt) -- (32.333pt,24.0003pt) -- (13.857pt,24.0003pt) -- (9.238pt,16.0002pt) -- cycle;\draw[thin,-my] (16.1665pt,4.0001pt) arc (300:360:4.619pt);\draw[thin,-my] (18.476pt,8.0001pt) arc (180:120:4.619pt);\draw[thin,-my] (20.7855pt,12.0002pt) arc (120:60:4.619pt);\draw[thin,-my] (25.4045pt,12.0002pt) arc (240:300:4.619pt);\draw[thin,-my] (30.0235pt,12.0002pt) arc (300:360:4.619pt);\draw[thin,-my] (32.333pt,16.0002pt) arc (0:60:4.619pt);\draw[thin,-my] (30.0235pt,20.0003pt) arc (60:120:4.619pt);\draw[thin,-my] (25.4045pt,20.0003pt) arc (300:240:4.619pt);\draw[thin,-my] (20.7855pt,20.0003pt) arc (60:120:4.619pt);\draw[thin,-my] (16.1665pt,20.0003pt) arc (120:180:4.619pt);\draw[thin,-my] (13.857pt,16.0002pt) arc (0:-60:4.619pt);\draw[thin,-my] (11.5475pt,12.0002pt) arc (300:240:4.619pt);\draw[thin,-my] (6.9285pt,12.0002pt) arc (60:120:4.619pt);\draw[thin,-my] (2.3095pt,12.0002pt) arc (120:180:4.619pt);\draw[thin,-my] (0.0pt,8.0001pt) arc (180:240:4.619pt);\draw[thin,-my] (2.3095pt,4.0001pt) arc (240:300:4.619pt);\draw[thin,-my] (6.9285pt,4.0001pt) arc (120:60:4.619pt);\draw[thin,-my] (11.5475pt,4.0001pt) arc (240:300:4.619pt);\fill (13.857pt,24.0003pt) circle (0.4pt);\fill (23.095pt,24.0003pt) circle (0.4pt);\fill (32.333pt,24.0003pt) circle (0.4pt);\fill (0.0pt,16.0002pt) circle (0.4pt);\fill (9.238pt,16.0002pt) circle (0.4pt);\fill (18.476pt,16.0002pt) circle (0.4pt);\fill (27.714pt,16.0002pt) circle (0.4pt);\fill (36.952pt,16.0002pt) circle (0.4pt);\fill (-4.619pt,8.0001pt) circle (0.4pt);\fill (4.619pt,8.0001pt) circle (0.4pt);\fill (13.857pt,8.0001pt) circle (0.4pt);\fill (23.095pt,8.0001pt) circle (0.4pt);\fill (0.0pt,0.0pt) circle (0.4pt);\fill (9.238pt,0.0pt) circle (0.4pt);\fill (18.476pt,0.0pt) circle (0.4pt);\fill (32.333pt,8.0001pt) circle (0.4pt);\draw[rhrhombithickside] (4.619pt,8.0001pt) -- (13.857pt,8.0001pt);\draw[rhrhombithickside] (13.857pt,8.0001pt) -- (18.476pt,16.0002pt);\draw[rhrhombithickside] (18.476pt,16.0002pt) -- (27.714pt,16.0002pt);\end{tikzpicture}
}
\hspace{0.1cm}
        \nopar\label{fig:ktt:pipes:thick}
}
\hspace{0.1cm}
      \subfigure[The left pipe encloses two regions. The right one encloses no region. Both pipes are neither thin nor thick.]
        {
\scalebox{1.2}{
\begin{tikzpicture}\fill[rhrhombifill] (-13.857pt,8.0001pt) -- (-9.238pt,16.0002pt) -- (9.238pt,16.0002pt) -- (13.857pt,24.0003pt) -- (32.333pt,24.0003pt) -- (36.952pt,16.0002pt) -- (27.714pt,0.0pt) -- (9.238pt,0.0pt) -- (4.619pt,-8.0001pt) -- (-4.619pt,-8.0001pt) -- (-13.857pt,8.0001pt) -- (-4.619pt,8.0001pt) -- (0.0pt,0.0pt) -- (4.619pt,8.0001pt) -- (23.095pt,8.0001pt) -- (27.714pt,16.0002pt) -- (18.476pt,16.0002pt) -- (13.857pt,8.0001pt) -- (-4.619pt,8.0001pt) -- cycle;\fill (-9.238pt,16.0002pt) circle (0.4pt);\fill (0.0pt,16.0002pt) circle (0.4pt);\fill (9.238pt,16.0002pt) circle (0.4pt);\fill (18.476pt,16.0002pt) circle (0.4pt);\fill (27.714pt,16.0002pt) circle (0.4pt);\fill (36.952pt,16.0002pt) circle (0.4pt);\fill (13.857pt,24.0003pt) circle (0.4pt);\fill (23.095pt,24.0003pt) circle (0.4pt);\fill (32.333pt,24.0003pt) circle (0.4pt);\fill (-13.857pt,8.0001pt) circle (0.4pt);\fill (-4.619pt,8.0001pt) circle (0.4pt);\fill (4.619pt,8.0001pt) circle (0.4pt);\fill (13.857pt,8.0001pt) circle (0.4pt);\fill (23.095pt,8.0001pt) circle (0.4pt);\fill (32.333pt,8.0001pt) circle (0.4pt);\fill (-9.238pt,0.0pt) circle (0.4pt);\fill (0.0pt,0.0pt) circle (0.4pt);\fill (9.238pt,0.0pt) circle (0.4pt);\fill (18.476pt,0.0pt) circle (0.4pt);\fill (27.714pt,0.0pt) circle (0.4pt);\fill (-4.619pt,-8.0001pt) circle (0.4pt);\fill (4.619pt,-8.0001pt) circle (0.4pt);\draw[thin,-my] (25.4045pt,4.0001pt) arc (300:360:4.619pt);\draw[thin,-my] (27.714pt,8.0001pt) arc (180:120:4.619pt);\draw[thin,-my] (30.0235pt,12.0002pt) arc (300:360:4.619pt);\draw[thin,-my] (32.333pt,16.0002pt) arc (0:60:4.619pt);\draw[thin,-my] (30.0235pt,20.0003pt) arc (60:120:4.619pt);\draw[thin,-my] (25.4045pt,20.0003pt) arc (300:240:4.619pt);\draw[thin,-my] (20.7855pt,20.0003pt) arc (60:120:4.619pt);\draw[thin,-my] (16.1665pt,20.0003pt) arc (120:180:4.619pt);\draw[thin,-my] (13.857pt,16.0002pt) arc (0:-60:4.619pt);\draw[thin,-my] (11.5475pt,12.0002pt) arc (300:240:4.619pt);\draw[thin,-my] (6.9285pt,12.0002pt) arc (60:120:4.619pt);\draw[thin,-my] (2.3095pt,12.0002pt) arc (300:240:4.619pt);\draw[thin,-my] (-2.3095pt,12.0002pt) arc (60:120:4.619pt);\draw[thin,-my] (-6.9285pt,12.0002pt) arc (120:180:4.619pt);\draw[thin,-my] (-9.238pt,8.0001pt) arc (180:240:4.619pt);\draw[thin,-my] (-6.9285pt,4.0001pt) arc (60:0:4.619pt);\draw[thin,-my] (-4.619pt,0.0pt) arc (180:240:4.619pt);\draw[thin,-my] (-2.3095pt,-4.0001pt) arc (240:300:4.619pt);\draw[thin,-my] (2.3095pt,-4.0001pt) arc (300:360:4.619pt);\draw[thin,-my] (4.619pt,0.0pt) arc (180:120:4.619pt);\draw[thin,-my] (6.9285pt,4.0001pt) arc (120:60:4.619pt);\draw[thin,-my] (11.5475pt,4.0001pt) arc (240:300:4.619pt);\draw[thin,-my] (16.1665pt,4.0001pt) arc (120:60:4.619pt);\draw[thin,-my] (20.7855pt,4.0001pt) arc (240:300:4.619pt);\end{tikzpicture}
\begin{tikzpicture}\fill[rhrhombifill] (0.0pt,-16.0002pt) -- (18.476pt,-16.0002pt) -- (27.714pt,0.0pt) -- (18.476pt,16.0002pt) -- (0.0pt,16.0002pt) -- (-4.619pt,8.0001pt) -- (0.0pt,0.0pt) -- (-4.619pt,-8.0001pt) -- cycle;\fill (4.619pt,8.0001pt) circle (0.4pt);\fill (13.857pt,8.0001pt) circle (0.4pt);\fill (9.238pt,0.0pt) circle (0.4pt);\fill (18.476pt,0.0pt) circle (0.4pt);\fill (4.619pt,-8.0001pt) circle (0.4pt);\fill (13.857pt,-8.0001pt) circle (0.4pt);\draw[thin,-my] (16.1665pt,-12.0002pt) arc (300:360:4.619pt);\draw[thin,-my] (18.476pt,-8.0001pt) arc (180:120:4.619pt);\draw[thin,-my] (20.7855pt,-4.0001pt) arc (300:360:4.619pt);\draw[thin,-my] (23.095pt,0.0pt) arc (0:60:4.619pt);\draw[thin,-my] (20.7855pt,4.0001pt) arc (240:180:4.619pt);\draw[thin,-my] (18.476pt,8.0001pt) arc (0:60:4.619pt);\draw[thin,-my] (16.1665pt,12.0002pt) arc (60:120:4.619pt);\draw[thin,-my] (11.5475pt,12.0002pt) arc (300:240:4.619pt);\draw[thin,-my] (6.9285pt,12.0002pt) arc (60:120:4.619pt);\draw[thin,-my] (2.3095pt,12.0002pt) arc (120:180:4.619pt);\draw[thin,-my] (0.0pt,8.0001pt) arc (180:240:4.619pt);\draw[thin,-my] (2.3095pt,-4.0001pt) arc (120:180:4.619pt);\draw[thin,-my] (0.0pt,-8.0001pt) arc (180:240:4.619pt);\draw[thin,-my] (2.3095pt,-12.0002pt) arc (240:300:4.619pt);\draw[thin,-my] (6.9285pt,-12.0002pt) arc (120:60:4.619pt);\draw[thin,-my] (11.5475pt,-12.0002pt) arc (240:300:4.619pt);\draw[thin,-my] (2.3095pt,4.0001pt) arc (240:300:4.619pt);\draw[thin,-my] (6.9285pt,4.0001pt) arc (120:60:4.619pt);\draw[thin,-my] (11.5475pt,4.0001pt) arc (60:0:4.619pt);\draw[thin,-my] (13.857pt,0.0pt) arc (0:-60:4.619pt);\draw[thin,-my] (11.5475pt,-4.0001pt) arc (300:240:4.619pt);\draw[thin,-my] (6.9285pt,-4.0001pt) arc (60:120:4.619pt);\fill (0.0pt,16.0002pt) circle (0.4pt);\fill (9.238pt,16.0002pt) circle (0.4pt);\fill (18.476pt,16.0002pt) circle (0.4pt);\fill (23.095pt,8.0001pt) circle (0.4pt);\fill (27.714pt,0.0pt) circle (0.4pt);\fill (23.095pt,-8.0001pt) circle (0.4pt);\fill (18.476pt,-16.0002pt) circle (0.4pt);\fill (9.238pt,-16.0002pt) circle (0.4pt);\fill (0.0pt,-16.0002pt) circle (0.4pt);\fill (-4.619pt,-8.0001pt) circle (0.4pt);\fill (0.0pt,0.0pt) circle (0.4pt);\fill (-4.619pt,8.0001pt) circle (0.4pt);\end{tikzpicture}
}
        \nopar\label{fig:ktt:pipes:impossible}
}
\hspace{0.1cm}
      \subfigure[A pipe enclosing two regions. The pipe is neither thin nor thick.]
        {
\scalebox{1.2}{
\begin{tikzpicture}\fill[rhrhombifill] (13.857pt,24.0003pt) -- (32.333pt,24.0003pt) -- (36.952pt,16.0002pt) -- (27.714pt,0.0pt) -- (18.476pt,0.0pt) -- (13.857pt,-8.0001pt) -- (4.619pt,-8.0001pt) -- (-4.619pt,8.0001pt) -- (0.0pt,16.0002pt) -- (9.238pt,16.0002pt) -- (13.857pt,24.0003pt) -- (18.476pt,16.0002pt) -- (13.857pt,8.0001pt) -- (4.619pt,8.0001pt) -- (9.238pt,0.0pt) -- (13.857pt,8.0001pt) -- (23.095pt,8.0001pt) -- (27.714pt,16.0002pt) -- (18.476pt,16.0002pt) -- cycle;\fill (0.0pt,16.0002pt) circle (0.4pt);\fill (9.238pt,16.0002pt) circle (0.4pt);\fill (18.476pt,16.0002pt) circle (0.4pt);\fill (27.714pt,16.0002pt) circle (0.4pt);\fill (36.952pt,16.0002pt) circle (0.4pt);\fill (13.857pt,24.0003pt) circle (0.4pt);\fill (23.095pt,24.0003pt) circle (0.4pt);\fill (32.333pt,24.0003pt) circle (0.4pt);\fill (-4.619pt,8.0001pt) circle (0.4pt);\fill (4.619pt,8.0001pt) circle (0.4pt);\fill (13.857pt,8.0001pt) circle (0.4pt);\fill (23.095pt,8.0001pt) circle (0.4pt);\fill (32.333pt,8.0001pt) circle (0.4pt);\fill (0.0pt,0.0pt) circle (0.4pt);\fill (9.238pt,0.0pt) circle (0.4pt);\fill (18.476pt,0.0pt) circle (0.4pt);\fill (27.714pt,0.0pt) circle (0.4pt);\fill (4.619pt,-8.0001pt) circle (0.4pt);\fill (13.857pt,-8.0001pt) circle (0.4pt);\draw[thin,-my] (25.4045pt,4.0001pt) arc (300:360:4.619pt);\draw[thin,-my] (27.714pt,8.0001pt) arc (180:120:4.619pt);\draw[thin,-my] (30.0235pt,12.0002pt) arc (300:360:4.619pt);\draw[thin,-my] (32.333pt,16.0002pt) arc (0:60:4.619pt);\draw[thin,-my] (30.0235pt,20.0003pt) arc (60:120:4.619pt);\draw[thin,-my] (25.4045pt,20.0003pt) arc (300:240:4.619pt);\draw[thin,-my] (20.7855pt,20.0003pt) arc (60:120:4.619pt);\draw[thin,-my] (16.1665pt,20.0003pt) arc (120:180:4.619pt);\draw[thin,-my] (13.857pt,16.0002pt) arc (0:-60:4.619pt);\draw[thin,-my] (11.5475pt,12.0002pt) arc (300:240:4.619pt);\draw[thin,-my] (6.9285pt,12.0002pt) arc (60:120:4.619pt);\draw[thin,-my] (2.3095pt,12.0002pt) arc (120:180:4.619pt);\draw[thin,-my] (0.0pt,8.0001pt) arc (180:240:4.619pt);\draw[thin,-my] (2.3095pt,4.0001pt) arc (60:0:4.619pt);\draw[thin,-my] (4.619pt,0.0pt) arc (180:240:4.619pt);\draw[thin,-my] (6.9285pt,-4.0001pt) arc (240:300:4.619pt);\draw[thin,-my] (11.5475pt,-4.0001pt) arc (300:360:4.619pt);\draw[thin,-my] (13.857pt,0.0pt) arc (180:120:4.619pt);\draw[thin,-my] (16.1665pt,4.0001pt) arc (120:60:4.619pt);\draw[thin,-my] (20.7855pt,4.0001pt) arc (240:300:4.619pt);\end{tikzpicture}
}
        \nopar\label{fig:ktt:pipes:alsoimpossible}
}
\caption{Pipes and their enclosed regions. The terms ``thin'' and ``thick'' refer to Definition~\protect\ref{def:ktt:pipetypes}.}
    \nopar\label{fig:ktt:pipes}
\end{center}
\end{figure}
The set of hive triangles from which $c_1$ uses turns is called the \emph{pipe}.
Note that this coincides with the set of hive triangles from which $c_2$ uses turns.
A hive triangle of the pipe is called a \emph{pipe triangle}.
The pipe partitions the plane into several connected components:
The pipe itself, the \emph{outer region} and the \emph{inner regions} enclosed by the pipe.
The \emph{pipe border} is defined to be the set of edges between the regions
and can be divided into the \emph{inner pipe border} and the \emph{outer pipe border}.

\begin{claim}
\label{cla:ktt:pipeborder}
Rhombi whose diagonal lies on the outer pipe border are not $f_1$\dash flat.
Likewise, rhombi whose diagonal lies on the inner pipe border are not $f_2$\dash flat.
\end{claim}
\begin{proof}
 This follows directly from $c_1$ traversing the pipe in counterclockwise direction.
\end{proof}

\begin{claim}
\label{cla:ktt:mustusepipe}
Each $\xi$\dash hive preserving cycle uses a pipe triangle.
\end{claim}
\begin{proof}
First we show that
each $\xi$\dash hive preserving cycle that runs only in the outer region is also $f_1$\dash hive preserving
and each $\xi$\dash hive preserving cycle that runs only in an inner region is also $f_2$\dash hive preserving:
Recall that the flows $\xi$, $f_1$, and $f_2$ only differ by multiples of $c_1$.
So for a rhombus in which both hive triangles are not pipe triangles,
$\xi$\dash flatness, $f_1$\dash flatness, and $f_2$\dash flatness coincide.
The first claim follows with Claim~\ref{cla:ktt:pipeborder}.

For the sake of contradiction,
assume now the existence of a $\xi$\dash hive preserving cycle $c$ that uses no pipe triangle.
Then $c$ runs only in the outer region and is thus $f_1$\dash hive preserving,
or $c$ runs only in the inner region and is thus $f_2$\dash hive preserving.
Key Lemma~\ref{keylem:ktt:secure} ensures that $c$ is $f_1$\dash secure or $f_2$\dash secure,
which is a contradiction to $\LRC \la \mu \nu = 2$.
\end{proof}

We show now that there are severe restrictions on the possible shape of the pipe,
forcing it to have at most one inner region.
The upcoming Claim~\ref{cla:ktt:twopipes} shows that the following two fundamentally different types,
introduced in the next definition, are the only types of pipes that can appear.
\begin{definition}
\label{def:ktt:pipetypes}
We call the pipe \emph{thin} if it has the following property, see Figure~\ref{fig:ktt:pipes:thin}:
Two pipe triangles share a side iff they are direct predecessors or direct successors when traversing~$c_1$.
Additionally, we require that the pipe encloses a single inner region,
see Figure~\ref{fig:ktt:pipes:impossible} and Figure~\ref{fig:ktt:pipes:alsoimpossible} for counterexamples.

We call the pipe \emph{thick} if it has the following property, see Figure~\ref{fig:ktt:pipes:thick}:
There exists a path~$\zeta$ in $\Delta$,
called the \emph{center curve}, such that $\zeta$ has only obtuse angles of $120^\circ$ and
the path $c_1$ runs around $\zeta$ as indicated in Figure~\ref{fig:ktt:pipes:thick}.
Additionally, we require that
two pipe triangles $\Delta_1$ and $\Delta_2$ share a side $k\in E(\Delta)$
iff either $k$ is an edge of $\zeta$ or $\Delta_1$ and $\Delta_2$ are direct
predecessors or direct successors when traversing~$c_\iot$.
The center curve may consist of a single vertex only.
%
%
\end{definition}

\begin{claim}
\label{cla:ktt:twopipes}
The pipe is either thin or thick.
\end{claim}
\begin{proof}
First of all, assume that $c_1$ runs as depicted in
\begin{tikzpictured}\fill[thick,fill=black!20] (0.0pt,0.0pt) -- (4.619pt,8.0001pt) -- (13.857pt,8.0001pt) -- (9.238pt,0.0pt) -- cycle;\fill[thick,fill=black!20] (-4.619pt,8.0001pt) -- (4.619pt,8.0001pt) -- (9.238pt,16.0002pt) -- (0.0pt,16.0002pt) -- cycle;\draw[thin,-my] (4.619pt,0.0pt) arc (180:120:4.619pt);\draw[thin,-my] (6.9285pt,4.0001pt) arc (120:60:4.619pt);\draw[thin,-my] (4.619pt,16.0002pt) arc (0:-60:4.619pt);\draw[thin,-my] (2.3095pt,12.0002pt) arc (300:240:4.619pt);\end{tikzpictured}%
.
Then $c_1$ can be rerouted in $\resf {f_1}$ to $\bar c$ via
\begin{tikzpictured}\fill[thick,fill=black!20] (0.0pt,0.0pt) -- (4.619pt,8.0001pt) -- (13.857pt,8.0001pt) -- (9.238pt,0.0pt) -- cycle;\fill[thick,fill=black!20] (-4.619pt,8.0001pt) -- (4.619pt,8.0001pt) -- (9.238pt,16.0002pt) -- (0.0pt,16.0002pt) -- cycle;\draw[thin,-my] (4.619pt,0.0pt) arc (0:60:4.619pt) arc (240:180:4.619pt) arc (0:60:4.619pt);\end{tikzpictured}%
.
However, $\bar c$ is $f_1$\dash hive preserving
and Key Lemma~\ref{keylem:ktt:secure} implies that $\bar c$ is $f_1$\dash secure,
in contradiction to $\LRC \la \mu \nu = 2$.
This excludes the cases in Figure~\ref{fig:ktt:pipes:alsoimpossible}.

Now suppose that the pipe is not thin, i.e.,
two adjacent pipe triangles are not direct successors when traversing~$c_1$.
It is a simple topological fact that then
there is a rhombus $\rhc$ where $c_1$ or $c_2$ runs like $\rhpoulMlXolrWl$.
We choose $\iot\in\{1,2\}$ such that $c_\iot$ runs like $\rhpoulMlXolrWl$.
So $\rhc$ is not $f_\iot$\dash flat.
It suffices to construct the center curve $\zeta$, which contains the diagonal~$\rhsc$.
To achieve this, we show that ``$c_\iot$ cannot diverge'', which precisely means the following:
If $c_\iot$ uses
$
\begin{tikzpictured}\draw[rhrhombidraw] (0.0pt,16.0002pt) -- (-4.619pt,8.0001pt) -- (0.0pt,0.0pt) -- (4.619pt,8.0001pt) -- cycle;\draw[thin,-my] (-2.3095pt,12.0002pt) arc (240:300:4.619pt) arc (300:360:4.619pt);\draw[thin,-my] (2.3095pt,4.0001pt) arc (60:120:4.619pt);\end{tikzpictured}
$, then $c_\iot$ uses
$
\begin{tikzpictured}\draw[rhrhombidraw] (0.0pt,16.0002pt) -- (-4.619pt,8.0001pt) -- (0.0pt,0.0pt) -- (4.619pt,8.0001pt) -- cycle;\draw[thin,-my] (-2.3095pt,12.0002pt) arc (240:300:4.619pt) arc (300:360:4.619pt);\draw[thin,-my] (11.5475pt,12.0002pt) arc (120:180:4.619pt) arc (0:-60:4.619pt) arc (300:240:4.619pt) arc (60:120:4.619pt);\end{tikzpictured}
$ and $\zeta$ continues with
$
\begin{tikzpictured}\draw[rhrhombidraw] (0.0pt,16.0002pt) -- (-4.619pt,8.0001pt) -- (0.0pt,0.0pt) -- (4.619pt,8.0001pt) -- cycle;\draw[thin,-my] (-2.3095pt,12.0002pt) arc (240:300:4.619pt) arc (300:360:4.619pt);\draw[thin,-my] (11.5475pt,12.0002pt) arc (120:180:4.619pt) arc (0:-60:4.619pt) arc (300:240:4.619pt) arc (60:120:4.619pt);\draw[rhrhombithickside] (-4.619pt,8.0001pt) -- (4.619pt,8.0001pt);\draw[rhrhombithickside] (4.619pt,8.0001pt) -- (9.238pt,16.0002pt);\end{tikzpictured}
$
;
if $c_\iot$ uses
$
\begin{tikzpictured}\draw[rhrhombidraw] (0.0pt,16.0002pt) -- (-4.619pt,8.0001pt) -- (0.0pt,0.0pt) -- (4.619pt,8.0001pt) -- cycle;\draw[thin,-my] (-2.3095pt,12.0002pt) arc (240:300:4.619pt);\draw[thin,-my] (4.619pt,0.0pt) arc (0:60:4.619pt) arc (60:120:4.619pt);\end{tikzpictured}
$, then $c_\iot$ uses
$
\begin{tikzpictured}\draw[rhrhombidraw] (0.0pt,16.0002pt) -- (-4.619pt,8.0001pt) -- (0.0pt,0.0pt) -- (4.619pt,8.0001pt) -- cycle;\draw[thin,-my] (-2.3095pt,12.0002pt) arc (240:300:4.619pt) arc (120:60:4.619pt) arc (60:0:4.619pt) arc (180:240:4.619pt);\draw[thin,-my] (4.619pt,0.0pt) arc (0:60:4.619pt) arc (60:120:4.619pt);\end{tikzpictured}
$
and $\zeta$ continues with
$
\begin{tikzpictured}\draw[rhrhombidraw] (0.0pt,16.0002pt) -- (-4.619pt,8.0001pt) -- (0.0pt,0.0pt) -- (4.619pt,8.0001pt) -- cycle;\draw[thin,-my] (-2.3095pt,12.0002pt) arc (240:300:4.619pt) arc (120:60:4.619pt) arc (60:0:4.619pt) arc (180:240:4.619pt);\draw[thin,-my] (4.619pt,0.0pt) arc (0:60:4.619pt) arc (60:120:4.619pt);\draw[rhrhombithickside] (-4.619pt,8.0001pt) -- (4.619pt,8.0001pt);\draw[rhrhombithickside] (4.619pt,8.0001pt) -- (9.238pt,0.0pt);\end{tikzpictured}
$
; and also both situations rotated by $180^\circ$.

We treat only one case, the others being similar.
So let $c_\iot$ use 
$
\begin{tikzpictured}\draw[rhrhombidraw] (0.0pt,16.0002pt) -- (-4.619pt,8.0001pt) -- (0.0pt,0.0pt) -- (4.619pt,8.0001pt) -- cycle;\draw[thin,-my] (-2.3095pt,12.0002pt) arc (240:300:4.619pt) arc (300:360:4.619pt);\draw[thin,-my] (2.3095pt,4.0001pt) arc (60:120:4.619pt);\end{tikzpictured}
$.
If $\rhrholr$ is not $f_\iot$\dash flat, then
$c_\iot$ can be rerouted via
$
\begin{tikzpictured}\draw[rhrhombidraw] (0.0pt,0.0pt) -- (4.619pt,8.0001pt) -- (0.0pt,16.0002pt) -- (-4.619pt,8.0001pt) -- cycle;\draw[thin,-my] (2.3095pt,4.0001pt) arc (240:180:4.619pt) arc (180:120:4.619pt) arc (300:360:4.619pt);\end{tikzpictured}
$
to give an $f_\iot$\dash hive preserving cycle $\bar c$.
Key Lemma~\ref{keylem:ktt:secure} ensures that $\bar c$ is $f_\iot$\dash secure.
This is a contradiction to $\LRC \la \mu \nu = 2$.
Therefore, $\rhrholr$ is $f_\iot$\dash flat.
If $c_\iot$ uses
$
\begin{tikzpictured}\draw[rhrhombidraw] (0.0pt,16.0002pt) -- (-4.619pt,8.0001pt) -- (0.0pt,0.0pt) -- (4.619pt,8.0001pt) -- cycle;\draw[thin,-my] (-2.3095pt,12.0002pt) arc (240:300:4.619pt) arc (300:360:4.619pt);\draw[thin,-my] (4.619pt,0.0pt) arc (0:60:4.619pt) arc (60:120:4.619pt);\end{tikzpictured}
$, then $c_\iot$ can be rerouted via
$
\begin{tikzpictured}\draw[rhrhombidraw] (0.0pt,16.0002pt) -- (-4.619pt,8.0001pt) -- (0.0pt,0.0pt) -- (4.619pt,8.0001pt) -- cycle;\draw[thin,-my] (4.619pt,0.0pt) arc (0:60:4.619pt) arc (240:180:4.619pt) arc (180:120:4.619pt) arc (300:360:4.619pt);\end{tikzpictured}
$, again in contradiction to $\LRC \la \mu \nu = 2$.
Therefore $c_\iot$ uses
$
\begin{tikzpictured}\draw[rhrhombidraw] (0.0pt,16.0002pt) -- (-4.619pt,8.0001pt) -- (0.0pt,0.0pt) -- (4.619pt,8.0001pt) -- cycle;\draw[thin,-my] (-2.3095pt,12.0002pt) arc (240:300:4.619pt) arc (300:360:4.619pt);\draw[thin,-my] (6.9285pt,4.0001pt) arc (300:240:4.619pt) arc (60:120:4.619pt);\end{tikzpictured}
$.
Since $c_\iot$ uses $\rhpourWl$, it follows that $\rhrhur$ is not $f_\iot$\dash flat.
The hexagon equality (Claim~\ref{cla:BZ}) implies that $\rhrhlr$ is not $f_\iot$\dash flat.
%
If $\rhrhpr$ were not $f_\iot$\dash flat, then $c_\iot$ could be rerouted via
$
\begin{tikzpictured}\draw[rhrhombidraw] (0.0pt,16.0002pt) -- (-4.619pt,8.0001pt) -- (0.0pt,0.0pt) -- (4.619pt,8.0001pt) -- cycle;\draw[thin,-my] (-2.3095pt,12.0002pt) arc (240:300:4.619pt) arc (120:60:4.619pt) arc (60:0:4.619pt) arc (0:-60:4.619pt) arc (300:240:4.619pt) arc (60:120:4.619pt);\end{tikzpictured}
$, which would again result in a contradition to $\LRC \la \mu \nu =2$.
Therefore $\rhrhpr$ is $f_\iot$\dash flat.
But this implies that $c_\iot$ uses
$
\begin{tikzpictured}\draw[rhrhombidraw] (0.0pt,16.0002pt) -- (-4.619pt,8.0001pt) -- (0.0pt,0.0pt) -- (4.619pt,8.0001pt) -- cycle;\draw[thin,-my] (-2.3095pt,12.0002pt) arc (240:300:4.619pt) arc (300:360:4.619pt);\draw[thin,-my] (11.5475pt,12.0002pt) arc (120:180:4.619pt) arc (0:-60:4.619pt) arc (300:240:4.619pt) arc (60:120:4.619pt);\end{tikzpictured}
$.
\end{proof}

If the pipe is thin, then we have exactly one inner region by definition.
If the pipe is thick, then we have no inner region.
The center curve of the thick pipe is \emph{not} considered part of the pipe border.
In fact, the thick pipe is defined to have an empty inner pipe border.

\paragraph{Flatspace sides and the Rerouting Theorem}
We will need a special version of the fundamental Rerouting Theorem, Thm.~4.19 in~\cite{bi:12}.
For this reason we introduce some terminology.

Fix a flow $f \in P$.
An \emph{$f$\dash flatspace side} is defined to be a side of an $f$\dash flatspace,
defined in \cite{bi:12}. But we can equivalently define $f$\dash flatspace sides as follows.

An \emph{$f$\dash flatspace side} $a$ is a maximal line segment
consisting of edges in $\Delta$ such that
\begin{compactenum}
\item each edge $\rhsc$ of~$a$ is the diagonal of a non\dash$f$\dash flat rhombus~$\rhc$.
\item For two successive edges
\begin{tikzpictured}\fill[white] (0.0pt,0.0pt) circle (0.1pt);\fill[white] (9.238pt,16.0002pt) circle (0.1pt);\draw[rhrhombidraw] (-4.619pt,8.0001pt) -- (13.857pt,8.0001pt) ;\fill (-4.619pt,8.0001pt) circle (1pt);\fill (4.619pt,8.0001pt) circle (1pt);\fill (13.857pt,8.0001pt) circle (1pt);\end{tikzpictured}
\ of~$a$
all four rhombi contained in the two incident trapezoids
\begin{tikzpictured}\fill[white] (0.0pt,0.0pt) circle (0.1pt);\fill (-4.619pt,8.0001pt) circle (1pt);\fill (4.619pt,8.0001pt) circle (1pt);\fill (13.857pt,8.0001pt) circle (1pt);\draw[rhrhombidraw] (-4.619pt,8.0001pt) -- (0.0pt,16.0002pt) -- (9.238pt,16.0002pt) -- (13.857pt,8.0001pt) -- cycle;\end{tikzpictured}
\ and
\begin{tikzpictured}\fill (-4.619pt,8.0001pt) circle (1pt);\fill (4.619pt,8.0001pt) circle (1pt);\fill (13.857pt,8.0001pt) circle (1pt);\draw[rhrhombidraw] (-4.619pt,8.0001pt) -- (0.0pt,0.0pt) -- (9.238pt,0.0pt) -- (13.857pt,8.0001pt) -- cycle;\fill[white] (9.238pt,16.0002pt) circle (0.1pt);\end{tikzpictured}
\ are $f$\dash flat.
\end{compactenum}

The pipe has \emph{sides}, which are defined to be maximal line segments of the pipe border.
\emph{Inner pipe sides} are contained in the inner pipe border,
while \emph{outer pipe sides} are contained in the outer pipe border.
\begin{claim}
\label{cla:flatspacesides}
All outer pipe sides can be partitioned into $f_1$\dash flatspace sides
and all inner pipe sides can be partitioned into $f_2$\dash flatspace sides.
\end{claim}
\begin{proof}
Claim~\ref{cla:ktt:pipeborder} implies that
all edges of outer pipe sides belong to $f_1$\dash flatspace sides.
We now prove that the $f_1$\dash flatspace sides do not exceed the pipe sides.
This is easy to see by looking at the following example, which represents the general case.
Consider $c_1$ and the pipe side $a$:
\begin{tikzpicture}\draw[rhrhombidraw] (-4.619pt,8.0001pt) -- (13.857pt,8.0001pt) ;\draw[thin,-my] (18.476pt,8.0001pt) arc (0:-60:4.619pt);\draw[thin,-my] (16.1665pt,4.0001pt) arc (300:240:4.619pt);\draw[thin,-my] (11.5475pt,4.0001pt) arc (60:120:4.619pt);\draw[thin,-my] (6.9285pt,4.0001pt) arc (300:240:4.619pt);\draw[thin,-my] (2.3095pt,4.0001pt) arc (60:120:4.619pt);\draw[thin,-my] (-2.3095pt,4.0001pt) arc (120:180:4.619pt);\end{tikzpicture}%
.
Then the following edges are diagonals of non\dash$f_1$\dash flat rhombi:
\begin{tikzpicture}\draw[rhrhombidraw] (-4.619pt,8.0001pt) -- (13.857pt,8.0001pt) ;\draw[thin,-my] (18.476pt,8.0001pt) arc (0:-60:4.619pt);\draw[thin,-my] (16.1665pt,4.0001pt) arc (300:240:4.619pt);\draw[thin,-my] (11.5475pt,4.0001pt) arc (60:120:4.619pt);\draw[thin,-my] (6.9285pt,4.0001pt) arc (300:240:4.619pt);\draw[thin,-my] (2.3095pt,4.0001pt) arc (60:120:4.619pt);\draw[thin,-my] (-2.3095pt,4.0001pt) arc (120:180:4.619pt);\draw[rhrhombidraw] (13.857pt,8.0001pt) -- (18.476pt,0.0pt) ;\draw[rhrhombidraw] (-4.619pt,8.0001pt) -- (-9.238pt,0.0pt) ;\end{tikzpicture}%
. This proves that $a$ can be partitioned into $f_1$\dash flatspace sides.
\end{proof}
The following result is a direct corollary of \cite[Thm.~4.19]{bi:12}.
\begin{proposition}
\label{pro:rerouting}
If an $f$\dash hive preserving flow~$d$ with zero throughput on the border of~$\Delta$
has nonzero throughput through an edge of an $f$\dash flatspace side~$a$,
then there exists a turncycle~$\bar c$ in $\resf f$ that crosses~$a$.
\end{proposition}

\begin{proof}[Proof of Proposition~\protect\ref{pro:ktt:goal}]
\textbf{Case 1:} We first analyze $\xi$\dash hive preserving cycles that use pipe triangles only.
If the pipe is \emph{thin},
we note that a cycle that uses only pipe triangles necessarily equals $c_1$ or~$c_2$.
Since $c_\iot$ is $f_\iot$\dash hive preserving, the assertion follows.

Now assume that the pipe is \emph{thick} and
assume by way of contradiction that
there is a $\xi$\dash hive preserving cycle $c$ that uses only pipe triangles
but is neither $f_1$\dash hive preserving nor $f_2$\dash hive preserving.
Hence there is a rhombus $\varrho_1$ such that $\s {\varrho_1} {f_1}=0$ and $\s {\varrho_1} {c}<0$.
If we had $\s {\varrho_1} {f_2} = 0$, then $\s {\varrho_1} \xi = 0$, which contradicts $\s {\varrho_1} c <0$
as $c$ is $\xi$\dash hive preserving.
Therefore $\s {\varrho_1} {f_2} > 0$.
Similarly,
by our assumption, there is also
a rhombus $\varrho_2$ with $\s {\varrho_2} {f_2}=0$ and $\s {\varrho_2} {c}<0$.
We know that $c$ must use a negative contribution in $\varrho_1$ and in~$\varrho_2$.
We can now analyze all positions in which $\varrho_1$ and $\varrho_2$ can lie in the thick pipe
and after a detailed but rather straightforward case distinction,
which we omit here, we end up with a contradiction to $\LRC\la\mu\nu=2$.

\textbf{Case 2:} For the sake of contradiction,
we now assume the existence of a $\xi$\dash hive preserving cycle $c$ that does not use pipe triangles only.
Claim~\ref{cla:ktt:mustusepipe} implies that $c$ uses at least one pipe triangle.
Since $c$ does not use pipe triangles only, it follows that
$c$ crosses the pipe border.
Use Claim~\ref{cla:flatspacesides} and
choose $\iot \in \{1,2\}$ such that $c$ crosses the pipe border through an $f_\iot$\dash flatspace side~$a$.
Choose $\tilde x \in \{x,1-x\}$ such that $\xi = f_\iot + \tilde x c_\iot$.
We have $f_\iot+\tilde x c_\iot+\varepsilon c \in P$ for a small $\varepsilon>0$
and hence $d:=\tilde x c_\iot+\varepsilon c$ is $f_\iot$\dash hive preserving.
Proposition~\ref{pro:rerouting} applied to $f_\iot$ and $d$ ensures
the existence of a turncycle in $\resf {f_\iot}$ that crosses the side~$a$.
According to Claim~\ref{cla:flatspacesides}, $a$ is contained in the pipe border.
Let $\bar c$ be a shortest turncycle in $\resf {f_\iot}$ that crosses the pipe border.
According to Proposition~\ref{pro:noreverse} (for turnpaths of length~0)
and Proposition~\ref{pro:special}, $\bar c$ uses no reverse turnvertices
and all self\dash intersections of the curve of $\bar c$ can only happen in $(\bar c,f_\iot)$\dash special rhombi,
defined as follows:

For a turncycle $c$ and a flow $f$ we call a rhombus $\rhc$ \emph{$(c,f)$\dash special}, if
(1)~$\rhc$ is $f$\dash flat and (2)~the diagonal of $\rhc$ is crossed by $c$ via
\begin{tikzpictured}\draw[rhrhombidraw] (0.0pt,0.0pt) -- (-4.619pt,8.0001pt) -- (0.0pt,16.0002pt) -- (4.619pt,8.0001pt) -- cycle;\draw[thin,-my] (-2.3095pt,4.0001pt) --  (2.3095pt,12.0002pt);;\draw[thin,-my] (2.3095pt,4.0001pt) --  (-2.3095pt,12.0002pt);;\end{tikzpictured}%
,
\begin{tikzpictured}\draw[rhrhombidraw] (0.0pt,0.0pt) -- (-4.619pt,8.0001pt) -- (0.0pt,16.0002pt) -- (4.619pt,8.0001pt) -- cycle;\draw[thin,-my] (2.3095pt,4.0001pt) --  (-2.3095pt,12.0002pt);;\draw[thin,-my] (2.3095pt,12.0002pt) --  (-2.3095pt,4.0001pt);;\end{tikzpictured}%
,
\begin{tikzpictured}\draw[rhrhombidraw] (0.0pt,0.0pt) -- (-4.619pt,8.0001pt) -- (0.0pt,16.0002pt) -- (4.619pt,8.0001pt) -- cycle;\draw[thin,-my] (2.3095pt,12.0002pt) --  (-2.3095pt,4.0001pt);;\draw[thin,-my] (-2.3095pt,12.0002pt) --  (2.3095pt,4.0001pt);;\end{tikzpictured}%
, or
\begin{tikzpictured}\draw[rhrhombidraw] (0.0pt,0.0pt) -- (-4.619pt,8.0001pt) -- (0.0pt,16.0002pt) -- (4.619pt,8.0001pt) -- cycle;\draw[thin,-my] (-2.3095pt,12.0002pt) --  (2.3095pt,4.0001pt);;\draw[thin,-my] (-2.3095pt,4.0001pt) --  (2.3095pt,12.0002pt);;\end{tikzpictured}
\ and (3)~$c$ does not use any additional turnvertex in $\rhc$.

First assume that $\bar c$ is ordinary.
But in this case, Key Lemma~\ref{keylem:ktt:secure} implies that $\bar c$ is $f_\iot$\dash secure,
in contradiction to $\LRC \la\mu\nu = 2$.

Now assume that $\bar c$ is not ordinary, i.e., that the curve of $\bar c$ has self\dash intersections.
We will refine Key Lemma~\ref{keylem:ktt:secure} to suit our needs (see Lemma~\ref{lem:ktt:moresecure} below).
To achieve this, we now precisely analyze the situation at the self\dash intersections.
We choose a $(\bar c,f_\iot)$\dash special rhombus $\rhc$ and reroute $\bar c$ in $\rhc$
to obtain a shorter turncycle $c'$ in $\resf {f_\iot}$ as follows:
\begin{align*}
 \tag{\ddag}
\scalebox{1.3}{
\mbox{
\begin{tikzpictured}\draw[rhrhombidraw] (0.0pt,0.0pt) -- (-4.619pt,8.0001pt) -- (0.0pt,16.0002pt) -- (4.619pt,8.0001pt) -- cycle;\draw[thin,-my] (-2.3095pt,4.0001pt) --  (2.3095pt,12.0002pt);;\draw[thin,-my] (2.3095pt,4.0001pt) --  (-2.3095pt,12.0002pt);;\end{tikzpictured}
$\ \rightsquigarrow \rhpollWll$
}
\hspace{0.9cm}
\mbox{
\begin{tikzpictured}\draw[rhrhombidraw] (0.0pt,0.0pt) -- (-4.619pt,8.0001pt) -- (0.0pt,16.0002pt) -- (4.619pt,8.0001pt) -- cycle;\draw[thin,-my] (2.3095pt,4.0001pt) --  (-2.3095pt,12.0002pt);;\draw[thin,-my] (2.3095pt,12.0002pt) --  (-2.3095pt,4.0001pt);;\end{tikzpictured}
$\ \rightsquigarrow \rhpourMr$
}
\hspace{0.9cm}
\mbox{
\begin{tikzpictured}\draw[rhrhombidraw] (0.0pt,0.0pt) -- (-4.619pt,8.0001pt) -- (0.0pt,16.0002pt) -- (4.619pt,8.0001pt) -- cycle;\draw[thin,-my] (2.3095pt,12.0002pt) --  (-2.3095pt,4.0001pt);;\draw[thin,-my] (-2.3095pt,12.0002pt) --  (2.3095pt,4.0001pt);;\end{tikzpictured}
$\ \rightsquigarrow \rhpourMll$
}
\hspace{0.9cm}
\mbox{
\begin{tikzpictured}\draw[rhrhombidraw] (0.0pt,0.0pt) -- (-4.619pt,8.0001pt) -- (0.0pt,16.0002pt) -- (4.619pt,8.0001pt) -- cycle;\draw[thin,-my] (-2.3095pt,12.0002pt) --  (2.3095pt,4.0001pt);;\draw[thin,-my] (-2.3095pt,4.0001pt) --  (2.3095pt,12.0002pt);;\end{tikzpictured}
$\ \rightsquigarrow \rhpollWr$
}
}
\end{align*}
Note that this rerouting is the unique way to reroute in an $f_\iot$\dash flat rhombus $\rhc$
such that the resulting turncycle uses no negative slack contribution in $\rhc$.
Because of the minimal length of~$\bar c$, the turncycle $c'$ does not cross the pipe border.
According to Claim~\ref{cla:ktt:mustusepipe}, $c'$ uses pipe triangles only.
We now show that $c'$ is ordinary.

If the pipe is thin, this is obvious.
Now assume that the pipe is thick.
For the sake of contradiction, assume that $c'$ is not ordinary.
Since $c'$ was obtained by rerouting from $\bar c$, self\dash intersections of the curve of $c'$ can only appear
in $(c',f_\iot)$\dash special rhombi.
As $c'$ uses pipe triangles only,
the diagonals of $(c',f_\iot)$\dash special rhombi are contained in the center curve,
see Figure~\ref{fig:ktt:thickpipeandspecialrhombi}.
\begin{figure}[h] 
  \begin{center}
\scalebox{1.2}{
\begin{tikzpicture}\fill[thick,fill=black!20] (-9.238pt,0.0pt) -- (-13.857pt,8.0001pt) -- (-9.238pt,16.0002pt) -- (9.238pt,16.0002pt) -- (13.857pt,24.0003pt) -- (41.571pt,24.0003pt) -- (46.19pt,16.0002pt) -- (41.571pt,8.0001pt) -- (23.095pt,8.0001pt) -- (18.476pt,0.0pt) -- cycle;\fill[thick,fill=black!40] (0.0pt,0.0pt) -- (-4.619pt,8.0001pt) -- (0.0pt,16.0002pt) -- (4.619pt,8.0001pt) -- cycle;\fill[thick,fill=black!40] (9.238pt,0.0pt) -- (4.619pt,8.0001pt) -- (9.238pt,16.0002pt) -- (13.857pt,8.0001pt) -- cycle;\fill[thick,fill=black!40] (23.095pt,8.0001pt) -- (13.857pt,8.0001pt) -- (9.238pt,16.0002pt) -- (18.476pt,16.0002pt) -- cycle;\fill[thick,fill=black!40] (23.095pt,8.0001pt) -- (18.476pt,16.0002pt) -- (23.095pt,24.0003pt) -- (27.714pt,16.0002pt) -- cycle;\fill[thick,fill=black!40] (32.333pt,8.0001pt) -- (27.714pt,16.0002pt) -- (32.333pt,24.0003pt) -- (36.952pt,16.0002pt) -- cycle;\fill (13.857pt,24.0003pt) circle (0.4pt);\fill (23.095pt,24.0003pt) circle (0.4pt);\fill (32.333pt,24.0003pt) circle (0.4pt);\fill (0.0pt,16.0002pt) circle (0.4pt);\fill (9.238pt,16.0002pt) circle (0.4pt);\fill (18.476pt,16.0002pt) circle (0.4pt);\fill (27.714pt,16.0002pt) circle (0.4pt);\fill (36.952pt,16.0002pt) circle (0.4pt);\fill (-4.619pt,8.0001pt) circle (0.4pt);\fill (4.619pt,8.0001pt) circle (0.4pt);\fill (13.857pt,8.0001pt) circle (0.4pt);\fill (23.095pt,8.0001pt) circle (0.4pt);\fill (0.0pt,0.0pt) circle (0.4pt);\fill (9.238pt,0.0pt) circle (0.4pt);\fill (18.476pt,0.0pt) circle (0.4pt);\fill (32.333pt,8.0001pt) circle (0.4pt);\draw[rhrhombidraw] (9.238pt,16.0002pt) -- (4.619pt,8.0001pt) -- (9.238pt,0.0pt) -- (13.857pt,8.0001pt) -- cycle;\draw[rhrhombidraw] (23.095pt,8.0001pt) -- (18.476pt,16.0002pt) -- (9.238pt,16.0002pt) -- (13.857pt,8.0001pt) -- cycle;\draw[rhrhombidraw] (23.095pt,8.0001pt) -- (18.476pt,16.0002pt) -- (23.095pt,24.0003pt) -- (27.714pt,16.0002pt) -- cycle;\fill (-9.238pt,0.0pt) circle (0.4pt);\fill (-13.857pt,8.0001pt) circle (0.4pt);\fill (-9.238pt,16.0002pt) circle (0.4pt);\fill (41.571pt,8.0001pt) circle (0.4pt);\fill (46.19pt,16.0002pt) circle (0.4pt);\fill (41.571pt,24.0003pt) circle (0.4pt);\draw[rhrhombidraw] (32.333pt,8.0001pt) -- (27.714pt,16.0002pt) -- (32.333pt,24.0003pt) -- (36.952pt,16.0002pt) -- cycle;\draw[rhrhombidraw] (0.0pt,0.0pt) -- (-4.619pt,8.0001pt) -- (0.0pt,16.0002pt) -- (4.619pt,8.0001pt) -- cycle;\end{tikzpicture}
}
    \caption{An example of a thick pipe with all possible special rhombi highlighted by a darker shading.}
    \nopar\label{fig:ktt:thickpipeandspecialrhombi}
  \end{center}
\end{figure}
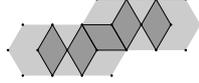
But rerouting iteratively at these rhombi finally results in an ordinary
$f_\iot$\dash hive preserving turncycle, which is shorter than~$c_\iot$.
According to Key Lemma~\ref{keylem:ktt:secure},
this turncycle is $f_\iot$\dash secure, in contradiction to $\LRC \la \mu \nu =2$.

Hence $c'$ is ordinary in both pipe cases.
Since $c'$ is $f_\iot$\dash hive preserving,
it is $f_\iot$\dash secure (Key Lemma~\ref{keylem:ktt:secure}).
The fact $\LRC\la\mu\nu=2$ implies that $c'$ coincides with $c_\iot$.
Thus $\bar c$ reroutes to $c_\iot$, no matter in which $(\bar c,f_\iot)$\dash special rhombus we reroute.


We can see that $\bar c$ is $f_\iot$\dash secure, in contradiction to $\LRC \la\mu\nu = 2$,
by using the following Lemma~\ref{lem:ktt:moresecure}.
\begin{lemma}
\label{lem:ktt:moresecure}
Let $\iot \in \{1,2\}.$
 If an $f_\iot$\dash hive preserving turncycle $c$ has no reverse turnvertices and
 all self\dash intersections of the curve of $c$
 occur in $(c,f_\iot)$\dash special rhombi in which $c$ reroutes to $c_\iot$ when applying~$(\ddag)$,
then $c$ is $f_\iot$\dash secure.
\end{lemma}
This finishes the proof of Proposition~\ref{pro:ktt:goal}.
\end{proof}
Lemma~\ref{lem:ktt:moresecure} is a refined version of Key Lemma~\ref{keylem:ktt:secure}.
We postpone its proof to Subsection~\ref{subsec:ktt:moresecure},
because it is based on ideas of the following subsection.
\subsection{Proof of Key Lemma~\protect\ref{keylem:ktt:secure}}
\label{subsec:ktt:secure}
\begin{proof}

We will show that
\begin{equation}
\tag{$\ast$}
\begin{minipage}{11cm}
for each ordinary turncycle in $\resf {f_\iot}$ that is not $f_\iot$\dash secure
there exist two distinct shorter ordinary turncycles in $\resf {f_\iot}$.
\end{minipage}
\end{equation}
If there exists an ordinary turncycle in $\resf {f_\iot}$ that is not $f_\iot$\dash secure,
then take one of minimal length. It follows from $(\ast)$
that there exist two distinct ordinary $f_\iot$\dash secure turncycles.
This is a contradiction to $\LRC \la \mu \nu = 2$.

It remains to prove $(\ast)$.
Recall that ordinary turncycles in $\res$ are in bijection to proper cycles in~$G$.
Let $c$ be an ordinary turncycle on $\resf {f_\iot}$ that is not $f_\iot$\dash secure.
By Definition~\ref{def:fhivepres} and Proposition~\ref{pro:secureneighbors}
we have $f_\iot+c\notin P$ and hence there exists
$0<\varepsilon<1$ such that $f_\iot+\varepsilon c \in P$.
According to Claim~\ref{cla:smallabsoluteslack} we have $\s \varrho c \in \{-2,-1,0,1,2\}$
for all rhombi~$\varrho$.
Hence there exists a rhombus $\varrho$ with $\s \varrho {f_\iot} = 1$ and $\s \varrho c = -2$.
We call such rhombi \emph{bad}.
Let $\rhc$ be a bad rhombus. Then $c$ uses $\rhpoulMlXolrWl$ by Proposition~\ref{pro:slackcalc}.

We begin by analyzing a very special case:
If all four rhombi $\rhrhoul$, $\rhrhour$, $\rhrholl$, and $\rhrholr$ are
not $f_\iot$\dash flat, then $c$ can be rerouted twice:
Once via $\rhpoulMrr$ and once via $\rhpolrWrr$,
which results in two ordinary turncycles in $\resf {f_\iot}$.
This proves $(\ast)$ in this special case.
In the more general case, we prove the following:
\begin{equation}
\tag{\ensuremath{\ast}\ensuremath{\ast}}
\begin{minipage}{11cm}
In each bad rhombus~$\rhc$ the ordinary turncycle~$c$ in $\resf {f_\iot}$ can be rerouted via~$\rhpoulMrr$ or
one of the three rhombi $\rhrhul$, $\rhrhpl$ or $\rhrhll$ is bad
such that $c$ uses 
\begin{tikzpictured}\draw[rhrhombidraw] (0.0pt,0.0pt) -- (4.619pt,8.0001pt) -- (0.0pt,16.0002pt) -- (-4.619pt,8.0001pt) -- cycle;\draw[thin,-my] (-4.619pt,16.0002pt) arc (180:240:4.619pt) arc (240:300:4.619pt);\draw[thin,-my] (2.3095pt,4.0001pt) arc (60:120:4.619pt) arc (300:240:4.619pt) arc (240:180:4.619pt) arc (0:60:4.619pt);\end{tikzpictured}%
,
\begin{tikzpictured}\draw[rhrhombidraw] (0.0pt,0.0pt) -- (4.619pt,8.0001pt) -- (0.0pt,16.0002pt) -- (-4.619pt,8.0001pt) -- cycle;\draw[thin,-my] (-11.5475pt,12.0002pt) arc (240:300:4.619pt) arc (120:60:4.619pt) arc (240:300:4.619pt);\draw[thin,-my] (2.3095pt,4.0001pt) arc (60:120:4.619pt) arc (300:240:4.619pt) arc (60:120:4.619pt);\end{tikzpictured}%
, or
\begin{tikzpictured}\draw[rhrhombidraw] (0.0pt,0.0pt) -- (4.619pt,8.0001pt) -- (0.0pt,16.0002pt) -- (-4.619pt,8.0001pt) -- cycle;\draw[thin,-my] (-11.5475pt,4.0001pt) arc (300:360:4.619pt) arc (180:120:4.619pt) arc (120:60:4.619pt) arc (240:300:4.619pt);\draw[thin,-my] (2.3095pt,4.0001pt) arc (60:120:4.619pt) arc (120:180:4.619pt);\end{tikzpictured}%
, respectively.
Additionally, in each bad rhombus~$\rhc$ the ordinary turncycle~$c$ in $\resf {f_\iot}$ can be rerouted via~$\rhpolrWrr$ or
one of the three rhombi $\rhrhur$, $\rhrhpr$ or $\rhrhlr$ is bad
such that $c$ uses 
\begin{tikzpictured}\draw[rhrhombidraw] (0.0pt,0.0pt) -- (4.619pt,8.0001pt) -- (0.0pt,16.0002pt) -- (-4.619pt,8.0001pt) -- cycle;\draw[thin,-my] (-2.3095pt,12.0002pt) arc (240:300:4.619pt) arc (300:360:4.619pt);\draw[thin,-my] (11.5475pt,12.0002pt) arc (120:180:4.619pt) arc (0:-60:4.619pt) arc (300:240:4.619pt) arc (60:120:4.619pt);\end{tikzpictured}%
, 
\begin{tikzpictured}\draw[rhrhombidraw] (0.0pt,0.0pt) -- (4.619pt,8.0001pt) -- (0.0pt,16.0002pt) -- (-4.619pt,8.0001pt) -- cycle;\draw[thin,-my] (-2.3095pt,12.0002pt) arc (240:300:4.619pt) arc (120:60:4.619pt) arc (240:300:4.619pt);\draw[thin,-my] (11.5475pt,4.0001pt) arc (60:120:4.619pt) arc (300:240:4.619pt) arc (60:120:4.619pt);\end{tikzpictured}%
, or 
\begin{tikzpictured}\draw[rhrhombidraw] (0.0pt,0.0pt) -- (4.619pt,8.0001pt) -- (0.0pt,16.0002pt) -- (-4.619pt,8.0001pt) -- cycle;\draw[thin,-my] (-2.3095pt,12.0002pt) arc (240:300:4.619pt) arc (120:60:4.619pt) arc (60:0:4.619pt) arc (180:240:4.619pt);\draw[thin,-my] (4.619pt,0.0pt) arc (0:60:4.619pt) arc (60:120:4.619pt);\end{tikzpictured}%
, respectively.
\end{minipage}
\end{equation}
We first show that (\ensuremath{\ast}\ensuremath{\ast}) implies $(\ast)$.
According to (\ensuremath{\ast}\ensuremath{\ast}),
each bad rhombus that cannot be rerouted at the left
has another bad rhombus located at its left, and analogously for its right side.
We continue finding bad rhombi in this manner and obtain a set of adjacent bad rhombi,
which we call the \emph{chain}.
The chain has two endings at which $c$ can be rerouted to shorter ordinary turncycles.
Hence $(\ast)$ follows.

We now show that (\ensuremath{\ast}\ensuremath{\ast}) holds.
First, we precisely characterize the situations in which $c$
can be rerouted in $\resf {f_\iot}$ via $\rhpoulMrr$:
This is exactly the case when
\begin{center}
both \quad $\Big($ $\rhrhoul$ is not $f_\iot$\dash flat or $c$ uses $\rhppulWl$ $\Big)$\quad and \quad
$\Big($ $\rhrholl$ is not $f_\iot$\dash flat or $c$ uses $\rhpollMl$ $\Big)$.
\end{center}
Now assume that $c$ cannot be rerouted via $\rhpoulMrr$, i.e.,
\begin{center}
\quad $\Big($ $\rhrhoul$ is $f_\iot$\dash flat and $c$ uses $\rhpulWr$ $\Big)$\quad or \quad
$\Big($ $\rhrholl$ is $f_\iot$\dash flat and $c$ uses $\rhpollMr$ $\Big)$.
\end{center}
We demonstrate how to prove (\ensuremath{\ast}\ensuremath{\ast})
in the following exemplary case, all others being similar:
Let $\rhrhoul$ be $f_\iot$\dash flat with $c$ using $\rhpulWr$
and let $\rhrholl$ be $f_\iot$\dash flat with $c$ using $\rhpollMl$.
The hexagon equality (Claim~\ref{cla:BZ}) applied twice
implies that we have either
$\s \rhrhpl {f_\iot} = 0$, $\s \rhrhll {f_\iot} = 1$, $\s \rhrhul {f_\iot} = 1$
or
$\s \rhrhpl {f_\iot} = 1$, $\s \rhrhll {f_\iot} = 0$, $\s \rhrhul {f_\iot} = 0$.
The latter is impossible, because $c$ uses $\rhpollMl$.
The fact that $c$ uses no negative contributions in $f_\iot$\dash flat rhombi
and that $c$ is ordinary leads to $c$ running as desired:
\begin{tikzpictured}\draw[rhrhombidraw] (0.0pt,0.0pt) -- (4.619pt,8.0001pt) -- (0.0pt,16.0002pt) -- (-4.619pt,8.0001pt) -- cycle;\draw[thin,-my] (-11.5475pt,4.0001pt) arc (300:360:4.619pt) arc (180:120:4.619pt) arc (120:60:4.619pt) arc (240:300:4.619pt);\draw[thin,-my] (2.3095pt,4.0001pt) arc (60:120:4.619pt) arc (120:180:4.619pt);\end{tikzpictured}
\ with $\rhrhll$ being bad.
All other cases are similar.
\end{proof}

\subsection{Proof of Lemma~\protect\ref{lem:ktt:moresecure}}
\label{subsec:ktt:moresecure}
We now complete the proof of Proposition~\ref{pro:ktt:goal} by proving Lemma~\ref{lem:ktt:moresecure}.
\begin{proof}
The proof is completely analogous to the proof of Key Lemma~\ref{keylem:ktt:secure}.
We only highlight the technical differences here.
Since, in contrast to Key Lemma~\ref{keylem:ktt:secure}, we are not dealing with ordinary turncycles only,
we make the following definition.
\begin{definition}
Let $\iot \in \{1,2\}$.
 If an $f_\iot$\dash hive preserving turncycle $c$ has no reverse turnvertices and
 all self\dash intersections of the curve of $c$
 occur in $(c,f_\iot)$\dash special rhombi in which $c$ reroutes to $c_\iot$ when applying~$(\ddag)$,
then $c$ is called \emph{almost ordinary}.
\end{definition}
Note that this notion depends on~$\iot$, which we think of being fixed in the following.
In analogy to Key Lemma~\ref{keylem:ktt:secure}, Lemma~\ref{lem:ktt:moresecure} now reads as follows:
\begin{center}
Every almost ordinary turncycle is $f_\iot$\dash secure.
\end{center}
We will show the following statement:
\begin{equation}
\tag{$\ast$'}
\begin{minipage}{11cm}
For each almost ordinary turncycle in $\resf {f_\iot}$ that is not $f_\iot$\dash secure
there exist two distinct shorter almost ordinary turncycles in $\resf {f_\iot}$.
\end{minipage}
\end{equation}
As in the proof of Key Lemma~\ref{keylem:ktt:secure},
in order to prove Lemma~\ref{lem:ktt:moresecure} it suffices to show ($\ast$').
So fix an almost ordinary turncycle~$c$.
First of all, since $c$ is not necessarily ordinary, we have $\s \rhc {c} \in \{-4,-3,\ldots,3,4\}$ for all rhombi $\rhc$.
But we show now that $\s \rhc {c} \geq -2$ for all rhombi $\rhc$.

We begin by showing that there is no rhombus~$\rhc$ with $\s \rhc {c} = -4$.
Recall that $\s \rhc {c} = \rhaourM(c)+\rhaollW(c)$.
If $\s \rhc {c} = -4$, then
both $\rhrhour$ and $\rhrholl$ are $(c,f_\iot)$\dash special.
But since $\s \rhc {c} = \rhaoulW(c)+\rhaolrM(c)$,
it follows that $\rhrhoul$ and $\rhrholr$ are $(c,f_\iot)$\dash special as well.
This is a contradiction to the fact that $(c,f_\iot)$\dash special rhombi do not overlap
(true by definition, cp.~Proposition~\ref{pro:special}(\ref{pro:special:nooverlap})).
Hence $\s \rhc {c} \geq -3$ for all rhombi $\rhc$.

We show next that $\s \rhc {c} \neq -3$.
Assume the contrary.
W.l.o.g.\ let $\rhaourM(c)=-2$ and $\rhaollW(c)=-1$, the other case being the same, just rotated by $180^\circ$.
Then $c$ is bound to use
$%
\begin{tikzpictured}\draw[rhrhombidraw] (0.0pt,0.0pt) -- (4.619pt,8.0001pt) -- (0.0pt,16.0002pt) -- (-4.619pt,8.0001pt) -- cycle;\draw[thin,-my] (0.0pt,8.0001pt) --  (4.619pt,16.0002pt);;\draw[thin,-my] (-2.3095pt,12.0002pt) --  (6.9285pt,12.0002pt);;\draw[thin,-my] (4.619pt,0.0pt) --  (0.0pt,8.0001pt);;\draw[thin,-my] (6.9285pt,4.0001pt) --  (-2.3095pt,4.0001pt);;\end{tikzpictured}%
$.
But, according to our assumption that in $(c,f_\iot)$\dash special rhombi $c$ reroutes to $c_\iot$,  this means that $c_\iot$ uses
$%
\begin{tikzpictured}\draw[rhrhombidraw] (0.0pt,0.0pt) -- (4.619pt,8.0001pt) -- (0.0pt,16.0002pt) -- (-4.619pt,8.0001pt) -- cycle;\draw[thin,-my] (-2.3095pt,12.0002pt) arc (240:300:4.619pt) arc (300:360:4.619pt);\draw[thin,-my] (4.619pt,0.0pt) arc (0:60:4.619pt) arc (60:120:4.619pt);\end{tikzpictured}%
$, which is a contradiction to Claim~\ref{cla:ktt:twopipes}.

So it follows that $\s \rhc {c} \geq -2$ for all rhombi.
This is exactly the same situation as in Key Lemma~\ref{keylem:ktt:secure}.
As in the proof of Key Lemma~\ref{keylem:ktt:secure},
we call rhombi $\varrho$ \emph{bad} if
$\s \varrho {c} = -2$ and $\s \varrho {f_\iot} = 1$
There exists a bad rhombus $\rhc$,
since $c$ is assumed to be not $f_\iot$\dash secure.
But here is a technical difference:
Unlike in Key Lemma~\ref{keylem:ktt:secure},
$c$~does not necessarily use exactly the turnvertices $\rhpoulMlXolrWl$ in $\rhc$, but
$c$ uses exactly one of the following sets of turnvertices in~$\rhc$:
$\rhpoulMlXolrWl$,
\begin{tikzpictured}\draw[rhrhombidraw] (0.0pt,0.0pt) -- (4.619pt,8.0001pt) -- (0.0pt,16.0002pt) -- (-4.619pt,8.0001pt) -- cycle;\draw[thin,-my] (-6.9285pt,12.0002pt) --  (2.3095pt,12.0002pt);;\draw[thin,-my] (-4.619pt,16.0002pt) --  (0.0pt,8.0001pt);;\draw[thin,-my] (0.0pt,8.0001pt) arc (0:-60:4.619pt);\end{tikzpictured}%
,
\begin{tikzpictured}\draw[rhrhombidraw] (0.0pt,0.0pt) -- (4.619pt,8.0001pt) -- (0.0pt,16.0002pt) -- (-4.619pt,8.0001pt) -- cycle;\draw[thin,-my] (0.0pt,8.0001pt) --  (4.619pt,16.0002pt);;\draw[thin,-my] (-2.3095pt,12.0002pt) --  (6.9285pt,12.0002pt);;\draw[thin,-my] (2.3095pt,4.0001pt) arc (240:180:4.619pt);\end{tikzpictured}%
,
\begin{tikzpictured}\draw[rhrhombidraw] (0.0pt,0.0pt) -- (4.619pt,8.0001pt) -- (0.0pt,16.0002pt) -- (-4.619pt,8.0001pt) -- cycle;\draw[thin,-my] (4.619pt,0.0pt) --  (0.0pt,8.0001pt);;\draw[thin,-my] (6.9285pt,4.0001pt) --  (-2.3095pt,4.0001pt);;\draw[thin,-my] (0.0pt,8.0001pt) arc (180:120:4.619pt);\end{tikzpictured}%
,
\begin{tikzpictured}\draw[rhrhombidraw] (0.0pt,0.0pt) -- (4.619pt,8.0001pt) -- (0.0pt,16.0002pt) -- (-4.619pt,8.0001pt) -- cycle;\draw[thin,-my] (2.3095pt,4.0001pt) --  (-6.9285pt,4.0001pt);;\draw[thin,-my] (0.0pt,8.0001pt) --  (-4.619pt,0.0pt);;\draw[thin,-my] (-2.3095pt,12.0002pt) arc (60:0:4.619pt);\end{tikzpictured}%
,
\begin{tikzpictured}\draw[rhrhombidraw] (0.0pt,0.0pt) -- (4.619pt,8.0001pt) -- (0.0pt,16.0002pt) -- (-4.619pt,8.0001pt) -- cycle;\draw[thin,-my] (-4.619pt,16.0002pt) --  (0.0pt,8.0001pt);;\draw[thin,-my] (-6.9285pt,12.0002pt) --  (2.3095pt,12.0002pt);;\draw[thin,-my] (6.9285pt,4.0001pt) --  (-2.3095pt,4.0001pt);;\draw[thin,-my] (0.0pt,8.0001pt) --  (4.619pt,0.0pt);;\end{tikzpictured}%
,
\begin{tikzpictured}\draw[rhrhombidraw] (0.0pt,0.0pt) -- (4.619pt,8.0001pt) -- (0.0pt,16.0002pt) -- (-4.619pt,8.0001pt) -- cycle;\draw[thin,-my] (0.0pt,8.0001pt) --  (4.619pt,16.0002pt);;\draw[thin,-my] (-2.3095pt,12.0002pt) --  (6.9285pt,12.0002pt);;\draw[thin,-my] (-4.619pt,0.0pt) --  (0.0pt,8.0001pt);;\draw[thin,-my] (2.3095pt,4.0001pt) --  (-6.9285pt,4.0001pt);;\end{tikzpictured}%
,
\begin{tikzpictured}\draw[rhrhombidraw] (0.0pt,0.0pt) -- (4.619pt,8.0001pt) -- (0.0pt,16.0002pt) -- (-4.619pt,8.0001pt) -- cycle;\draw[thin,-my] (0.0pt,8.0001pt) --  (-4.619pt,16.0002pt);;\draw[thin,-my] (-6.9285pt,12.0002pt) --  (2.3095pt,12.0002pt);;\draw[thin,-my] (4.619pt,0.0pt) --  (0.0pt,8.0001pt);;\draw[thin,-my] (6.9285pt,4.0001pt) --  (-2.3095pt,4.0001pt);;\end{tikzpictured}%
, or
\begin{tikzpictured}\draw[rhrhombidraw] (0.0pt,0.0pt) -- (4.619pt,8.0001pt) -- (0.0pt,16.0002pt) -- (-4.619pt,8.0001pt) -- cycle;\draw[thin,-my] (4.619pt,16.0002pt) --  (0.0pt,8.0001pt);;\draw[thin,-my] (-2.3095pt,12.0002pt) --  (6.9285pt,12.0002pt);;\draw[thin,-my] (0.0pt,8.0001pt) --  (-4.619pt,0.0pt);;\draw[thin,-my] (2.3095pt,4.0001pt) --  (-6.9285pt,4.0001pt);;\end{tikzpictured}%
.

We now show that only the first 5 cases can appear,
because the last 4 cases cases are in contradiction to the fact that $c$
reroutes to~$c_\iot$ in $(c,f_\iot)$\dash special rhombi:
In the case where $c$ uses
\begin{tikzpictured}\draw[rhrhombidraw] (0.0pt,0.0pt) -- (4.619pt,8.0001pt) -- (0.0pt,16.0002pt) -- (-4.619pt,8.0001pt) -- cycle;\draw[thin,-my] (-4.619pt,16.0002pt) --  (0.0pt,8.0001pt);;\draw[thin,-my] (-6.9285pt,12.0002pt) --  (2.3095pt,12.0002pt);;\draw[thin,-my] (6.9285pt,4.0001pt) --  (-2.3095pt,4.0001pt);;\draw[thin,-my] (0.0pt,8.0001pt) --  (4.619pt,0.0pt);;\end{tikzpictured}%
, then the cycle $c_\iot$ on~$G$ uses
\begin{tikzpictured}\draw[rhrhombidraw] (0.0pt,0.0pt) -- (4.619pt,8.0001pt) -- (0.0pt,16.0002pt) -- (-4.619pt,8.0001pt) -- cycle;\draw[thin,-my] (-4.619pt,16.0002pt) arc (180:240:4.619pt) arc (240:300:4.619pt);\draw[thin,-my] (0.0pt,8.0001pt) arc (0:-60:4.619pt);\end{tikzpictured}%
, which is impossible. The other three cases are treated similarly.

We want to prove that there exists a chain of adjacent bad rhombi as we did in Key Lemma~\ref{keylem:ktt:secure}.
Analogously to Key Lemma~\ref{keylem:ktt:secure}, we can prove the following statement, which is more technical than (\ensuremath{\ast}\ensuremath{\ast}):
\begin{equation}
\tag{\ensuremath{\ast}\ensuremath{\ast}'}
\begin{minipage}{11cm}
\textbf{Case $\rhpoulMlXolrWl$:}
For each bad rhombus~$\rhc$ where the turncycle~$c$ in $\resf {f_\iot}$ uses only the turnvertices $\rhpoulMlXolrWl$,
one of the following holds:
(1) $c$ can be rerouted via~$\rhpoulMrr$,
\begin{tikzpictured}\draw[rhrhombidraw] (0.0pt,0.0pt) -- (4.619pt,8.0001pt) -- (0.0pt,16.0002pt) -- (-4.619pt,8.0001pt) -- cycle;\draw[thin,-my] (-11.5475pt,4.0001pt) arc (300:360:4.619pt) arc (0:60:4.619pt);\end{tikzpictured}
,\ 
\begin{tikzpictured}\draw[rhrhombidraw] (0.0pt,0.0pt) -- (4.619pt,8.0001pt) -- (0.0pt,16.0002pt) -- (-4.619pt,8.0001pt) -- cycle;\draw[thin,-my] (-4.619pt,0.0pt) arc (0:60:4.619pt) arc (60:120:4.619pt);\end{tikzpictured}
,\ or
\begin{tikzpictured}\draw[rhrhombidraw] (0.0pt,0.0pt) -- (4.619pt,8.0001pt) -- (0.0pt,16.0002pt) -- (-4.619pt,8.0001pt) -- cycle;\draw[thin,-my] (-11.5475pt,12.0002pt) arc (240:300:4.619pt) arc (300:360:4.619pt);\end{tikzpictured}
, or (2)
one of the three rhombi $\rhrhul$, $\rhrhpl$ or $\rhrhll$ is bad
such that $c$ uses 
\begin{tikzpictured}\draw[rhrhombidraw] (0.0pt,0.0pt) -- (4.619pt,8.0001pt) -- (0.0pt,16.0002pt) -- (-4.619pt,8.0001pt) -- cycle;\draw[thin,-my] (-4.619pt,16.0002pt) arc (180:240:4.619pt) arc (240:300:4.619pt);\draw[thin,-my] (2.3095pt,4.0001pt) arc (60:120:4.619pt) arc (300:240:4.619pt) arc (240:180:4.619pt) arc (0:60:4.619pt);\end{tikzpictured}%
,
\begin{tikzpictured}\draw[rhrhombidraw] (0.0pt,0.0pt) -- (4.619pt,8.0001pt) -- (0.0pt,16.0002pt) -- (-4.619pt,8.0001pt) -- cycle;\draw[thin,-my] (-11.5475pt,12.0002pt) arc (240:300:4.619pt) arc (120:60:4.619pt) arc (240:300:4.619pt);\draw[thin,-my] (2.3095pt,4.0001pt) arc (60:120:4.619pt) arc (300:240:4.619pt) arc (60:120:4.619pt);\end{tikzpictured}%
, or
\begin{tikzpictured}\draw[rhrhombidraw] (0.0pt,0.0pt) -- (4.619pt,8.0001pt) -- (0.0pt,16.0002pt) -- (-4.619pt,8.0001pt) -- cycle;\draw[thin,-my] (-11.5475pt,4.0001pt) arc (300:360:4.619pt) arc (180:120:4.619pt) arc (120:60:4.619pt) arc (240:300:4.619pt);\draw[thin,-my] (2.3095pt,4.0001pt) arc (60:120:4.619pt) arc (120:180:4.619pt);\end{tikzpictured}%
, respectively.
Additionally, as in (\ensuremath{\ast}\ensuremath{\ast}), this holds for the situation rotated by $180^\circ$.

\textbf{Case
\begin{tikzpictured}\draw[rhrhombidraw] (0.0pt,0.0pt) -- (4.619pt,8.0001pt) -- (0.0pt,16.0002pt) -- (-4.619pt,8.0001pt) -- cycle;\draw[thin,-my] (-6.9285pt,12.0002pt) --  (2.3095pt,12.0002pt);;\draw[thin,-my] (-4.619pt,16.0002pt) --  (0.0pt,8.0001pt);;\draw[thin,-my] (0.0pt,8.0001pt) arc (0:-60:4.619pt);\end{tikzpictured}%
:}
For each bad rhombus~$\rhc$ where $c$ uses exactly the turnvertices
\begin{tikzpictured}\draw[rhrhombidraw] (0.0pt,0.0pt) -- (4.619pt,8.0001pt) -- (0.0pt,16.0002pt) -- (-4.619pt,8.0001pt) -- cycle;\draw[thin,-my] (-6.9285pt,12.0002pt) --  (2.3095pt,12.0002pt);;\draw[thin,-my] (-4.619pt,16.0002pt) --  (0.0pt,8.0001pt);;\draw[thin,-my] (0.0pt,8.0001pt) arc (0:-60:4.619pt);\end{tikzpictured}%
\ in $\rhc$, the turncycle $c$ can be rerouted via 
\begin{tikzpictured}\draw[rhrhombidraw] (0.0pt,0.0pt) -- (4.619pt,8.0001pt) -- (0.0pt,16.0002pt) -- (-4.619pt,8.0001pt) -- cycle;\draw[thin,-my] (-4.619pt,16.0002pt) arc (180:240:4.619pt) arc (240:300:4.619pt);\end{tikzpictured}
\ and $\rhrhll$ or $\rhrhpl$ is bad.

\textbf{Remaining cases:}
Results that are analogous to the second case hold for $c$ using
\begin{tikzpictured}\draw[rhrhombidraw] (0.0pt,0.0pt) -- (4.619pt,8.0001pt) -- (0.0pt,16.0002pt) -- (-4.619pt,8.0001pt) -- cycle;\draw[thin,-my] (0.0pt,8.0001pt) --  (4.619pt,16.0002pt);;\draw[thin,-my] (-2.3095pt,12.0002pt) --  (6.9285pt,12.0002pt);;\draw[thin,-my] (2.3095pt,4.0001pt) arc (240:180:4.619pt);\end{tikzpictured}%
,
\begin{tikzpictured}\draw[rhrhombidraw] (0.0pt,0.0pt) -- (4.619pt,8.0001pt) -- (0.0pt,16.0002pt) -- (-4.619pt,8.0001pt) -- cycle;\draw[thin,-my] (4.619pt,0.0pt) --  (0.0pt,8.0001pt);;\draw[thin,-my] (6.9285pt,4.0001pt) --  (-2.3095pt,4.0001pt);;\draw[thin,-my] (0.0pt,8.0001pt) arc (180:120:4.619pt);\end{tikzpictured}%
, or
\begin{tikzpictured}\draw[rhrhombidraw] (0.0pt,0.0pt) -- (4.619pt,8.0001pt) -- (0.0pt,16.0002pt) -- (-4.619pt,8.0001pt) -- cycle;\draw[thin,-my] (2.3095pt,4.0001pt) --  (-6.9285pt,4.0001pt);;\draw[thin,-my] (0.0pt,8.0001pt) --  (-4.619pt,0.0pt);;\draw[thin,-my] (-2.3095pt,12.0002pt) arc (60:0:4.619pt);\end{tikzpictured}%
, respectively.
\end{minipage}
\end{equation}
We remark that the strange reroutings 
\begin{tikzpictured}\draw[rhrhombidraw] (0.0pt,0.0pt) -- (4.619pt,8.0001pt) -- (0.0pt,16.0002pt) -- (-4.619pt,8.0001pt) -- cycle;\draw[thin,-my] (-11.5475pt,4.0001pt) arc (300:360:4.619pt) arc (0:60:4.619pt);\end{tikzpictured}
,\ 
\begin{tikzpictured}\draw[rhrhombidraw] (0.0pt,0.0pt) -- (4.619pt,8.0001pt) -- (0.0pt,16.0002pt) -- (-4.619pt,8.0001pt) -- cycle;\draw[thin,-my] (-4.619pt,0.0pt) arc (0:60:4.619pt) arc (60:120:4.619pt);\end{tikzpictured}
,\ and
\begin{tikzpictured}\draw[rhrhombidraw] (0.0pt,0.0pt) -- (4.619pt,8.0001pt) -- (0.0pt,16.0002pt) -- (-4.619pt,8.0001pt) -- cycle;\draw[thin,-my] (-11.5475pt,12.0002pt) arc (240:300:4.619pt) arc (300:360:4.619pt);\end{tikzpictured}
\ occur in the cases
where $c$ uses
\begin{tikzpictured}\draw[rhrhombidraw] (0.0pt,0.0pt) -- (4.619pt,8.0001pt) -- (0.0pt,16.0002pt) -- (-4.619pt,8.0001pt) -- cycle;\draw[thin,-my] (2.3095pt,4.0001pt) arc (60:120:4.619pt) arc (300:240:4.619pt);\draw[thin,-my] (-6.9285pt,4.0001pt) --  (-11.5475pt,12.0002pt);;\draw[thin,-my] (-11.5475pt,4.0001pt) --  (-6.9285pt,12.0002pt);;\draw[thin,-my] (-6.9285pt,12.0002pt) arc (120:60:4.619pt) arc (240:300:4.619pt);\end{tikzpictured}
,\ 
\begin{tikzpictured}\draw[rhrhombidraw] (0.0pt,0.0pt) -- (4.619pt,8.0001pt) -- (0.0pt,16.0002pt) -- (-4.619pt,8.0001pt) -- cycle;\draw[thin,-my] (-9.238pt,8.0001pt) arc (180:120:4.619pt) arc (120:60:4.619pt) arc (240:300:4.619pt);\draw[thin,-my] (2.3095pt,4.0001pt) arc (60:120:4.619pt);\draw[thin,-my] (-2.3095pt,4.0001pt) --  (-11.5475pt,4.0001pt);;\draw[thin,-my] (-4.619pt,0.0pt) --  (-9.238pt,8.0001pt);;\end{tikzpictured}
,\ or
\begin{tikzpictured}\draw[rhrhombidraw] (0.0pt,0.0pt) -- (4.619pt,8.0001pt) -- (0.0pt,16.0002pt) -- (-4.619pt,8.0001pt) -- cycle;\draw[thin,-my] (2.3095pt,4.0001pt) arc (60:120:4.619pt) arc (300:240:4.619pt) arc (240:180:4.619pt);\draw[thin,-my] (-9.238pt,8.0001pt) --  (-4.619pt,16.0002pt);;\draw[thin,-my] (-11.5475pt,12.0002pt) --  (-2.3095pt,12.0002pt);;\draw[thin,-my] (-2.3095pt,12.0002pt) arc (240:300:4.619pt);\end{tikzpictured}
, respectively.

As in the proof of Key Lemma~\ref{keylem:ktt:secure},
(\ensuremath{\ast}\ensuremath{\ast}') can be seen to imply (\ensuremath{\ast}') by constructing a chain of bad rhombi.
\end{proof}
Theorem~\ref{thm:ktt} is completely proved.

\end{document}